\documentclass[twoside]{article}
\usepackage[utf8]{inputenc}
\usepackage[french,english]{babel}
\usepackage[margin=1in]{geometry}
\usepackage[titletoc,title]{appendix}
\usepackage{amsmath,amsfonts,amssymb,mathtools}
\usepackage{graphicx,float}
\usepackage[cyr]{aeguill}

\usepackage{hyperref}
\usepackage{cite}
\usepackage{cleveref}
\usepackage{epsfig}
\usepackage{pinlabel}
\usepackage{subcaption}

\usepackage{tikz}
\usepackage{tikz-cd}
\usetikzlibrary{math}
\usepackage{pgf}
\usetikzlibrary{arrows}
\usetikzlibrary{shapes.geometric}
\usetikzlibrary{shapes.misc}

\usepackage{array}
\usepackage{setspace, caption, subcaption}

\usepackage{amsfonts,amssymb,amsmath,amsthm}
\usepackage{enumitem}
\usepackage[all]{xy}

\newcommand{\p}{\ensuremath{\mathbb{P}}}
\newcommand{\Z}{\ensuremath{\mathbb{Z}}}
\newcommand{\Q}{\ensuremath{\mathbb{Q}}}
\newcommand{\F}{\ensuremath{\mathbb{F}}}

\newcommand{\Ql}{\mathcal{Q}}

\DeclareMathOperator{\PGL}{PGL}
\DeclareMathOperator{\Aut}{Aut}
\DeclareMathOperator{\Bir}{Bir}
\newcommand{\Pic}{{\mathrm{Pic}}}
\newcommand{\NS}{{\mathrm{NS}}}
\newcommand{\rk}{{\mathrm{rk}}}
\DeclareMathOperator{\Gal}{Gal}
\DeclareMathOperator{\Sym}{Sym}
\DeclareMathOperator{\id}{id}
\DeclareMathOperator{\D}{D}
\DeclareMathOperator{\GA}{GA}
\DeclareMathOperator{\Spec}{Spec}
\DeclareMathOperator{\modulo}{mod}

\renewcommand{\k}{\mathbf{k}}
\newcommand{\kb}{\overline{\mathbf{k}}}

\def\dashmapsto{\mapstochar\dashrightarrow}

\newtheorem{definition}{Definition}[section]
\newtheorem{lemma}[definition]{Lemma}
\newtheorem{proposition}[definition]{Proposition}
\newtheorem{theorem}[definition]{Theorem}
\newtheorem{remark}[definition]{Remark}

\newtheorem{example}[definition]{Example}
\newtheorem{examples}[definition]{Examples}

\author{Aurore Boitrel}
\newcommand{\Addresses}{{
  \bigskip
  \footnotesize
  \textsc{Aurore Boitrel, Aix Marseille Univ, CNRS, I2M, SMF, CIRM, Marseille, France}\par\nopagebreak
  \textit{E-mail address:} \texttt{aurore.boitrel@universite-paris-saclay.fr}
}}
\title{Del Pezzo surfaces of degree $5$ over perfect fields}

\date{}

\usepackage{tocloft}

\usepackage{fancyhdr}
\renewcommand{\thepage}{\arabic{page}}
\begin{document}

\maketitle

\begin{abstract} In this paper we study the classification of del Pezzo surfaces $X$ of degree $5$ over any perfect field $\mathbf{k}$ in explicit geometric terms. More precisely, in each case we use the Petersen graph to illustrate the $\operatorname{Gal}(\overline{\mathbf{k}}/\mathbf{k})$-action on the $(-1)$-curves of $X$ and we describe explicitly its group of automorphisms, $\operatorname{Aut}_{\mathbf{k}}(X)$. For the cases when $X$ is not minimal, we describe how to realize it as the blow-up of $\mathbb{P}^{2}$, or of a (minimal) quadric in $\mathbb{P}^{3}$, and classify them up to $\mathbf{k}$-isomorphism. In all cases, the elements of the group $\operatorname{Aut}_{\mathbf{k}}(X)$ are described geometrically.
\end{abstract}

\tableofcontents

\definecolor{forestgreen}{rgb}{0,0.65,0}
\definecolor{cyan}{rgb}{0.871,0.494,0}

\section{Introduction}

Del Pezzo surfaces, i.e.\ smooth projective surfaces with ample anticanonical divisor class, play an important role in the classification of algebraic projective surfaces up to birational equivalence. The classical theory of del Pezzo surfaces over $\mathbb{C}$ can be found in \cite{bea78}, and over an arbitrary field $\k$ we refer the reader to \cite{man66,dem80,man86,bad01,KSC04}. Del Pezzo surfaces also play a key role in the study of algebraic subgroups of the Cremona group $\Bir_{\k}(\p^{2})$ of birational automorphisms of the projective plane over $\k$, and the group $\Aut_{\k}(X)$ of automorphisms of any rational del Pezzo surface $X$ defined over $\k$ provides examples of such subgroups. For a field $\k=\kb$ of characteristic zero, the complete description of automorphism groups of del Pezzo surfaces can be found in \cite{di09}, where the classification up to conjugacy of finite subgroups of $\Bir_{\kb}(\p^{2})$ is obtained. See also \cite{bla09} for the exposition of these results. Recent partial results about automorphism groups of del Pezzo surfaces over various perfect fields can be found in \cite{rz18,yas22,sz21}.
 
Over an algebraically closed field $\k=\kb$, a del Pezzo surface $X$ is either isomorphic to $\p^{2}_{\kb}$ or to $\p^{1}_{\kb} \times \p^{1}_{\kb}$ or to the blow-up of $\p^{2}_{\kb}$ in $1 \leq r \leq 8$ points in general position, where $K_{X}^{2}=9-r$ is called the degree of $X$ (\hspace{1sp}\cite{man86}). In particular, one can see that there is only one isomorphism class of del Pezzo surfaces of degree $5$ over $\kb$, obtained by blowing up $\p^{2}_{\kb}$ in four points in general position. In this paper, we focus on del Pezzo surfaces of degree $5$ defined over a perfect field. Over a perfect field $\k$, the base-change of a del Pezzo surface $X$ to the algebraic closure $X_{\kb}$ is still a del Pezzo surface of the same degree endowed with an action of the Galois group $\Gal(\kb/\k)$ of $\kb$ over $\k$. The classification of del Pezzo surfaces of degree $5$ over $\k$ can then be obtained by classifying Galois actions on del Pezzo surfaces of degree $5$ over $\kb$, and in particular the elements of the first Galois cohomology set $H^{1}(\Gal(\kb/\k),\Aut(X_{\kb}))$ are in bijection with the $\k$-isomorphism classes of del Pezzo surfaces of degree $5$ (see \cite[§2.6]{BS64} and \cite[Theorem 3.1.3]{sko01}).\\

Assume again that $\k$ is a perfect field and let $X$ be a smooth projective surface over $\k$. By a point of degree $d$ we mean a $\Gal(\kb/\k)$-orbit $p=\lbrace p_{1},\dots,p_{d} \rbrace \subset X(\kb)$ of cardinality $d\geq1$. A splitting field of $p=\lbrace p_{1},\dots,p_{d} \rbrace$ is a finite normal extension $F/\k$ of smallest degree such that $p_{1},\dots,p_{d} \in X(F)$.
Suppose that $\k$ has a quadratic extension $L/\k$ and let $g$ be the generator of $\Gal(L/\k) \simeq \Z/2\Z$. By $\mathcal{Q}^{L}$ we denote the $\k$-structure on $\p^{1}_{L} \times \p^{1}_{L}$ given by $(x,y)^{g}=(y^{g},x^{g})$. Our main result is the following classification result:
 
\begin{theorem}\label{thm:Theorem_Main}
Let $\k$ be a perfect field and $X$ a del Pezzo surface of degree $5$ over $\k$. Then one of the following cases holds:
\begin{enumerate}
\item $\rk \, \NS(X) = 1$ and $X$ is obtained by blowing up $\p^{2}$ in a point $p=\lbrace p_{1},\dots,p_{5} \rbrace$ of degree $5$ with splitting field $F$ to a del Pezzo surface of degree $4$, and then blowing down the strict transform of the conic passing through $p$. Moreover we are in one of the following cases:
\label{cas:Case1_Main_Theorem}
\begin{enumerate}
\item $\Gal(F/\k) \simeq \Z/5\Z$ and $\Aut_{\k}(X)=\langle \widehat{\phi} \rangle \simeq \Z/5\Z$, where $\widehat{\phi}$ is the lift of a quadratic birational map of $\p^{2}_{\kb}$ of order five,
\item $\Gal(F/\k) \simeq \D_{5}$ and $\Aut_{\k}(X)=\lbrace \id \rbrace$,
\item $\Gal(F/\k) \simeq \GA(1,5)$ and $\Aut_{\k}(X)=\lbrace \id \rbrace$,
\item $\Gal(F/\k) \simeq \mathcal{A}_{5}$ and $\Aut_{\k}(X)=\lbrace \id \rbrace$,
\item $\Gal(F/\k) \simeq \Sym_{5}$ and $\Aut_{\k}(X)=\lbrace \id \rbrace$.
\end{enumerate}
\item $\rk \, \NS(X) = 5$, $\rk \, \NS(X)^{\Aut_{\k}(X)} = 1$ and $X$ is isomorphic to the blow-up of $\p^{2}$ in the four coordinate points with $\Aut_{\k}(X) \simeq \Sym_{5}$.
\label{cas:Case2_Main_Theorem}
\item $\rk \, \NS(X) \geq 2$, $\rk \, \NS(X)^{\Aut_{\k}(X)} = 2$ and $X$ is one of the following:
\label{cas:Case3_Main_Theorem}
\begin{enumerate}
\item $X$ is isomorphic to the blow-up of $\Ql^{L}$ in the three $\k$-rational points $p_{1}=([1:0],[1:0]), p_{2}=([0:1],[0:1]), p_{3}=([1:1],[1:1]) $, for some quadratic extension $L/\k$. Its isomorphism class depends only on the $\k$-isomorphism class of $L$. Moreover, $\Aut_{\k}(X) = \langle \widehat{\alpha},\widehat{\gamma} \rangle \times \langle \widehat{\beta} \rangle \simeq \Sym_{3}\times\Z/2\Z$, where $\widehat{\alpha}, \widehat{\beta}, \widehat{\gamma}$ are the lifts of involutions $\alpha, \beta, \gamma$ of $\Ql^{L}$, 
\label{s-cas:Case(a)_Case3_Main_Theorem}
\item $X$ is isomorphic to the blow-up of $\p^{2}$ in two points $r=\lbrace r_{1},r_{2} \rbrace$, $s=\lbrace s_{1},s_{2} \rbrace$ of degree $2$ with splitting field a quadratic extension $L/\k$. Its isomorphism class depends only on the $\k$-isomorphism class of $L$. Moreover, $\Aut_{\k}(X)=\langle \widehat{\alpha},\widehat{\beta} \rangle \simeq \D_{4} $, where $\widehat{\alpha}$ is the lift of an involution $\alpha$ of $\p^{2}$ and $\widehat{\beta}$ the lift of an automorphism $\beta$ of $\p^{2}$ of order four,
\label{s-cas:Case(b)_Case3_Main_Theorem}
\item $X$ is isomorphic to the blow-up of $\p^{2}$ in a point $q=\lbrace q_{1},q_{2},q_{3} \rbrace$ of degree $3$ with splitting field $F$ such that $\Gal(F/\k) \simeq \Z/3\Z$ and in a $\k$-rational point $r \in \p^{2}(\k)$. Its isomorphism class depends only on the $\k$-isomorphism class of $F$. Moreover, $\Aut_{\k}(X)=\langle \widehat{\alpha} \rangle \times \langle \widehat{\beta} \rangle \simeq \Z/3\Z \times \Z/2\Z $, where $\widehat{\alpha}$ is the lift of an automorphism of $\p^{2}$ of order three and $\widehat{\beta}$ is the lift of a birational quadratic involution of $\p^{2}$ with base-point $q$,
\label{s-cas:Case(c)_Case3_Main_Theorem}
\item $X$ is isomorphic to the blow-up of $\p^{2}$ in a point $q=\lbrace q_{1},q_{2},q_{3} \rbrace$ of degree $3$ with splitting field $F$ such that $\Gal(F/\k) \simeq \Sym_{3}$ and in a $\k$-rational point $r \in \p^{2}(\k)$. Its isomorphism class depends only on the $\k$-isomorphism class of $F$. Moreover, $\Aut_{\k}(X)=\langle \widehat{\alpha} \rangle \simeq \Z/2\Z $, where $\widehat{\alpha}$ is the lift of a birational quadratic involution of $\p^{2}$ with base-point $q$,
\label{s-cas:Case(d)_Case3_Main_Theorem}
\item $X$ is isomorphic to the blow-up of $\p^{2}$ in a point $q=\lbrace q_{1},q_{2},q_{3},q_{4} \rbrace$ of degree $4$ with splitting field $F$ such that $\Gal(F/\k) \simeq \Z/2\Z \times \Z/2\Z$. Its isomorphism class depends only on the $\k$-isomorphism class of $F$. Moreover, $\Aut_{\k}(X)=\langle \widehat{\alpha} \rangle \times \langle \widehat{\beta} \rangle \simeq \Z/2\Z \times \Z/2\Z$, where $\widehat{\alpha}, \widehat{\beta}$ are the lifts of involutions of $\p^{2}$,  
\label{s-cas:Case(e)_Case3_Main_Theorem}
\item $X$ is isomorphic to the blow-up of $\p^{2}$ in a point $q=\lbrace q_{1},q_{2},q_{3},q_{4} \rbrace$ of degree $4$ with splitting field $F$ such that $\Gal(F/\k) \simeq \Z/4\Z$. Its isomorphism class depends only on the $\k$-isomorphism class of $F$. Moreover, $\Aut_{\k}(X)=\langle \widehat{\alpha} \rangle \simeq \Z/4\Z $, where $\widehat{\alpha}$ is the lift of an automorphism of $\p^{2}$ of order four,
\label{s-cas:Case(f)_Case3_Main_Theorem}
\item $X$ is isomorphic to the blow-up of $\p^{2}$ in a point $q=\lbrace q_{1},q_{2},q_{3},q_{4} \rbrace$ of degree $4$ with splitting field $F$ such that $\Gal(F/\k) \simeq \D_{4}$. Its isomorphism class depends only on the $\k$-isomorphism class of the residue field of $q$. Moreover, $\Aut_{\k}(X)=\langle \widehat{\alpha} \rangle \simeq \Z/2\Z $, where $\widehat{\alpha}$ is the lift of an involution of $\p^{2}$,
\label{s-cas:Case(g)_Case3_Main_Theorem}
\item $X$ is isomorphic to the blow-up of $\p^{2}$ in a point $q=\lbrace q_{1},q_{2},q_{3},q_{4} \rbrace$ of degree $4$ with splitting field $F$ such that $\Gal(F/\k) \simeq \mathcal{A}_{4}$. Its isomorphism class depends only on the $\k$-isomorphism class of $F$. Moreover, $\Aut_{\k}(X)=\lbrace \id \rbrace$,
\label{s-cas:Case(h)_Case3_Main_Theorem}
\item $X$ is isomorphic to the blow-up of $\p^{2}$ in a point $q=\lbrace q_{1},q_{2},q_{3},q_{4} \rbrace$ of degree $4$ with splitting field $F$ such that $\Gal(F/\k) \simeq \Sym_{4}$. Its isomorphism class depends only on the $\k$-isomorphism class of $F$. Moreover, $\Aut_{\k}(X)=\lbrace \id \rbrace$,
\label{s-cas:Case(i)_Case3_Main_Theorem}
\item $X$ is isomorphic to the blow-up of $\mathcal{Q}^{L}$ in a point $p=\lbrace p_{1},p_{2},p_{3} \rbrace$ of degree $3$ with splitting field $F$ such that $\Gal(F/\k) \simeq \Gal(FL/L) \simeq \Z/3\Z$, for some quadratic extension $L/\k$. Its isomorphism class depends on the $\k$-isomorphism classes of $L$ and $F$. Moreover, $\Aut_{\k}(X)=\langle \widehat{\alpha} \rangle \times \langle \widehat{\Phi_{q}} \rangle \simeq \Z/3\Z \times \Z/2\Z$, where $\widehat{\alpha}$ is the lift of an automorphism of $\p^{2}_{L}$ of order three and $\widehat{\Phi_{q}}$ is the lift of a birational quadratic involution of $\p^{2}_{L}$ with a base-point $q$ of degree $3$,
\label{s-cas:Case(j)_Case3_Main_Theorem}
\item $X$ is isomorphic to the blow-up of $\mathcal{Q}^{L}$ in a point $p=\lbrace p_{1},p_{2},p_{3} \rbrace$ of degree $3$ with splitting field $F$ such that $\Gal(F/\k) \simeq \Gal(FL/L) \simeq \Sym_{3}$, for some quadratic extension $L/\k$. Its isomorphism class depends on the $\k$-isomorphism classes of $L$ and $F$. Moreover, $\Aut_{\k}(X)=\langle \widehat{\Phi_{q}} \rangle \simeq \Z/2\Z$, where $\widehat{\Phi_{q}}$ is the lift of a birational quadratic involution of $\p^{2}_{L}$ with base-point $q$ of degree $3$,
\label{s-cas:Case(k)_Case3_Main_Theorem}
\item $X$ is isomorphic to the blow-up of $\mathcal{Q}^{L}$ in a point $p=\lbrace p_{1},p_{2},p_{3} \rbrace$ of degree $3$ with splitting field $F$ such that $\Gal(F/\k) \simeq \Sym_{3}$ and $\Gal(FL/L) \simeq \Z/3\Z$, for some quadratic extension $L/\k$. Its isomorphism class depends on the $\k$-isomorphism classes of $L$ and $F$. Moreover, $\Aut_{\k}(X)=\langle \widehat{\Phi_{q}} \rangle \simeq \Z/2\Z$, where $\widehat{\Phi_{q}}$ is the lift of a birational quadratic involution of $\p^{2}_{L}$ with base-point $q$ of degree $3$.
\label{s-cas:Case(l)_Case3_Main_Theorem}
\end{enumerate}
\item $\rk \, \NS(X) \geq 2$, $\rk \, \NS(X)^{\Aut_{\k}(X)} = 3$ and $X$ is isomorphic to the blow-up of $\p^{2}$ in two points $r=\lbrace r_{1},r_{2} \rbrace$, $s=\lbrace s_{1},s_{2} \rbrace$ of degree $2$ whose splitting fields are respectively non $\k$-isomorphic quadratic extensions $L/\k$, $L'/\k$. Its isomorphism class depends only on the $\k$-isomorphism classes of $L$ and $L'$. Moreover, $\Aut_{\k}(X)=\langle \widehat{\alpha} \rangle \times \langle \widehat{\beta} \rangle \simeq \Z/2\Z \times \Z/2\Z $, where $\widehat{\alpha}, \widehat{\beta}$ are the lifts of two involutions $\alpha, \beta$ of $\p^{2}$.
\label{cas:Case4_Main_Theorem}
\end{enumerate}
\end{theorem}

This paper is organized as follows. In Section \ref{sec:section_2} we recall some basic facts about del Pezzo surfaces as well as some of their arithmetic properties; it also gathers some technical lemmas we will need and refer to in Section~\ref{sec:section_3}. Section \ref{sec:section_3} is devoted to the proof of Theorem \ref{thm:Theorem_Main}, with which we extend the analysis of \cite{sz21} for del Pezzo surfaces of degree $ d \geq 6 $. We classify del Pezzo surfaces of degree $5$ over a perfect field $\k$ according to the possible actions of the Galois group $\Gal(\kb/\k)$ on the graph of $(-1)$-curves (see Figure \ref{Fig:Figure_options_for_rho(Gal(kbarre/k))_actions_on_Pikbarre}). We go through all the cases in subsections \ref{subsec:subsection_1}, \ref{subsec:subsection_2}, \ref{subsec:subsection_3} and \ref{subsec:subsection_4}, we give a geometric realization of del Pezzo surfaces of degree $5$, provide some examples of explicit constructions of such surfaces, and in addition we describe their automorphism groups.\\

During the redaction of this paper, \cite{zai23} appeared, where the results about automorphism groups of del Pezzo surfaces of degree $5$ over a field $\k$ have been proven independently. In this paper we also give a precise geometric description of generators of the automorphism groups of del Pezzo surfaces of degree $5$, and of the surfaces themselves.\\

\noindent \textbf{Acknowledgements}\\
For all interesting discussions about the subject, for their tremendous help and for being available and supportive throughout, I am very grateful to my PhD advisors Andrea Fanelli and Susanna Zimmermann, without whom this work would not have been possible.
I thank Paolo Cascini for hosting me at Imperial College where part of this work was made, and for interesting comments. I also thank Jihun Park for hosting me at IBS Center for Geometry and Physics.
I would like to thank Fabio Bernasconi, Jérémy Blanc and Julia Schneider for interesting comments and very helpful discussions.
I am very grateful to the two anonymous referees for their very careful and helpful reading of this paper. This work was supported by the French National Centre for Scientific Research (CNRS), the University of Angers, the ANR project FIBALGA [ANR-18-CE40-0003] and the ERC Saphidir.

\section{Conventions and preliminary results}
\label{sec:section_2}

For us, a variety is a geometrically integral projective scheme defined over a field $\k$, and a surface is a variety of dimension $2$. Throughout the article, $\k$ denotes a perfect field and $\kb$ an algebraic closure. Let $X$ be a smooth projective surface over $\k$. 
We denote by $X(\k)$ the set of $\k$-rational points of $X$. The Galois group $\Gal(\kb/\k)$ acts on $X_{\kb}:=X \times_{\Spec(\k)} \Spec(\kb)$ through the second factor. A point of degree $d$ of $X$ is a $\Gal(\kb/\k)$-orbit $p=\lbrace p_{1},\dots,p_{d} \rbrace \subset X(\kb)$ of cardinality $d\geq1$. Let $L/\k$ be an algebraic extension of $\k$ such that all the components $p_{i}$ are $L$-rational points. By the blow-up of $p$ we mean the blow-up of these $d$ points, which is a birational morphism $\pi : X' \rightarrow X$ defined over $\k$ with exceptional divisor $E=E_{1}+\cdots+E_{d}$, where the $E_{i}$ are disjoint $(-1)$-curves defined over $L$, meaning $E_{i} \simeq \p^{1}_{L}$ and $E_{i}^{2}=-1$ for $i=1,\dots,d$, and $E^{2}=-d$. We call $E$ the exceptional divisor of $p$. More generally, a birational map $f : X \dashrightarrow X'$ is defined over $\k$ if and only if the birational map $f \times \id : X_{\kb} \dashrightarrow X'_{\kb}$ is $\Gal(\kb/\k)$-equivariant. In  particular, $X \simeq X'$ if and only if there is a $\Gal(\kb/\k)$-equivariant isomorphism $X_{\kb} \rightarrow X'_{\kb}$ (see also \cite[§2.4]{BS64}). The surface $X$ is said to be rational (or $\k$-rational) if it is $\k$-birational to $\p^{2}$. If $X(\k) \neq \varnothing$, we have $\Pic(X)=\Pic(X_{\kb})^{\Gal(\kb/\k)}$ (\hspace{1sp}\cite[Lemma 6.3(iii)]{sans81}), and this holds in particular if $X$ is rational because then it has a $\k$-rational point by the Lang-Nishimura theorem. Since numerical equivalence is stable under $\Gal(\kb/\k)$-action, also algebraic equivalence is, and hence $\NS(X)=\NS(X_{\kb})^{\Gal(\kb/\k)}$. If not otherwise mentioned, any surface, curve, point, morphism or rational map will be defined over the perfect field $\k$. By a geometric component of a curve $C$ (resp. a point $p=\lbrace p_{1},\dots,p_{d} \rbrace$), we mean an irreducible component of $C_{\kb}$ (resp. one of $p_{1},\dots,p_{d}$). For $n\geq1$, any smooth projective variety $X$ over $\k$ with $X(\k) \neq \varnothing$ such that $X_{\kb} \simeq \p^{n}_{\kb}$ is in fact already isomorphic to $\p^{n}$ over $\k$, by Châtelet's theorem. This means in particular that $\p^{2}$ is the only rational del Pezzo surface of degree $9$ and that a smooth curve of genus $0$ with rational points is isomorphic to $\p^{1}$. For a surface $X$, we denote by $\Aut_{\k}(X)$ its group of $\k$-automorphisms, which is the subgroup of automorphisms of $X_{\kb}$ that commute with the action of $\Gal(\kb/\k)$ (\hspace{1sp}\cite[§2.3]{BS64}). For a $\Gal(\kb/\k)$-invariant collection $p_{1},\dots,p_{r} \in X(\kb)$ of points, we denote by $\Aut_{\k}(X,p_{1},\dots,p_{r})$, resp. $\Aut_{\k}(X,\lbrace p_{1},\dots,p_{r} \rbrace)$, the subgroups of $\Aut_{\k}(X)$ fixing each $p_{i}$, resp. preserving the set $\lbrace p_{1},\dots,p_{r} \rbrace$.

\subsection{A quick review on del Pezzo surfaces}

A del Pezzo surface is a smooth projective surface $X$ with ample anticanonical divisor class $-K_{X}$. Then $X_{\kb}$ is isomorphic to $ \p^{1}_{\kb} \times \p^{1}_{\kb} $ or to the blow-up of $\p^{2}_{\kb} $ in at most $8$ points in general position. We recall that $ K_{X}^{2} $ is called the degree of $X$ and that $ 1 \leq K_{X}^{2} \leq 9 $.\\

For the sake of completeness, we recall the following well-known result about birational maps between del Pezzo surfaces, which will be used several times throughout this article.

\begin{lemma}[{\hspace{1sp}\cite[CH. IV, Corollary 24.5.2]{man86}}]
Let $v : X \rightarrow Z $ be any birational morphism between smooth surfaces. If $X$ is a del Pezzo surface, then $Z$ is also a del Pezzo surface.
\end{lemma}

In this paper, we will be interested in del Pezzo surfaces that are minimal, in the following sense.
 
\begin{definition}
Let $X$ be a smooth surface, $B$ a point or a smooth curve and $\pi : X \rightarrow B$ a surjective morphism with connected fibres such that $-K_{X}$ is $\pi$-ample. We call $\pi : X \rightarrow B $ a Mori fibre space if $\rk \, \NS(X/B)=1$, where $\NS(X/B) := \NS(X)/\pi^{\ast}\NS(B) $ is the relative Néron-Severi group.
\label{def:definition_Mori_fibre_space}
A del Pezzo surface $X$ is said to be an $\Aut_{\k}(X)$-Mori fibre space if $\rk \, \NS(X)^{\Aut_{\k}(X)}=1$.
\end{definition}

Notice that if $B = \lbrace \ast \rbrace$ is a point, a Mori fibre space $X$ is a del Pezzo surface with $\rk \, \NS(X)=1$.
Mori fibre spaces play an important role in the birational classification of algebraic varieties. They arise as extremal contractions in the Minimal Model Program (MMP) starting for instance from rational surfaces.\\

We have the following criterion of rationality for del Pezzo surfaces of large degrees. The second part of the statement is a classical result of Enriques.

\begin{proposition}[{[\!\cite{enr97}, \cite[Theorem 3.15]{man66}, \cite{sd72}, \cite[§4]{isk96}, see also \cite[Theorem 2.1]{valv13}]}]
Let $ \k $ be a field. Let $X$ be a smooth del Pezzo surface of degree $ d \geq 5 $. If $ X(\k) \neq \varnothing $, then $ X $ is rational. This hypothesis is automatically satisfied if $ d=5 $. 
\label{Prop:Proposition_rationality_of_dP_5_over_a_field}
\end{proposition}

\subsection{Automorphisms of del Pezzo surfaces of degree 9 and 8}

Let us recall some transitivity results for the action of automorphism groups of $\p^{2}$ and of quadric surfaces on some special points. Lemma \ref{lem:Lemme_6.11_&_6.15_article_Julia} refers to the action of $\PGL_{3}(\k)$ onto points of degree $1 < d \leq 4$ in general position in $\p^{2}$, Lemma \ref{lem:Lemme1_action_trans_Aut(Q^L)_triplets_k_points} to the action of $\Aut_{\k}(\Ql^{L})$ onto triplets of $\k$-rational points in general position in $\Ql^{L}$ and Lemma \ref{lem:Lemma_transitive_action_Aut_k(Q^L)_points_of_degree_3} to the action of $\Aut_{\k}(\Ql^{L})$ onto points of degree $3$ in general position in $\Ql^{L}$.

\begin{lemma}[{\hspace{1sp}\cite[Lemmas 6.11 and 6.15]{sch19}}] Let $L/\k$ be a finite Galois extension. Let $p_{1},\dots,p_{4} \in \p^{2}(L)$, no three collinear, and $q_{1},\dots,q_{4} \in \p^{2}(L)$, no three collinear, be points such that the sets $\lbrace p_{1},\dots,p_{4} \rbrace$ and $\lbrace q_{1},\dots,q_{4} \rbrace$ are invariant under the Galois action of $\Gal(L/\k)$. Assume that for all $g \in \Gal(L/\k)$ there exists $\sigma \in \Sym_{4}$ such that $g(p_{i})=p_{\sigma(i)}$ and $g(q_{i})=q_{\sigma(i)}$ for $i=1,\dots,4$. Then, there exists $A \in \PGL_{3}(\k)$ such that $Ap_{i}=q_{i}$ for $i=1,\dots,4$.\\ Moreover, we have the following specific result for Galois-orbits of size $4$ and $2$ respectively:
\begin{enumerate}
\item If $\lbrace p_{1},\dots,p_{4} \rbrace$ is an orbit of size $4$ of the action of $\Gal(\kb/\k)$, then there exists $ A \in \PGL_{3}(\k) $ such that $Ap_{i}=[1:a_{i}:a_{i}^{2}]$ with $a_{i} \in \mathbb{A}^{1}(\kb)$ forming an orbit $\lbrace a_{1},\dots,a_{4} \rbrace \subset \mathbb{A}^{1}(\kb) $ of size $4$ under the Galois action.
\label{it:item_(a)_Lemme_6.15_Julia}
\item If $\lbrace p_{1},p_{2} \rbrace$ and $\lbrace p_{3},p_{4} \rbrace$ form two orbits of size $2$, then there exists $A \in \PGL_{3}(\k)$ such that $Ap_{i}=[1:a_{i}:0]$ for $i=1,2$ and $Ap_{i}=[1:0:a_{i}]$ for $i=3,4$ with $a_{i} \in \mathbb{A}^{1}(\kb)$ and $\lbrace a_{1},a_{2} \rbrace$ as well as $\lbrace a_{3},a_{4} \rbrace $ form an orbit in $\mathbb{A}^{1}(\kb)$.
\label{it:item_(b)_Lemme_6.15_Julia}\\
In both cases, the field of definition of $\lbrace p_{1},\dots,p_{4} \rbrace$ is $\k(a_{1},\dots,a_{4})$.
\end{enumerate}
\label{lem:Lemme_6.11_&_6.15_article_Julia}
\end{lemma}

The argument of Lemma \ref{lem:Lemme_6.11_&_6.15_article_Julia} can be applied to show the following analogue on $\p^{1}$.

\begin{lemma}[{\hspace{1sp}\cite[Remark 2.7]{sz21}}]
Let $F/\k$ be a finite extension. Let $p_{1},p_{2},p_{3},q_{1},q_{2},q_{3} \in \p^{1}(F)$ such that the sets $\lbrace p_{1},p_{2},p_{3} \rbrace$ and $\lbrace q_{1},q_{2},q_{3} \rbrace$ are $\Gal(\kb/\k)$-invariant. Suppose that for any $g \in \Gal(F/\k)$ there exists $\sigma \in \Sym_{3}$ such that $p_{i}^{g}=p_{\sigma(i)}$ and $q_{i}^{g}=q_{\sigma(i)}$ for $i=1,2,3$. Then there exists $\alpha \in \PGL_{2}(\k)$ such that $\alpha(p_{i}) = q_{i}$ for $i=1,2,3$.
\label{rem:Remarque_2.7_suivant_Lemme_2.6}
\end{lemma} 

Let $L=\k(a_{1})$ be a quadratic extension of $\k$ and let $t^{2}+\overline{a}t+\tilde{a}=(t-a_{1})(t-a_{2}) \in \k[t]$ be the minimal polynomial of $a_{1}$. Recall that the surface $\Ql^{L}$ is the $\k$-structure on $\p^{1}_{L}\times\p^{1}_{L}$ given by $ \left([u_{0}:u_{1}],[v_{0}:v_{1}]\right) \mapsto \left([v_{0}^{g}:v_{1}^{g}],[u_{0}^{g}:u_{1}^{g}]\right) $, where $g$ is the generator of the Galois group $\Gal(L/\k)$, and that $\Ql^{L}$ is isomorphic to the quadric surface given by $WZ = X^{2}+\overline{a}XY+\tilde{a}Y^{2}$ in $\p^{3}_{WXYZ}$ (\hspace{1sp}\cite[Lemma 3.3]{sz21}). The automorphism group of $\mathcal{Q}^{L}$ over $\k$ is given by $ \Aut_{\k}(\mathcal{Q}^{L}) \simeq \lbrace \left(A,A^{g}\right) \, \vert \, A \in \PGL_{2}(L) \rbrace \rtimes \langle (u,v) \overset{\tau}{\mapsto} (v,u) \rangle $ (\hspace{1sp}\cite[Lemma 3.5]{sz21}). The following lemmas give ways to move points on the surface $\mathcal{Q}^{L}$.
 
\begin{lemma}
Let $L/\k$ be a quadratic extension and let $ p_{1}, p_{2}, p_{3} \in \mathcal{Q}^{L}(\k) $ be three $\k$-rational points contained in pairwise distinct rulings of $ \mathcal{Q}^{L}_{L} $. Then there exists $ \alpha \in \Aut_{\k}(\mathcal{Q}^{L}) $ such that $ \alpha(p_{1}) = \left([1:0],[1:0]\right) $, $ \alpha(p_{2}) = \left([0:1],[0:1]\right) $ and $ \alpha(p_{3}) = \left([1:1],[1:1]\right) $.
\label{lem:Lemme1_action_trans_Aut(Q^L)_triplets_k_points}
\end{lemma}

\begin{proof}
We have $ p_{1}=\left([a:b],[a^{g}:b^{g}]\right) $, $ p_{2}=\left([c:d],[c^{g}:d^{g}]\right) $, $ p_{3}=\left([e:f],[e^{g}:f^{g}]\right) $ for some $ a, b, c, d, e, f \in L $, and $ ad-bc \neq 0 $ because $ p_{1} $ and $p_{2}$ are not on the same ruling of $ \mathcal{Q}^{L}_{L} \simeq \mathbb{P}^{1}_{L} \times \mathbb{P}^{1}_{L} $.
It follows that the map $ A : \left[u:v\right] \mapsto \left[\frac{d}{\alpha}u-\frac{c}{\alpha}v:\frac{-b}{\beta}u+\frac{a}{\beta}v\right] $, where $ \alpha, \beta \in \k^{*} $ are such that $ \left(e,f\right) = \alpha \left(a,b\right) +\beta \left(c,d\right) $ for some representative vectors in $ L^{2} $, is contained in $ \PGL_{2}(L) $.
Then$$ \left(A,A^{g}\right) = \left( \begin{pmatrix} \frac{d}{\alpha} & \frac{-c}{\alpha} \\[1mm] \frac{-b}{\beta} & \frac{a}{\beta}\end{pmatrix} , \begin{pmatrix} \frac{d^{g}}{\alpha} & \frac{-c^{g}}{\alpha} \\[1mm] \frac{-b^{g}}{\beta} & \frac{a^{g}}{\beta}\end{pmatrix} \right) \in \Aut_{\k}(\mathcal{Q}^{L}) $$and it sends $ p_{1}, p_{2} $ and $ p_{3} $ onto $ \left([1:0],[1:0]\right) $, $\left([0:1],[0:1]\right)$ and $ \left([1:1],[1:1]\right) $, respectively.
\end{proof}

\begin{lemma}[{\hspace{1sp}\cite[Lemma 3.7]{sz21}}]
Let $p=\lbrace p_{1},p_{2},p_{3} \rbrace$ and $q=\lbrace q_{1},q_{2},q_{3} \rbrace$ be points of degree $3$ in $\mathcal{Q}^{L}$, such that for any $h \in \Gal(\kb/\k)$ there exists $\sigma \in \Sym_{3}$ such that $p_{i}^{h}=p_{\sigma(i)}$ and $q_{i}^{h}=q_{\sigma(i)}$ for $i=1,2,3$. Suppose that the geometric components of $p$ (resp. of $q$) are in general position on $\mathcal{Q}^{L}_{\kb}$. Then there exists $\alpha \in \Aut_{\k}(\mathcal{Q}^{L})$ such that $\alpha(p_{i})=q_{i}$ for $i=1,2,3$.\label{lem:Lemma_transitive_action_Aut_k(Q^L)_points_of_degree_3}
\end{lemma}

\section{Del Pezzo surfaces of degree $5$: geometric realization and automorphism group}
\label{sec:section_3}

In this section, we classify smooth del Pezzo surfaces of degree $5$ over a perfect field $ \textbf{k} $ according to the Galois-action on the graph of $(-1)$-curves, and describe their automorphism groups and their generators. By Proposition \ref{Prop:Proposition_rationality_of_dP_5_over_a_field}, such del Pezzo surfaces are always $\mathbf{k}$-rational.

\subsection{Actions on the Petersen diagram}
\label{subsec:subsection_0}

Let $X$ be a del Pezzo surface of degree $5$. Then $ X_{\kb} $ is the blow-up of $ \mathbb{P}^{2}_{\kb} $ in four points $p_{1},p_{2},p_{3},p_{4}$ in general position. Denote by $E_{1},E_{2},E_{3},E_{4}$ the four exceptional divisors above the blown up points and by $D_{12},D_{13},D_{14},D_{23},D_{24},D_{34}$ the six strict transforms of the lines passing through two of the four points (see Figure \ref{fig:fig(a)_4_points_in_general_position}), that are the ten $(-1)$-curves of $X_{\kb}$. We can describe their intersection relations with the so called Petersen diagram denoted $ \Pi_{\kb} $, as represented in Figure \ref{fig:fig(b)_intersection_relations_blow-up_four_points_general}.
The Galois group $ \Gal(\kb/\k) $ acts on $\Pi_{\kb}$ by symmetries via a homomorphism of groups $$ \Gal(\kb/\k) \overset{\rho}{\longrightarrow} \Aut(\Pi_{\kb}) \simeq \Sym_{5} \subseteq \Aut(\Pic(X_{\kb})) .$$ By Petersen diagram of $X$ we mean the Petersen diagram of $X_{\kb}$ endowed with the $ \Gal(\kb/\k) $-action induced by the $\Gal(\kb/\k)$-action on $X_{\kb}$. Denote it by $ \Pi_{\kb,\rho(\Gal(\kb/\k))} $ or $\Pi_{\kb,\rho}$. The options for the action of $ \rho(\Gal(\kb/\k)) $ on the Petersen diagram of $ X_{\kb} $ are given by the nineteen conjugacy classes of subgroups in $\Sym_{5}$ and are listed in Figure \ref{Fig:Figure_options_for_rho(Gal(kbarre/k))_actions_on_Pikbarre}. Replacing each $E_{i}$ with the set $\{i,5\}$ for $i=1,\dots,4$, and each $D_{ij}$ with the set $\{k,l\}$ where $\{i,j,k,l\}=\{1,2,3,4\}$, then this other combinatorial way of picturing the Petersen graph (see Figure \ref{fig:fig(c)_combinatorial_picture_of_Petersen_graph}) makes clear the action of $\Sym_{5}$ on the ten $(-1)$-curves.

We denote by $ \Aut_{\k}(X) $ the $\k$-automorphism group of $X$, which is the subgroup of automorphisms of $X_{\kb}$ that commute with the action of $ \Gal(\kb/\k) $. The groups $ \Aut(X_{\kb}) $ and $ \Aut_{\k}(X) $ act respectively on $ \Pi_{\kb} $ and $ \Pi_{\kb,\rho} $ preserving the intersection form, which induces homomorphisms $$ \Aut(X_{\kb}) \overset{\Theta_{\kb}}{\longrightarrow} \Aut(\Pi_{\kb}) \simeq \Sym_{5} \;\;\; , \;\;\; \Aut_{\k}(X) \overset{\Theta}{\longrightarrow} \Aut(\Pi_{\kb,\rho}) \subseteq \Sym_{5}. $$ These homomorphisms are always injective. Indeed, if an element $\alpha \in \Aut(X_{\kb}) $ fixes the ten $(-1)$-curves on $\Pi_{\kb}$, then any $\kb$-birational contraction $X_{\kb} \rightarrow \p^{2}_{\kb}$ conjugates $\alpha$ to an element of $\PGL_{3}(\kb)$ that fixes four points in general position, so it must be trivial. It is well known that $\Theta_{\kb}$ is actually surjective, i.e. $\Aut(X_{\kb}) \simeq \Sym_{5} $ (see for instance \cite[§8.5.4]{dol12}, \cite[Proposition 6.3.7]{bla06} or \cite[Lemma 3.1.7]{sko01}). It follows that $\Theta$ is also surjective and the argument above also motivates the following lemma, to which we will refer extensively. 

\begin{lemma}
Let $X$ be a del Pezzo surface of degree $5$ and let $L/\k$ be a normal extension of smallest degree such that all the $(-1)$-curves of $X_{\kb}$ are defined over $L$. Then the action of $\Aut_{\k}(X)$ on the set of $(-1)$-curves of $X$ induces an isomorphism $\Psi : \Aut_{\k}(X) \overset{\sim}{\longrightarrow} \Aut(\Pi_{L,\rho}) $, where $\Pi_{L,\rho}$ denotes the Petersen diagram of $X_{L}$ equipped with the induced action of $\Gal(L/\k)$, and $ \Aut(\Pi_{L,\rho}) = \lbrace \alpha \in \Aut(\Pi_{L}) \, \vert \, \rho(\Gal(L/\k)) \circ \alpha = \alpha \circ \rho(\Gal(L/\k)) \rbrace $.
\label{lem:Lemma_faithful_action_Aut_k(X)_on_Pi_L_rho}
\end{lemma}

\begin{proof}
The group $\Aut_{\k}(X)$ acts on the set of $(-1)$-curves of $X$ preserving the intersection form, so we get a homomorphism $\Psi : \Aut_{\k}(X) \rightarrow \Aut(\Pi_{L,\rho})$. Let $\pi : X_{L} \rightarrow \p^{2}_{L}$ be the contraction of $E_{1},\dots,E_{4}$. Then an element of the kernel of $\Psi$ descends via $\pi$ to an automorphism of $\p^{2}_{L}$ that fixes four points in general position, so is trivial, and so $\Psi$ is injective. Now let $\phi \in \Aut(\Pi_{L,\rho})$. Then, since $\Theta_{L}$ defines an isomorphism, there exists a unique $f \in  \Aut(X_{L})$ such that $\Theta_{L}(f)=\phi$; and because $\Theta_{L}$ is $\Gal(L/\k)$-equivariant, we get: $\Theta_{L}(g \circ f \circ g^{-1})=\rho(g) \circ \Theta_{L}(f) \circ \rho(g)^{-1}=\rho(g) \circ \phi \circ \rho(g)^{-1}=\phi=\Theta_{L}(f)$ for all $g \in \Gal(L/\k)$. Because $\Theta_{L}$ is injective, this implies $g \circ f \circ g^{-1}=f$ for all $g \in \Gal(L/\k)$, and so $f \in \Aut_{\k}(X)$ is defined over $\k$.
\end{proof}

\begin{remark}
As an abstract group, $\Aut(\Pi_{\kb,\rho}) $ is isomorphic to the centralizer of the subgroup $\rho(\Gal(\kb/\k))$ in $\Sym_{5}$, denoted by $C_{\Sym_{5}}(\rho(\Gal(\kb/\k)))$. We will describe the geometric actions of $\Aut(\Pi_{\kb,\rho})$ in each case in sections \ref{subsec:subsection_1}, \ref{subsec:subsection_2}, \ref{subsec:subsection_3}, \ref{subsec:subsection_4}, and will realize geometrically the generators of the group $\Aut_{\k}(X)$.
\end{remark}

In the following sections we go through all the cases in Figure \ref{Fig:Figure_options_for_rho(Gal(kbarre/k))_actions_on_Pikbarre} and present the group $\Aut_{\k}(X)$. Section \ref{subsec:subsection_1} groups together the del Pezzo surfaces obtained by blowing up the quadric $\mathcal{Q}^{L}$ or equivalently $\mathbb{P}^{2}$, and for which the generators of the groups of automorphisms are given as the lifts of automorphisms of $\mathcal{Q}^{L}$ (resp. birational maps of $\mathbb{P}^{2}$) or of automorphisms of $\mathbb{P}^{2}$ (resp. birational maps of $\mathcal{Q}^{L}$). Section \ref{subsec:subsection_2} groups together the del Pezzo surfaces obtained by blowing up $\mathbb{P}^{2}$, and for which the automorphism groups are either trivial or generated by the lifts of automorphisms or of birational maps of $\mathbb{P}^{2}$. Section \ref{subsec:subsection_3} groups together the del Pezzo surfaces obtained by blowing up the quadric $\mathcal{Q}^{L}$ or equivalently a rational del Pezzo surface of degree $6$, and in this case, the generators of the automorphism groups are given as the lifts over $L$ of birational involutions or of automorphisms of $\mathbb{P}^{2}_{L}$. Finally section \ref{subsec:subsection_4} groups together the del Pezzo surfaces that are minimal, and for which the groups of automorphisms are either trivial or generated by the lift over $\kb$ of some birational quadratic map of $\mathbb{P}^{2}_{\kb}$.

\begin{figure}[]
\centering
\begin{subfigure}[b]{0.40\textwidth}
\begin{tikzpicture}[x=0.6pt,y=0.6pt,yscale=-1,xscale=1, scale=0.8, every node/.style={scale=0.8}]

\draw  [fill={rgb, 255:red, 0; green, 0; blue, 0 }  ,fill opacity=1 ] (146.6,130.2) .. controls (146.6,127.88) and (148.48,126) .. (150.8,126) .. controls (153.12,126) and (155,127.88) .. (155,130.2) .. controls (155,132.52) and (153.12,134.4) .. (150.8,134.4) .. controls (148.48,134.4) and (146.6,132.52) .. (146.6,130.2) -- cycle ;
\draw  [fill={rgb, 255:red, 0; green, 0; blue, 0 }  ,fill opacity=1 ] (265.6,129.2) .. controls (265.6,126.88) and (267.48,125) .. (269.8,125) .. controls (272.12,125) and (274,126.88) .. (274,129.2) .. controls (274,131.52) and (272.12,133.4) .. (269.8,133.4) .. controls (267.48,133.4) and (265.6,131.52) .. (265.6,129.2) -- cycle ;
\draw  [fill={rgb, 255:red, 0; green, 0; blue, 0 }  ,fill opacity=1 ] (146.6,229.2) .. controls (146.6,226.88) and (148.48,225) .. (150.8,225) .. controls (153.12,225) and (155,226.88) .. (155,229.2) .. controls (155,231.52) and (153.12,233.4) .. (150.8,233.4) .. controls (148.48,233.4) and (146.6,231.52) .. (146.6,229.2) -- cycle ;
\draw  [fill={rgb, 255:red, 0; green, 0; blue, 0 }  ,fill opacity=1 ] (265.6,230.2) .. controls (265.6,227.88) and (267.48,226) .. (269.8,226) .. controls (272.12,226) and (274,227.88) .. (274,230.2) .. controls (274,232.52) and (272.12,234.4) .. (269.8,234.4) .. controls (267.48,234.4) and (265.6,232.52) .. (265.6,230.2) -- cycle ;
\draw    (151.6,110.4) -- (150.6,248.4) ;
\draw    (270.6,110.4) -- (270.43,134.39) -- (269.76,226.39) -- (269.6,248.4) ;
\draw    (131.6,129.4) -- (271.61,130.29) -- (290.6,130.4) ;
\draw    (130.6,228.4) -- (272.58,229.29) -- (289.6,229.4) ;
\draw    (138.6,119.4) -- (268.88,230.42) -- (272.02,233.09) -- (280.6,240.4) ;
\draw    (138.1,240.9) -- (266.21,131.6) -- (271.32,127.24) -- (281.1,118.9) ;
\draw (125,236.4) node [anchor=north west][inner sep=0.75pt]    {$p_{1}$};
\draw (280,236.4) node [anchor=north west][inner sep=0.75pt]    {$p_{2}$};
\draw (279,101.4) node [anchor=north west][inner sep=0.75pt]    {$p_{3}$};
\draw (126,102.4) node [anchor=north west][inner sep=0.75pt]    {$p_{4}$};
\draw (197,230.4) node [anchor=north west][inner sep=0.75pt]    {$L_{1}{}_{2}$};
\draw (272,171.4) node [anchor=north west][inner sep=0.75pt]    {$L_{23}{}$};
\draw (198,112.4) node [anchor=north west][inner sep=0.75pt]    {$L_{3}{}_{4}$};
\draw (124,171.4) node [anchor=north west][inner sep=0.75pt]    {$L_{1}{}_{4}$};
\draw (168.52,189.83) node [anchor=north west][inner sep=0.75pt]  [rotate=-319.77]  {$L_{1}{}_{3}$};
\draw (236.2,176.88) node [anchor=north west][inner sep=0.75pt]  [rotate=-42.42]  {$L_{2}{}_{4}$};
\draw (84,170.4) node [anchor=north west][inner sep=0.75pt]    {$\mathbb{P}_{\overline{\textbf{k}}}^{2}$};
\end{tikzpicture}
\caption{The three singular conics through $p_{1},p_{2},p_{3}$ and $p_{4}$}
\label{fig:fig(a)_4_points_in_general_position}
\end{subfigure}
\begin{subfigure}[b]{0.40\textwidth}
        \tikzmath{
        real \a,\b;
        \a=1.6;
        \b=0.7;
    	}
    	\begin{tikzpicture}[every node/.style={inner sep=2ex, scale=0.6}, rotate=18]
        \foreach \i in {0,1,...,4}{
        		\path (\i*72:\a) node[draw, circle] (N-\i) {};
        		\path (\i*72:\b) node[draw, circle] (Q-\i) {};
        		\draw (N-\i) -- (Q-\i);
        }
        \draw (N-0) -- (N-1) -- (N-2) -- (N-3) -- (N-4) -- (N-0);
        \draw (Q-0) -- (Q-2) -- (Q-4) -- (Q-1) -- (Q-3) -- (Q-0);
        \path (N-0.center) node {$E_2$};
        \path (N-1.center) node[] {$D_{12}$};
        \path (N-2.center) node[] {$E_1$};
        \path (N-3.center) node[] {$D_{14}$};
        \path (N-4.center) node[] {$D_{23}$};
        \path (Q-0.center) node[] {$D_{24}$};
        \path (Q-1.center) node[] {$D_{34}$};
        \path (Q-2.center) node[] {$D_{13}$};
        \path (Q-3.center) node[] {$E_4$};
        \path (Q-4.center) node[] {$E_3$};
    	\end{tikzpicture}    	
    \caption{Graph of incidence relations of $(-1)$-curves on $X_{\overline{\textbf{k}}}$ =: $\Pi_{\overline{\textbf{k}}}$}
    \label{fig:fig(b)_intersection_relations_blow-up_four_points_general}
\end{subfigure}
\begin{subfigure}[b]{0.40\textwidth}
\captionsetup{justification=centering}
        \tikzmath{
        real \a,\b;
        \a=1.6;
        \b=0.7;
    	}
    	\begin{tikzpicture}[every node/.style={inner sep=2.2ex, scale=0.6}, rotate=18]
        \foreach \i in {0,1,...,4}{
        		\path (\i*72:\a) node[draw, circle] (N-\i) {};
        		\path (\i*72:\b) node[draw, circle] (Q-\i) {};
        		\draw (N-\i) -- (Q-\i);
        }
        \draw (N-0) -- (N-1) -- (N-2) -- (N-3) -- (N-4) -- (N-0);
        \draw (Q-0) -- (Q-2) -- (Q-4) -- (Q-1) -- (Q-3) -- (Q-0);
        \path (N-0.center) node {$\{2,5\}$};
        \path (N-1.center) node[] {$\{3,4\}$};
        \path (N-2.center) node[] {$\{1,5\}$};
        \path (N-3.center) node[] {$\{2,3\}$};
        \path (N-4.center) node[] {$\{1,4\}$};
        \path (Q-0.center) node[] {$\{1,3\}$};
        \path (Q-1.center) node[] {$\{1,2\}$};
        \path (Q-2.center) node[] {$\{2,4\}$};
        \path (Q-3.center) node[] {$\{4,5\}$};
        \path (Q-4.center) node[] {$\{3,5\}$};
    	\end{tikzpicture}    	
    \caption{Another way of picturing $\Pi_{\overline{\textbf{k}}}$}
    \label{fig:fig(c)_combinatorial_picture_of_Petersen_graph}
\end{subfigure}
\begin{subfigure}[b]{0.40\textwidth}
\begin{tikzpicture}[x=0.38pt,y=0.38pt,yscale=-1,xscale=1]
\draw    (200,41) -- (130.6,151.6) ;
\draw    (311,40) -- (241.6,150.6) ;
\draw    (420,40) -- (350.6,150.6) ;
\draw    (199.6,231.6) -- (129.6,110.6) ;
\draw    (309.6,230.6) -- (239.6,109.6) ;
\draw    (419.6,230.6) -- (349.6,109.6) ;
\draw    (133,204) -- (445.6,202.8) ;
\draw    (129.6,169.8) .. controls (282.6,150.8) and (167.6,22.8) .. (444.6,61.8) ;
\draw    (137.6,74.8) .. controls (154.6,107.8) and (164.6,121.8) .. (185.6,143.8) ;
\draw    (196.6,154.8) .. controls (318.6,219.8) and (310.6,47.8) .. (443.6,100.8) ;
\draw    (141,48) .. controls (196.6,56.8) and (211.6,64.8) .. (240.6,84.8) ;
\draw    (248.6,90.8) .. controls (274.6,104.8) and (282.6,126.8) .. (301.6,138.8) ;
\draw    (310,146) .. controls (348.6,176.8) and (408.6,183.8) .. (445.6,163.8) ;
\draw (185,234.4) node [anchor=north west][inner sep=0.75pt]  [font=\small]  {$D_{12}$};
\draw (184,18.4) node [anchor=north west][inner sep=0.75pt]  [font=\small]  {$D_{34}$};
\draw (296,233.4) node [anchor=north west][inner sep=0.75pt]  [font=\small]  {$D_{13}$};
\draw (295,17.4) node [anchor=north west][inner sep=0.75pt]  [font=\small]  {$D_{24}$};
\draw (407,233.4) node [anchor=north west][inner sep=0.75pt]  [font=\small]  {$D_{14}$};
\draw (407,17.4) node [anchor=north west][inner sep=0.75pt]  [font=\small]  {$D_{23}$};
\draw (451,195.4) node [anchor=north west][inner sep=0.75pt]  [font=\small]  {$E_{1}$};
\draw (450,52.4) node [anchor=north west][inner sep=0.75pt]  [font=\small]  {$E_{2}$};
\draw (449,95.4) node [anchor=north west][inner sep=0.75pt]  [font=\small]  {$E_{3}$};
\draw (450,154.4) node [anchor=north west][inner sep=0.75pt]  [font=\small]  {$E_{4}$};
\draw (84,118.4) node [anchor=north west][inner sep=0.75pt]  [font=\small]  {$X_{\kb}$};
\end{tikzpicture}
\caption{Conic bundle structure point of view}
\label{fig:fig(d)_conic_bundle_point_of_view}
\end{subfigure}
\caption{Representation of the ten $(-1)$-curves on $X_{\kb}$.}
\label{Fig:Figure_blow-up_model_dP5_Xkbarre}
\end{figure}

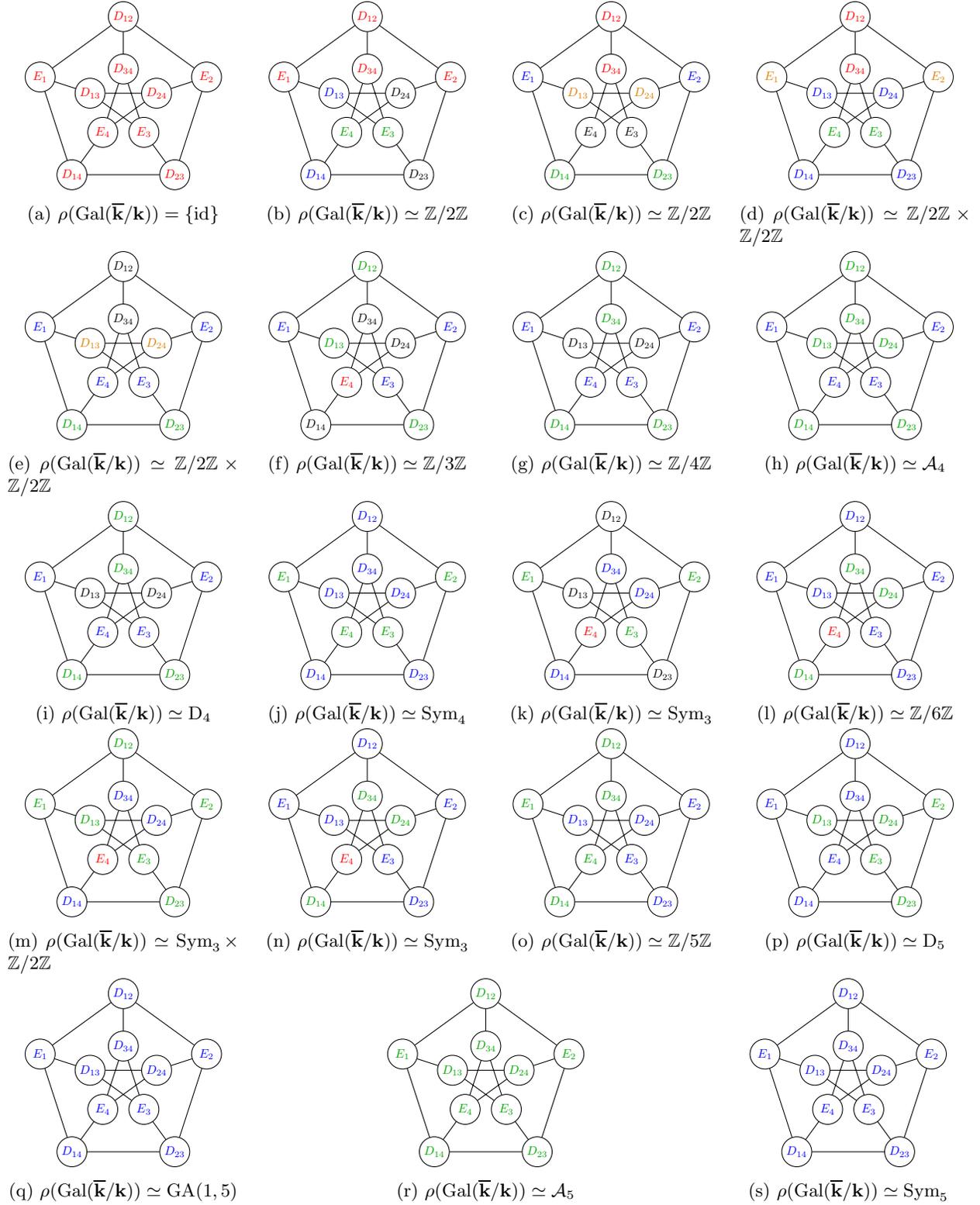
\begin{figure}
  \centering
  \begin{subfigure}[t]{0.24\textwidth}
    \centering
    \tikzmath{
      real \a,\b;
      \a=1.5;
      \b=0.6;
    }
    \begin{tikzpicture}[every node/.style={inner sep=2ex, scale=0.6}, 	rotate=18]
      \foreach \i in {0,1,...,4}{
        \path (\i*72:\a) node[draw, circle] (N-\i) {};
        \path (\i*72:\b) node[draw, circle] (Q-\i) {};
        \draw (N-\i) -- (Q-\i);
      }
      \draw (N-0) -- (N-1) -- (N-2) -- (N-3) -- (N-4) -- (N-0);
      \draw (Q-0) -- (Q-2) -- (Q-4) -- (Q-1) -- (Q-3) -- (Q-0);
      \path (N-0.center) node[red]{$E_2$};
      \path (N-1.center) node[red] {$D_{12}$};
      \path (N-2.center) node[red] {$E_1$};
      \path (N-3.center) node[red] {$D_{14}$};
      \path (N-4.center) node[red] {$D_{23}$};
      \path (Q-0.center) node[red] {$D_{24}$};
      \path (Q-1.center) node[red] {$D_{34}$};
      \path (Q-2.center) node[red] {$D_{13}$};
      \path (Q-3.center) node[red] {$E_4$};
      \path (Q-4.center) node[red] {$E_3$};
    \end{tikzpicture}    	
    \caption{$\rho(\Gal(\kb/\k))=\lbrace \id \rbrace$}
    \label{fig:figure(0)_option_Gal(kbarre/k)-action_on_Pikbarre}
  \end{subfigure}
  \hfill
  \begin{subfigure}[t]{0.24\textwidth}
    \centering
    \tikzmath{
      real \a,\b;
      \a=1.5;
      \b=0.6;
    }
    \begin{tikzpicture}[every node/.style={inner sep=2ex, scale=0.6}, 	rotate=18]
      \foreach \i in {0,1,...,4}{
        \path (\i*72:\a) node[draw, circle] (N-\i) {};
        \path (\i*72:\b) node[draw, circle] (Q-\i) {};
        \draw (N-\i) -- (Q-\i);
      }
      \draw (N-0) -- (N-1) -- (N-2) -- (N-3) -- (N-4) -- (N-0);
      \draw (Q-0) -- (Q-2) -- (Q-4) -- (Q-1) -- (Q-3) -- (Q-0);
      \path (N-0.center) node[red]{$E_2$};
      \path (N-1.center) node[red] {$D_{12}$};
      \path (N-2.center) node[red] {$E_1$};
      \path (N-3.center) node[blue] {$D_{14}$};
      \path (N-4.center) node[] {$D_{23}$};
      \path (Q-0.center) node[] {$D_{24}$};
      \path (Q-1.center) node[red] {$D_{34}$};
      \path (Q-2.center) node[blue] {$D_{13}$};
      \path (Q-3.center) node[forestgreen] {$E_4$};
      \path (Q-4.center) node[forestgreen] {$E_3$};
    \end{tikzpicture}    	
    \caption{$\rho(\Gal(\kb/\k)) \simeq \Z/2\Z$}
    \label{fig:figure(a)_option_Gal(kbarre/k)-action_on_Pikbarre}
  \end{subfigure}
  \hfill
  \begin{subfigure}[t]{0.24\textwidth}
    \centering
    \tikzmath{
      real \a,\b;
      \a=1.5;
      \b=0.6;
    }
    \begin{tikzpicture}[every node/.style={inner sep=2ex, scale=0.6}, 	rotate=18]
      \foreach \i in {0,1,...,4}{
        \path (\i*72:\a) node[draw, circle] (N-\i) {};
        \path (\i*72:\b) node[draw, circle] (Q-\i) {};
        \draw (N-\i) -- (Q-\i);
      }
      \draw (N-0) -- (N-1) -- (N-2) -- (N-3) -- (N-4) -- (N-0);
      \draw (Q-0) -- (Q-2) -- (Q-4) -- (Q-1) -- (Q-3) -- (Q-0);
      \path (N-0.center) node[blue]{$E_2$};
      \path (N-1.center) node[red] {$D_{12}$};
      \path (N-2.center) node[blue] {$E_1$};
      \path (N-3.center) node[forestgreen] {$D_{14}$};
      \path (N-4.center) node[forestgreen] {$D_{23}$};
      \path (Q-0.center) node[cyan] {$D_{24}$};
      \path (Q-1.center) node[red] {$D_{34}$};
      \path (Q-2.center) node[cyan] {$D_{13}$};
      \path (Q-3.center) node[] {$E_4$};
      \path (Q-4.center) node[] {$E_3$};
    \end{tikzpicture}    	
    \caption{$\rho(\Gal(\kb/\k)) \simeq \Z/2\Z$}
    \label{fig:figure(b)_option_Gal(kbarre/k)-action_on_Pikbarre}
  \end{subfigure}  \hfill
  \begin{subfigure}[t]{0.24\textwidth}
    \centering
    \tikzmath{
      real \a,\b;
      \a=1.5;
      \b=0.6;
    }
    \begin{tikzpicture}[every node/.style={inner sep=2ex, scale=0.6}, 	rotate=18]
      \foreach \i in {0,1,...,4}{
        \path (\i*72:\a) node[draw, circle] (N-\i) {};
        \path (\i*72:\b) node[draw, circle] (Q-\i) {};
        \draw (N-\i) -- (Q-\i);
      }
      \draw (N-0) -- (N-1) -- (N-2) -- (N-3) -- (N-4) -- (N-0);
      \draw (Q-0) -- (Q-2) -- (Q-4) -- (Q-1) -- (Q-3) -- (Q-0);
      \path (N-0.center) node[cyan]{$E_2$};
      \path (N-1.center) node[red] {$D_{12}$};
      \path (N-2.center) node[cyan] {$E_1$};
      \path (N-3.center) node[blue] {$D_{14}$};
      \path (N-4.center) node[blue] {$D_{23}$};
      \path (Q-0.center) node[blue] {$D_{24}$};
      \path (Q-1.center) node[red] {$D_{34}$};
      \path (Q-2.center) node[blue] {$D_{13}$};
      \path (Q-3.center) node[forestgreen] {$E_4$};
      \path (Q-4.center) node[forestgreen] {$E_3$};
    \end{tikzpicture}    	
    \caption{$\rho(\Gal(\kb/\k)) \simeq \Z/2\Z \times$ $\Z/2\Z$}
    \label{fig:figure(c)_option_Gal(kbarre/k)-action_on_Pikbarre}
  \end{subfigure}
  \begin{subfigure}[t]{0.24\textwidth}
    \centering
    \tikzmath{
      real \a,\b;
      \a=1.5;
      \b=0.6;
    }
    \begin{tikzpicture}[every node/.style={inner sep=2ex, scale=0.6}, 	rotate=18]
      \foreach \i in {0,1,...,4}{
        \path (\i*72:\a) node[draw, circle] (N-\i) {};
        \path (\i*72:\b) node[draw, circle] (Q-\i) {};
        \draw (N-\i) -- (Q-\i);
      }
      \draw (N-0) -- (N-1) -- (N-2) -- (N-3) -- (N-4) -- (N-0);
      \draw (Q-0) -- (Q-2) -- (Q-4) -- (Q-1) -- (Q-3) -- (Q-0);
      \path (N-0.center) node[blue]{$E_2$};
      \path (N-1.center) node[] {$D_{12}$};
      \path (N-2.center) node[blue] {$E_1$};
      \path (N-3.center) node[forestgreen] {$D_{14}$};
      \path (N-4.center) node[forestgreen] {$D_{23}$};
      \path (Q-0.center) node[cyan] {$D_{24}$};
      \path (Q-1.center) node[] {$D_{34}$};
      \path (Q-2.center) node[cyan] {$D_{13}$};
      \path (Q-3.center) node[blue] {$E_4$};
      \path (Q-4.center) node[blue] {$E_3$};
    \end{tikzpicture}    	
    \caption{$\rho(\Gal(\kb/\k)) \simeq \Z/2\Z \times$ $\Z/2\Z$}
    \label{fig:figure(d)_option_Gal(kbarre/k)-action_on_Pikbarre}
  \end{subfigure}
  \hfill
  \begin{subfigure}[t]{0.24\textwidth}
    \centering
    \tikzmath{
      real \a,\b;
      \a=1.5;
      \b=0.6;
    }
    \begin{tikzpicture}[every node/.style={inner sep=2ex, scale=0.6}, 	rotate=18]
      \foreach \i in {0,1,...,4}{
        \path (\i*72:\a) node[draw, circle] (N-\i) {};
        \path (\i*72:\b) node[draw, circle] (Q-\i) {};
        \draw (N-\i) -- (Q-\i);
      }
      \draw (N-0) -- (N-1) -- (N-2) -- (N-3) -- (N-4) -- (N-0);
      \draw (Q-0) -- (Q-2) -- (Q-4) -- (Q-1) -- (Q-3) -- (Q-0);
      \path (N-0.center) node[blue]{$E_2$};
      \path (N-1.center) node[forestgreen] {$D_{12}$};
      \path (N-2.center) node[blue] {$E_1$};
      \path (N-3.center) node[] {$D_{14}$};
      \path (N-4.center) node[forestgreen] {$D_{23}$};
      \path (Q-0.center) node[] {$D_{24}$};
      \path (Q-1.center) node[] {$D_{34}$};
      \path (Q-2.center) node[forestgreen] {$D_{13}$};
      \path (Q-3.center) node[red] {$E_4$};
      \path (Q-4.center) node[blue] {$E_3$};
    \end{tikzpicture}    	
    \caption{$\rho(\Gal(\kb/\k)) \simeq \Z/3\Z$}
    \label{fig:figure(e)_option_Gal(kbarre/k)-action_on_Pikbarre}
  \end{subfigure}
  \hfill
  \begin{subfigure}[t]{0.24\textwidth}
    \centering
    \tikzmath{
      real \a,\b;
      \a=1.5;
      \b=0.6;
    }
    \begin{tikzpicture}[every node/.style={inner sep=2ex, scale=0.6}, 	rotate=18]
      \foreach \i in {0,1,...,4}{
        \path (\i*72:\a) node[draw, circle] (N-\i) {};
        \path (\i*72:\b) node[draw, circle] (Q-\i) {};
        \draw (N-\i) -- (Q-\i);
      }
      \draw (N-0) -- (N-1) -- (N-2) -- (N-3) -- (N-4) -- (N-0);
      \draw (Q-0) -- (Q-2) -- (Q-4) -- (Q-1) -- (Q-3) -- (Q-0);
      \path (N-0.center) node[blue]{$E_2$};
      \path (N-1.center) node[forestgreen] {$D_{12}$};
      \path (N-2.center) node[blue] {$E_1$};
      \path (N-3.center) node[forestgreen] {$D_{14}$};
      \path (N-4.center) node[forestgreen] {$D_{23}$};
      \path (Q-0.center) node[] {$D_{24}$};
      \path (Q-1.center) node[forestgreen] {$D_{34}$};
      \path (Q-2.center) node[] {$D_{13}$};
      \path (Q-3.center) node[blue] {$E_4$};
      \path (Q-4.center) node[blue] {$E_3$};
    \end{tikzpicture}    	
    \caption{$\rho(\Gal(\kb/\k)) \simeq \Z/4\Z$}
    \label{fig:figure(f)_option_Gal(kbarre/k)-action_on_Pikbarre}
  \end{subfigure}
  \hfill
  \begin{subfigure}[t]{0.24\textwidth}
    \centering
    \tikzmath{
      real \a,\b;
      \a=1.5;
      \b=0.6;
    }
    \begin{tikzpicture}[every node/.style={inner sep=2ex, scale=0.6}, 	rotate=18]
      \foreach \i in {0,1,...,4}{
        \path (\i*72:\a) node[draw, circle] (N-\i) {};
        \path (\i*72:\b) node[draw, circle] (Q-\i) {};
        \draw (N-\i) -- (Q-\i);
      }
      \draw (N-0) -- (N-1) -- (N-2) -- (N-3) -- (N-4) -- (N-0);
      \draw (Q-0) -- (Q-2) -- (Q-4) -- (Q-1) -- (Q-3) -- (Q-0);
      \path (N-0.center) node[blue]{$E_2$};
      \path (N-1.center) node[forestgreen] {$D_{12}$};
      \path (N-2.center) node[blue] {$E_1$};
      \path (N-3.center) node[forestgreen] {$D_{14}$};
      \path (N-4.center) node[forestgreen] {$D_{23}$};
      \path (Q-0.center) node[forestgreen] {$D_{24}$};
      \path (Q-1.center) node[forestgreen] {$D_{34}$};
      \path (Q-2.center) node[forestgreen] {$D_{13}$};
      \path (Q-3.center) node[blue] {$E_4$};
      \path (Q-4.center) node[blue] {$E_3$};
    \end{tikzpicture}    	
    \caption{$\rho(\Gal(\kb/\k)) \simeq \mathcal{A}_{4}$}
    \label{fig:figure(g)_option_Gal(kbarre/k)-action_on_Pikbarre}
  \end{subfigure}
  \begin{subfigure}[t]{0.24\textwidth}
    \centering
    \tikzmath{
      real \a,\b;
      \a=1.5;
      \b=0.6;
    }
    \begin{tikzpicture}[every node/.style={inner sep=2ex, scale=0.6}, 	rotate=18]
      \foreach \i in {0,1,...,4}{
        \path (\i*72:\a) node[draw, circle] (N-\i) {};
        \path (\i*72:\b) node[draw, circle] (Q-\i) {};
        \draw (N-\i) -- (Q-\i);
      }
      \draw (N-0) -- (N-1) -- (N-2) -- (N-3) -- (N-4) -- (N-0);
      \draw (Q-0) -- (Q-2) -- (Q-4) -- (Q-1) -- (Q-3) -- (Q-0);
      \path (N-0.center) node[blue]{$E_2$};
      \path (N-1.center) node[forestgreen] {$D_{12}$};
      \path (N-2.center) node[blue] {$E_1$};
      \path (N-3.center) node[forestgreen] {$D_{14}$};
      \path (N-4.center) node[forestgreen] {$D_{23}$};
      \path (Q-0.center) node[] {$D_{24}$};
      \path (Q-1.center) node[forestgreen] {$D_{34}$};
      \path (Q-2.center) node[] {$D_{13}$};
      \path (Q-3.center) node[blue] {$E_4$};
      \path (Q-4.center) node[blue] {$E_3$};
    \end{tikzpicture}    	
    \caption{$\rho(\Gal(\kb/\k)) \simeq \D_{4}$}
    \label{fig:figure(h)_option_Gal(kbarre/k)-action_on_Pikbarre}
  \end{subfigure}
  \hfill
  \begin{subfigure}[t]{0.24\textwidth}
    \centering
    \tikzmath{
      real \a,\b;
      \a=1.5;
      \b=0.6;
    }
    \begin{tikzpicture}[every node/.style={inner sep=2ex, scale=0.6}, 	rotate=18]
      \foreach \i in {0,1,...,4}{
        \path (\i*72:\a) node[draw, circle] (N-\i) {};
        \path (\i*72:\b) node[draw, circle] (Q-\i) {};
        \draw (N-\i) -- (Q-\i);
      }
      \draw (N-0) -- (N-1) -- (N-2) -- (N-3) -- (N-4) -- (N-0);
      \draw (Q-0) -- (Q-2) -- (Q-4) -- (Q-1) -- (Q-3) -- (Q-0);
      \path (N-0.center) node[forestgreen]{$E_2$};
      \path (N-1.center) node[blue] {$D_{12}$};
      \path (N-2.center) node[forestgreen] {$E_1$};
      \path (N-3.center) node[blue] {$D_{14}$};
      \path (N-4.center) node[blue] {$D_{23}$};
      \path (Q-0.center) node[blue] {$D_{24}$};
      \path (Q-1.center) node[blue] {$D_{34}$};
      \path (Q-2.center) node[blue] {$D_{13}$};
      \path (Q-3.center) node[forestgreen] {$E_4$};
      \path (Q-4.center) node[forestgreen] {$E_3$};
    \end{tikzpicture}    	
    \caption{$\rho(\Gal(\kb/\k)) \simeq \Sym_{4}$}
    \label{fig:figure(i)_option_Gal(kbarre/k)-action_on_Pikbarre} 
  \end{subfigure}
  \hfill
  \begin{subfigure}[t]{0.24\textwidth}
    \centering
    \tikzmath{
      real \a,\b;
      \a=1.5;
      \b=0.6;
    }
    \begin{tikzpicture}[every node/.style={inner sep=2ex, scale=0.6}, 	rotate=18]
      \foreach \i in {0,1,...,4}{
        \path (\i*72:\a) node[draw, circle] (N-\i) {};
        \path (\i*72:\b) node[draw, circle] (Q-\i) {};
        \draw (N-\i) -- (Q-\i);
      }
      \draw (N-0) -- (N-1) -- (N-2) -- (N-3) -- (N-4) -- (N-0);
      \draw (Q-0) -- (Q-2) -- (Q-4) -- (Q-1) -- (Q-3) -- (Q-0);
      \path (N-0.center) node[forestgreen]{$E_2$};
      \path (N-1.center) node[] {$D_{12}$};
      \path (N-2.center) node[forestgreen] {$E_1$};
      \path (N-3.center) node[blue] {$D_{14}$};
      \path (N-4.center) node[] {$D_{23}$};
      \path (Q-0.center) node[blue] {$D_{24}$};
      \path (Q-1.center) node[blue] {$D_{34}$};
      \path (Q-2.center) node[] {$D_{13}$};
      \path (Q-3.center) node[red] {$E_4$};
      \path (Q-4.center) node[forestgreen] {$E_3$};
    \end{tikzpicture}    	
    \caption{$\rho(\Gal(\kb/\k)) \simeq \Sym_{3}$}
    \label{fig:figure(j)_option_Gal(kbarre/k)-action_on_Pikbarre}
  \end{subfigure}
  \hfill
  \begin{subfigure}[t]{0.24\textwidth}
    \centering
    \tikzmath{
      real \a,\b;
      \a=1.5;
      \b=0.6;
    }
    \begin{tikzpicture}[every node/.style={inner sep=2ex, scale=0.6}, 	rotate=18]
      \foreach \i in {0,1,...,4}{
        \path (\i*72:\a) node[draw, circle] (N-\i) {};
        \path (\i*72:\b) node[draw, circle] (Q-\i) {};
        \draw (N-\i) -- (Q-\i);
      }
      \draw (N-0) -- (N-1) -- (N-2) -- (N-3) -- (N-4) -- (N-0);
      \draw (Q-0) -- (Q-2) -- (Q-4) -- (Q-1) -- (Q-3) -- (Q-0);
      \path (N-0.center) node[blue]{$E_2$};
      \path (N-1.center) node[blue] {$D_{12}$};
      \path (N-2.center) node[blue] {$E_1$};
      \path (N-3.center) node[forestgreen] {$D_{14}$};
      \path (N-4.center) node[blue] {$D_{23}$};
      \path (Q-0.center) node[forestgreen] {$D_{24}$};
      \path (Q-1.center) node[forestgreen] {$D_{34}$};
      \path (Q-2.center) node[blue] {$D_{13}$};
      \path (Q-3.center) node[red] {$E_4$};
      \path (Q-4.center) node[blue] {$E_3$};
    \end{tikzpicture}    	
    \caption{$\rho(\Gal(\kb/\k)) \simeq \Z/6\Z$}
    \label{fig:figure(k)_option_Gal(kbarre/k)-action_on_Pikbarre}
  \end{subfigure}
  \begin{subfigure}[t]{0.24\textwidth}
    \centering
    \tikzmath{
      real \a,\b;
      \a=1.5;
      \b=0.6;
    }
    \begin{tikzpicture}[every node/.style={inner sep=2ex, scale=0.6}, 	rotate=18]
      \foreach \i in {0,1,...,4}{
        \path (\i*72:\a) node[draw, circle] (N-\i) {};
        \path (\i*72:\b) node[draw, circle] (Q-\i) {};
        \draw (N-\i) -- (Q-\i);
      }
      \draw (N-0) -- (N-1) -- (N-2) -- (N-3) -- (N-4) -- (N-0);
      \draw (Q-0) -- (Q-2) -- (Q-4) -- (Q-1) -- (Q-3) -- (Q-0);
      \path (N-0.center) node[forestgreen]{$E_2$};
      \path (N-1.center) node[forestgreen] {$D_{12}$};
      \path (N-2.center) node[forestgreen] {$E_1$};
      \path (N-3.center) node[blue] {$D_{14}$};
      \path (N-4.center) node[forestgreen] {$D_{23}$};
      \path (Q-0.center) node[blue] {$D_{24}$};
      \path (Q-1.center) node[blue] {$D_{34}$};
      \path (Q-2.center) node[forestgreen] {$D_{13}$};
      \path (Q-3.center) node[red] {$E_4$};
      \path (Q-4.center) node[forestgreen] {$E_3$};
    \end{tikzpicture}    	
    \caption{$\rho(\Gal(\kb/\k)) \simeq \Sym_{3}\times$ $\Z/2\Z$}
    \label{fig:figure(l)_option_Gal(kbarre/k)-action_on_Pikbarre}
  \end{subfigure}
  \hfill
  \begin{subfigure}[t]{0.24\textwidth}
    \centering
    \tikzmath{
      real \a,\b;
      \a=1.5;
      \b=0.6;
    }
    \begin{tikzpicture}[every node/.style={inner sep=2ex, scale=0.6}, 	rotate=18]
      \foreach \i in {0,1,...,4}{
        \path (\i*72:\a) node[draw, circle] (N-\i) {};
        \path (\i*72:\b) node[draw, circle] (Q-\i) {};
        \draw (N-\i) -- (Q-\i);
      }
      \draw (N-0) -- (N-1) -- (N-2) -- (N-3) -- (N-4) -- (N-0);
      \draw (Q-0) -- (Q-2) -- (Q-4) -- (Q-1) -- (Q-3) -- (Q-0);
      \path (N-0.center) node[blue]{$E_2$};
      \path (N-1.center) node[blue] {$D_{12}$};
      \path (N-2.center) node[blue] {$E_1$};
      \path (N-3.center) node[forestgreen] {$D_{14}$};
      \path (N-4.center) node[blue] {$D_{23}$};
      \path (Q-0.center) node[forestgreen] {$D_{24}$};
      \path (Q-1.center) node[forestgreen] {$D_{34}$};
      \path (Q-2.center) node[blue] {$D_{13}$};
      \path (Q-3.center) node[red] {$E_4$};
      \path (Q-4.center) node[blue] {$E_3$};
    \end{tikzpicture}    	
    \caption{$\rho(\Gal(\kb/\k)) \simeq \Sym_{3}$}
    \label{fig:figure(m)_option_Gal(kbarre/k)-action_on_Pikbarre}
  \end{subfigure}
  \hfill
  \begin{subfigure}[t]{0.24\textwidth}
    \centering
    \tikzmath{
      real \a,\b;
      \a=1.5;
      \b=0.6;
    }
    \begin{tikzpicture}[every node/.style={inner sep=2ex, scale=0.6}, 	rotate=18]
      \foreach \i in {0,1,...,4}{
        \path (\i*72:\a) node[draw, circle] (N-\i) {};
        \path (\i*72:\b) node[draw, circle] (Q-\i) {};
        \draw (N-\i) -- (Q-\i);
      }
      \draw (N-0) -- (N-1) -- (N-2) -- (N-3) -- (N-4) -- (N-0);
      \draw (Q-0) -- (Q-2) -- (Q-4) -- (Q-1) -- (Q-3) -- (Q-0);
      \path (N-0.center) node[blue]{$E_2$};
      \path (N-1.center) node[forestgreen] {$D_{12}$};
      \path (N-2.center) node[forestgreen] {$E_1$};
      \path (N-3.center) node[forestgreen] {$D_{14}$};
      \path (N-4.center) node[blue] {$D_{23}$};
      \path (Q-0.center) node[blue] {$D_{24}$};
      \path (Q-1.center) node[forestgreen] {$D_{34}$};
      \path (Q-2.center) node[blue] {$D_{13}$};
      \path (Q-3.center) node[forestgreen] {$E_4$};
      \path (Q-4.center) node[blue] {$E_3$};
    \end{tikzpicture}    	
    \caption{$\rho(\Gal(\kb/\k)) \simeq \Z/5\Z$}
    \label{fig:figure(n)_option_Gal(kbarre/k)-action_on_Pikbarre}
  \end{subfigure}
  \hfill
  \begin{subfigure}[t]{0.24\textwidth}
    \centering
    \tikzmath{
      real \a,\b;
      \a=1.5;
      \b=0.6;
    }
    \begin{tikzpicture}[every node/.style={inner sep=2ex, scale=0.6}, 	rotate=18]
      \foreach \i in {0,1,...,4}{
        \path (\i*72:\a) node[draw, circle] (N-\i) {};
        \path (\i*72:\b) node[draw, circle] (Q-\i) {};
        \draw (N-\i) -- (Q-\i);
      }
      \draw (N-0) -- (N-1) -- (N-2) -- (N-3) -- (N-4) -- (N-0);
      \draw (Q-0) -- (Q-2) -- (Q-4) -- (Q-1) -- (Q-3) -- (Q-0);
      \path (N-0.center) node[forestgreen]{$E_2$};
      \path (N-1.center) node[blue] {$D_{12}$};
      \path (N-2.center) node[blue] {$E_1$};
      \path (N-3.center) node[blue] {$D_{14}$};
      \path (N-4.center) node[forestgreen] {$D_{23}$};
      \path (Q-0.center) node[forestgreen] {$D_{24}$};
      \path (Q-1.center) node[blue] {$D_{34}$};
      \path (Q-2.center) node[forestgreen] {$D_{13}$};
      \path (Q-3.center) node[blue] {$E_4$};
      \path (Q-4.center) node[forestgreen] {$E_3$};
    \end{tikzpicture}    	
    \caption{$\rho(\Gal(\kb/\k)) \simeq \D_{5}$}
    \label{fig:figure(o)_option_Gal(kbarre/k)-action_on_Pikbarre}
  \end{subfigure}
  \hfill
  \begin{subfigure}[t]{0.24\textwidth}
    \centering
    \tikzmath{
      real \a,\b;
      \a=1.5;
      \b=0.6;
    }
    \begin{tikzpicture}[every node/.style={inner sep=2ex, scale=0.6}, 	rotate=18]
      \foreach \i in {0,1,...,4}{
        \path (\i*72:\a) node[draw, circle] (N-\i) {};
        \path (\i*72:\b) node[draw, circle] (Q-\i) {};
        \draw (N-\i) -- (Q-\i);
      }
      \draw (N-0) -- (N-1) -- (N-2) -- (N-3) -- (N-4) -- (N-0);
      \draw (Q-0) -- (Q-2) -- (Q-4) -- (Q-1) -- (Q-3) -- (Q-0);
      \path (N-0.center) node[blue]{$E_2$};
      \path (N-1.center) node[blue] {$D_{12}$};
      \path (N-2.center) node[blue] {$E_1$};
      \path (N-3.center) node[blue] {$D_{14}$};
      \path (N-4.center) node[blue] {$D_{23}$};
      \path (Q-0.center) node[blue] {$D_{24}$};
      \path (Q-1.center) node[blue] {$D_{34}$};
      \path (Q-2.center) node[blue] {$D_{13}$};
      \path (Q-3.center) node[blue] {$E_4$};
      \path (Q-4.center) node[blue] {$E_3$};
    \end{tikzpicture}    	
    \caption{$\rho(\Gal(\kb/\k)) \simeq \GA(1,5)$}
    \label{fig:figure(p)_option_Gal(kbarre/k)-action_on_Pikbarre}
  \end{subfigure}
  \hfill
  \begin{subfigure}[t]{0.24\textwidth}
    \centering
    \tikzmath{
      real \a,\b;
      \a=1.5;
      \b=0.6;
    }
    \begin{tikzpicture}[every node/.style={inner sep=2ex, scale=0.6}, 	rotate=18]
      \foreach \i in {0,1,...,4}{
        \path (\i*72:\a) node[draw, circle] (N-\i) {};
        \path (\i*72:\b) node[draw, circle] (Q-\i) {};
        \draw (N-\i) -- (Q-\i);
      }
      \draw (N-0) -- (N-1) -- (N-2) -- (N-3) -- (N-4) -- (N-0);
      \draw (Q-0) -- (Q-2) -- (Q-4) -- (Q-1) -- (Q-3) -- (Q-0);
      \path (N-0.center) node[forestgreen]{$E_2$};
      \path (N-1.center) node[forestgreen] {$D_{12}$};
      \path (N-2.center) node[forestgreen] {$E_1$};
      \path (N-3.center) node[forestgreen] {$D_{14}$};
      \path (N-4.center) node[forestgreen] {$D_{23}$};
      \path (Q-0.center) node[forestgreen] {$D_{24}$};
      \path (Q-1.center) node[forestgreen] {$D_{34}$};
      \path (Q-2.center) node[forestgreen] {$D_{13}$};
      \path (Q-3.center) node[forestgreen] {$E_4$};
      \path (Q-4.center) node[forestgreen] {$E_3$};
    \end{tikzpicture}    	
    \caption{$\rho(\Gal(\kb/\k)) \simeq \mathcal{A}_{5}$}
    \label{fig:figure(q)_option_Gal(kbarre/k)-action_on_Pikbarre}
  \end{subfigure}
  \hfill
  \begin{subfigure}[t]{0.24\textwidth}
    \centering
    \tikzmath{
      real \a,\b;
      \a=1.5;
      \b=0.6;
    }
    \begin{tikzpicture}[every node/.style={inner sep=2ex, scale=0.6}, 	rotate=18]
      \foreach \i in {0,1,...,4}{
        \path (\i*72:\a) node[draw, circle] (N-\i) {};
        \path (\i*72:\b) node[draw, circle] (Q-\i) {};
        \draw (N-\i) -- (Q-\i);
      }
      \draw (N-0) -- (N-1) -- (N-2) -- (N-3) -- (N-4) -- (N-0);
      \draw (Q-0) -- (Q-2) -- (Q-4) -- (Q-1) -- (Q-3) -- (Q-0);
      \path (N-0.center) node[blue]{$E_2$};
      \path (N-1.center) node[blue] {$D_{12}$};
      \path (N-2.center) node[blue] {$E_1$};
      \path (N-3.center) node[blue] {$D_{14}$};
      \path (N-4.center) node[blue] {$D_{23}$};
      \path (Q-0.center) node[blue] {$D_{24}$};
      \path (Q-1.center) node[blue] {$D_{34}$};
      \path (Q-2.center) node[blue] {$D_{13}$};
      \path (Q-3.center) node[blue] {$E_4$};
      \path (Q-4.center) node[blue] {$E_3$};
    \end{tikzpicture}    	
    \caption{$\rho(\Gal(\kb/\k)) \simeq \Sym_{5}$}
    \label{fig:figure(r)_option_Gal(kbarre/k)-action_on_Pikbarre_S5}
  \end{subfigure}
  \hfill
  
  \caption[]{The $\Gal(\kb/\k)$-orbits of $(-1)$-curves on the Petersen diagram of a del Pezzo surface of degree $5$; where the curves in one $\Gal(\kb/\k)$-orbit are indicated by the same color and where the ones that are defined over $\k$ are in particular represented in red.}
  \label{Fig:Figure_options_for_rho(Gal(kbarre/k))_actions_on_Pikbarre}
\end{figure}

\subsection{Del Pezzo surfaces in Figures (\ref{fig:figure(a)_option_Gal(kbarre/k)-action_on_Pikbarre}), (\ref{fig:figure(b)_option_Gal(kbarre/k)-action_on_Pikbarre}), (\ref{fig:figure(c)_option_Gal(kbarre/k)-action_on_Pikbarre})}
\label{subsec:subsection_1}
In this section we treat the cases where the Galois action on the Petersen diagram is as in (b), (c), (d) in Figure \ref{Fig:Figure_options_for_rho(Gal(kbarre/k))_actions_on_Pikbarre}. The corresponding del Pezzo surfaces of degree $5$ are obtained by blowing up the quadric $\Ql^{L}$ or the projective plane $\p^{2}$.

\begin{proposition}
Let $X$ be a del Pezzo surface of degree $5$ such that $ \rho(\Gal(\overline{\k}/\k)) = \langle(34)\rangle \simeq \mathbb{Z}/2\mathbb{Z} $ in $\text{Sym}_{5}$ as indicated in Figure \ref{fig:figure(a)_option_Gal(kbarre/k)-action_on_Pikbarre}. Then the following holds.\begin{enumerate}
\item There exist a quadratic extension $ L/\k $ and a birational morphism $ \eta : X \longrightarrow \mathcal{Q}^{L} $ contracting three disjoint $(-1)$-curves defined over $\k$ onto $ p_{1} = \left([1:0],[1:0]\right), p_{2} = \left([0:1],[0:1]\right) $ and $ p_{3} = \left([1:1],[1:1]\right) $.\label{it:item_(1)_Proposition1}
\item Any two such surfaces are isomorphic if and only if the respective quadratic extensions are $\k$-isomorphic.\label{it:item_(2)_Proposition1}
\item $ \Aut_{\k}(X) = \langle \widehat{\alpha},\widehat{\gamma} \rangle \times \langle \widehat{\beta} \rangle \simeq \Sym_{3} \times \mathbb{Z}/2\mathbb{Z} $, where $\eta\widehat{\alpha}\eta^{-1}, \eta\widehat{\beta}\eta^{-1}, \eta\widehat{\gamma}\eta^{-1}$ are involutions of $\mathcal{Q}^{L}$.\label{it:item_(3)_Proposition1}
\item $ \rk \, \NS(X)^{\Aut_{\k}(X)} = 2 $, so in particular $ X \longrightarrow \ast $ is not an $ \Aut_{\k}(X) $-Mori fibre space, and the $\Aut_{\k}(X)$-minimal models of $X$ are the quadric $\mathcal{Q}^{L}$ and the rational del Pezzo surface $Y$ of degree $6$ obtained by blowing up $\F_{0} \simeq \p^{1} \times \p^{1}$ in a point of degree $2$ whose splitting field is $L$.
\label{it:item_(4)_Proposition1}
\end{enumerate}
\label{Prop:Proposition1_Z/2Z_transposition}
\end{proposition}

\begin{proof}
(\ref{it:item_(1)_Proposition1}) The incidence diagram $ \Pi_{\overline{\k},\rho} $ presents exactly four $(-1)$-curves defined over $\k$ of $X$, three of which are disjoint namely $ E_{1}, E_{2} $ and $ D_{34} $. Their contraction yields a birational morphism $ \eta : X \rightarrow Z $ onto a rational del Pezzo surface of degree $8$ with three $\k$-rational points, $ p_{1},p_{2},p_{3} $, and the map $ \eta $ conjugates the $ \Gal(\kb/\k) $-action on $ X_{\kb} $ to an action that exchanges the fibrations of $Z_{\kb}$. Thus, $ Z \simeq \mathcal{Q}^{L} $ for some quadratic extension $ L/\k $ by \cite[Lemma 3.2(1)]{sz21}, and by Lemma \ref{lem:Lemme1_action_trans_Aut(Q^L)_triplets_k_points} we can assume that $ p_{1} = \left([1:0],[1:0]\right), p_{2} = \left([0:1],[0:1]\right) $ and $ p_{3} = \left([1:1],[1:1]\right) $. For any of the $\k$-points $ p_{i} \in \mathcal{Q}^{L}(\k) $, $ i=1,2,3 $, there is a birational map $ \psi : \mathcal{Q}^{L} \dashrightarrow \mathbb{P}^{2} $ that is the composition of the blow-up of $ p_{i} $ with the contraction of a curve onto a point of degree $2$ in $ \mathbb{P}^{2} $ whose splitting field is $L$. The surface $X$ is therefore isomorphic to the blow-up of $ \mathbb{P}^{2} $ in two $\k$-rational points $ r_{1}, r_{2} $ and one point of degree two $ r=\lbrace r_{3},r_{4} \rbrace $ whose splitting field is $L$ (see Figure \ref{Fig:Figure_blow-up_model_Z/2Z_1_case}).
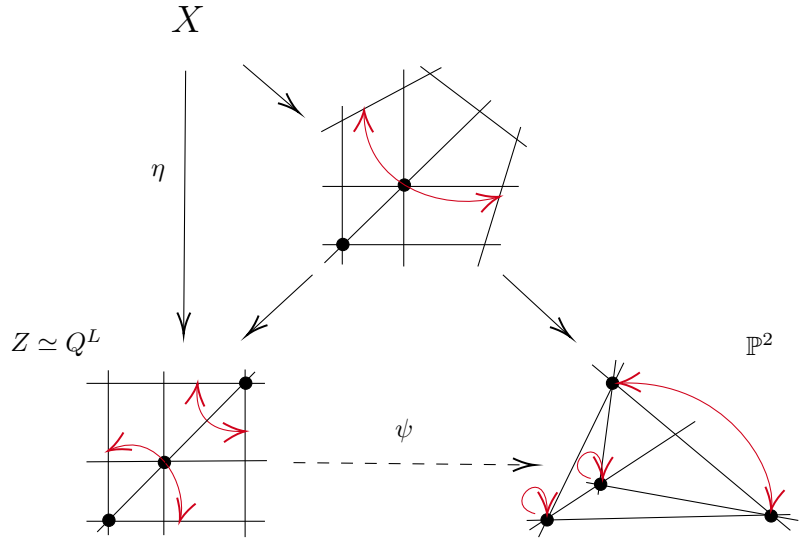
\begin{figure}[h]
\centering
\begin{tikzpicture}[x=0.6pt,y=0.6pt,yscale=-1,xscale=1, scale=0.8, every node/.style={scale=0.8}]
\draw    (61.6,213.4) -- (61.6,296.6) ;
\draw    (89.6,214) -- (89.6,296.6) ;
\draw    (130.6,213) -- (130.6,297.6) ;
\draw    (179.6,80) -- (179.6,160) ;
\draw    (210.6,61.6) -- (210.6,161) ;
\draw    (140.6,289.6) -- (50.6,289.6) ;
\draw    (259.6,150) -- (169.6,150) ;
\draw    (141.6,259) -- (50.6,259.6) ;
\draw    (140.6,220) -- (50.6,220) ;
\draw    (268.6,120.6) -- (169.6,120.6) ;
\draw    (169,93) -- (229.6,60.6) ;
\draw    (248,161) -- (269.6,90.6) ;
\draw    (219,60) -- (272.6,100.6) ;
\draw    (164.6,165.2) -- (133.07,194.24) ;
\draw [shift={(131.6,195.6)}, rotate = 317.35] [color={rgb, 255:red, 0; green, 0; blue, 0 }  ][line width=0.75]    (10.93,-3.29) .. controls (6.95,-1.4) and (3.31,-0.3) .. (0,0) .. controls (3.31,0.3) and (6.95,1.4) .. (10.93,3.29)   ;
\draw    (260.6,165.2) -- (293.13,195.24) ;
\draw [shift={(294.6,196.6)}, rotate = 222.72] [color={rgb, 255:red, 0; green, 0; blue, 0 }  ][line width=0.75]    (10.93,-3.29) .. controls (6.95,-1.4) and (3.31,-0.3) .. (0,0) .. controls (3.31,0.3) and (6.95,1.4) .. (10.93,3.29)   ;
\draw    (129.6,59.2) -- (156.08,81.9) ;
\draw [shift={(157.6,83.2)}, rotate = 220.6] [color={rgb, 255:red, 0; green, 0; blue, 0 }  ][line width=0.75]    (10.93,-3.29) .. controls (6.95,-1.4) and (3.31,-0.3) .. (0,0) .. controls (3.31,0.3) and (6.95,1.4) .. (10.93,3.29)   ;
\draw    (321.6,209.6) -- (278.6,297.6) ;
\draw    (305.6,210.6) -- (405.6,294.2) ;
\draw    (272.6,289.2) -- (405.6,286.2) ;
\draw    (317.6,208.2) -- (308.6,276.2) ;
\draw    (302.6,270) -- (409.6,289.2) ;
\draw    (100.6,62.2) -- (99.62,193.2) ;
\draw [shift={(99.6,195.2)}, rotate = 270.43] [color={rgb, 255:red, 0; green, 0; blue, 0 }  ][line width=0.75]    (10.93,-3.29) .. controls (6.95,-1.4) and (3.31,-0.3) .. (0,0) .. controls (3.31,0.3) and (6.95,1.4) .. (10.93,3.29)   ;
\draw  [fill={rgb, 255:red, 0; green, 0; blue, 0 }  ,fill opacity=1 ] (58.8,289.1) .. controls (58.8,287.39) and (60.19,286) .. (61.9,286) .. controls (63.61,286) and (65,287.39) .. (65,289.1) .. controls (65,290.81) and (63.61,292.2) .. (61.9,292.2) .. controls (60.19,292.2) and (58.8,290.81) .. (58.8,289.1) -- cycle ;
\draw  [fill={rgb, 255:red, 0; green, 0; blue, 0 }  ,fill opacity=1 ] (86.8,259.9) .. controls (86.8,258.19) and (88.19,256.8) .. (89.9,256.8) .. controls (91.61,256.8) and (93,258.19) .. (93,259.9) .. controls (93,261.61) and (91.61,263) .. (89.9,263) .. controls (88.19,263) and (86.8,261.61) .. (86.8,259.9) -- cycle ;
\draw  [fill={rgb, 255:red, 0; green, 0; blue, 0 }  ,fill opacity=1 ] (127.8,219.9) .. controls (127.8,218.19) and (129.19,216.8) .. (130.9,216.8) .. controls (132.61,216.8) and (134,218.19) .. (134,219.9) .. controls (134,221.61) and (132.61,223) .. (130.9,223) .. controls (129.19,223) and (127.8,221.61) .. (127.8,219.9) -- cycle ;
\draw  [fill={rgb, 255:red, 0; green, 0; blue, 0 }  ,fill opacity=1 ] (176.8,149.9) .. controls (176.8,148.19) and (178.19,146.8) .. (179.9,146.8) .. controls (181.61,146.8) and (183,148.19) .. (183,149.9) .. controls (183,151.61) and (181.61,153) .. (179.9,153) .. controls (178.19,153) and (176.8,151.61) .. (176.8,149.9) -- cycle ;
\draw  [fill={rgb, 255:red, 0; green, 0; blue, 0 }  ,fill opacity=1 ] (207.8,119.9) .. controls (207.8,118.19) and (209.19,116.8) .. (210.9,116.8) .. controls (212.61,116.8) and (214,118.19) .. (214,119.9) .. controls (214,121.61) and (212.61,123) .. (210.9,123) .. controls (209.19,123) and (207.8,121.61) .. (207.8,119.9) -- cycle ;
\draw  [fill={rgb, 255:red, 0; green, 0; blue, 0 }  ,fill opacity=1 ] (279.8,288.9) .. controls (279.8,287.19) and (281.19,285.8) .. (282.9,285.8) .. controls (284.61,285.8) and (286,287.19) .. (286,288.9) .. controls (286,290.61) and (284.61,292) .. (282.9,292) .. controls (281.19,292) and (279.8,290.61) .. (279.8,288.9) -- cycle ;
\draw  [fill={rgb, 255:red, 0; green, 0; blue, 0 }  ,fill opacity=1 ] (306.8,270.9) .. controls (306.8,269.19) and (308.19,267.8) .. (309.9,267.8) .. controls (311.61,267.8) and (313,269.19) .. (313,270.9) .. controls (313,272.61) and (311.61,274) .. (309.9,274) .. controls (308.19,274) and (306.8,272.61) .. (306.8,270.9) -- cycle ;
\draw  [fill={rgb, 255:red, 0; green, 0; blue, 0 }  ,fill opacity=1 ] (312.8,219.9) .. controls (312.8,218.19) and (314.19,216.8) .. (315.9,216.8) .. controls (317.61,216.8) and (319,218.19) .. (319,219.9) .. controls (319,221.61) and (317.61,223) .. (315.9,223) .. controls (314.19,223) and (312.8,221.61) .. (312.8,219.9) -- cycle ;
\draw  [fill={rgb, 255:red, 0; green, 0; blue, 0 }  ,fill opacity=1 ] (392.8,286.9) .. controls (392.8,285.19) and (394.19,283.8) .. (395.9,283.8) .. controls (397.61,283.8) and (399,285.19) .. (399,286.9) .. controls (399,288.61) and (397.61,290) .. (395.9,290) .. controls (394.19,290) and (392.8,288.61) .. (392.8,286.9) -- cycle ;
\draw  [dash pattern={on 4.5pt off 4.5pt}]  (155,260) -- (274.6,261.18) ;
\draw [shift={(276.6,261.2)}, rotate = 180.57] [color={rgb, 255:red, 0; green, 0; blue, 0 }  ][line width=0.75]    (10.93,-3.29) .. controls (6.95,-1.4) and (3.31,-0.3) .. (0,0) .. controls (3.31,0.3) and (6.95,1.4) .. (10.93,3.29)   ;
\draw    (274,294) -- (357.6,239.2) ;
\draw    (170.6,159.2) -- (254.6,77.2) ;
\draw [color={rgb, 255:red, 208; green, 2; blue, 27 }  ,draw opacity=1 ]   (63.53,253.12) .. controls (87.35,243.25) and (101.33,272.11) .. (98.09,287.72) ;
\draw [shift={(97.6,289.6)}, rotate = 287.99] [color={rgb, 255:red, 208; green, 2; blue, 27 }  ,draw opacity=1 ][line width=0.75]    (10.93,-4.9) .. controls (6.95,-2.3) and (3.31,-0.67) .. (0,0) .. controls (3.31,0.67) and (6.95,2.3) .. (10.93,4.9)   ;
\draw [shift={(61.6,254)}, rotate = 333.79] [color={rgb, 255:red, 208; green, 2; blue, 27 }  ,draw opacity=1 ][line width=0.75]    (10.93,-4.9) .. controls (6.95,-2.3) and (3.31,-0.67) .. (0,0) .. controls (3.31,0.67) and (6.95,2.3) .. (10.93,4.9)   ;
\draw [color={rgb, 255:red, 208; green, 2; blue, 27 }  ,draw opacity=1 ]   (106.39,222.32) .. controls (105.38,236.48) and (112.86,244.58) .. (128.83,244.27) ;
\draw [shift={(130.6,244.2)}, rotate = 176.82] [color={rgb, 255:red, 208; green, 2; blue, 27 }  ,draw opacity=1 ][line width=0.75]    (10.93,-4.9) .. controls (6.95,-2.3) and (3.31,-0.67) .. (0,0) .. controls (3.31,0.67) and (6.95,2.3) .. (10.93,4.9)   ;
\draw [shift={(106.6,220.2)}, rotate = 97.13] [color={rgb, 255:red, 208; green, 2; blue, 27 }  ,draw opacity=1 ][line width=0.75]    (10.93,-4.9) .. controls (6.95,-2.3) and (3.31,-0.67) .. (0,0) .. controls (3.31,0.67) and (6.95,2.3) .. (10.93,4.9)   ;
\draw [color={rgb, 255:red, 208; green, 2; blue, 27 }  ,draw opacity=1 ]   (190.59,84.45) .. controls (191.22,120.03) and (229.58,133.49) .. (257.13,126.27) ;
\draw [shift={(258.8,125.8)}, rotate = 163.41] [color={rgb, 255:red, 208; green, 2; blue, 27 }  ,draw opacity=1 ][line width=0.75]    (10.93,-4.9) .. controls (6.95,-2.3) and (3.31,-0.67) .. (0,0) .. controls (3.31,0.67) and (6.95,2.3) .. (10.93,4.9)   ;
\draw [shift={(190.6,82.2)}, rotate = 91.51] [color={rgb, 255:red, 208; green, 2; blue, 27 }  ,draw opacity=1 ][line width=0.75]    (10.93,-4.9) .. controls (6.95,-2.3) and (3.31,-0.67) .. (0,0) .. controls (3.31,0.67) and (6.95,2.3) .. (10.93,4.9)   ;
\draw [color={rgb, 255:red, 208; green, 2; blue, 27 }  ,draw opacity=1 ]   (321.4,219.82) .. controls (371.94,218.7) and (397.96,260.77) .. (396.15,281.9) ;
\draw [shift={(395.9,283.8)}, rotate = 280.13] [color={rgb, 255:red, 208; green, 2; blue, 27 }  ,draw opacity=1 ][line width=0.75]    (10.93,-4.9) .. controls (6.95,-2.3) and (3.31,-0.67) .. (0,0) .. controls (3.31,0.67) and (6.95,2.3) .. (10.93,4.9)   ;
\draw [shift={(319,219.9)}, rotate = 357.23] [color={rgb, 255:red, 208; green, 2; blue, 27 }  ,draw opacity=1 ][line width=0.75]    (10.93,-4.9) .. controls (6.95,-2.3) and (3.31,-0.67) .. (0,0) .. controls (3.31,0.67) and (6.95,2.3) .. (10.93,4.9)   ;
\draw [color={rgb, 255:red, 208; green, 2; blue, 27 }  ,draw opacity=1 ]   (276.9,285.8) .. controls (259.06,281.31) and (282.38,259.34) .. (282.91,283.82) ;
\draw [shift={(282.9,285.8)}, rotate = 271.45] [color={rgb, 255:red, 208; green, 2; blue, 27 }  ,draw opacity=1 ][line width=0.75]    (10.93,-4.9) .. controls (6.95,-2.3) and (3.31,-0.67) .. (0,0) .. controls (3.31,0.67) and (6.95,2.3) .. (10.93,4.9)   ;
\draw [color={rgb, 255:red, 208; green, 2; blue, 27 }  ,draw opacity=1 ]   (304.9,268) .. controls (287.06,263.51) and (310.38,241.54) .. (310.91,266.02) ;
\draw [shift={(310.9,268)}, rotate = 271.45] [color={rgb, 255:red, 208; green, 2; blue, 27 }  ,draw opacity=1 ][line width=0.75]    (10.93,-4.9) .. controls (6.95,-2.3) and (3.31,-0.67) .. (0,0) .. controls (3.31,0.67) and (6.95,2.3) .. (10.93,4.9)   ;
\draw    (135.6,214.2) -- (56,295) ;

\draw (92,28.4) node [anchor=north west][inner sep=0.75pt]  [font=\Large]  {$X$};
\draw (11,191.4) node [anchor=north west][inner sep=0.75pt]  [font=\normalsize]  {$Z\simeq Q^{L}$};
\draw (382,193.4) node [anchor=north west][inner sep=0.75pt]    {$\mathbb{P}^{2}$};
\draw (205,236.4) node [anchor=north west][inner sep=0.75pt]    {$\psi $};
\draw (82,108.4) node [anchor=north west][inner sep=0.75pt]    {$\eta $};
\end{tikzpicture}
\caption{Blow-up model for a del Pezzo surface as in Proposition \ref{Prop:Proposition1_Z/2Z_transposition}.}
\label{Fig:Figure_blow-up_model_Z/2Z_1_case}
\end{figure}

(\ref{it:item_(2)_Proposition1}) Any two triplets of $\k$-rational points on $ \mathcal{Q}^{L} $ whose components are contained in pairwise distinct rulings of $ \mathcal{Q}^{L}_{L} $ can be sent onto each other by an element of $ \Aut_{\k}(\mathcal{Q}^{L}) $ by Lemma \ref{lem:Lemme1_action_trans_Aut(Q^L)_triplets_k_points}. It follows that any two del Pezzo surfaces satisfying our hypothesis are isomorphic if and only if there exist birational contractions $\eta, \eta'$ onto isomorphic del Pezzo surfaces $ \mathcal{Q}^{L} $ and $ \mathcal{Q}^{L'} $ of degree $ 8 $. This is the case if and only if $L$ and $L'$ are $\k$-isomorphic by \cite[Lemma 3.2(3)]{sz21}.

(\ref{it:item_(3)_Proposition1}) The group $ \Aut_{\k}(X) $ acts on the set of $(-1)$-curves of $ X $ preserving the intersection form, so we get an isomorphism $ \Psi : \Aut_{\k}(X) \rightarrow \Aut(\Pi_{L,\rho}) $ by Lemma \ref{lem:Lemma_faithful_action_Aut_k(X)_on_Pi_L_rho}.
Let us determine $ \Aut(\Pi_{L,\rho}) $. Since we have $\Aut(\Pi_{L,\rho}) \simeq \lbrace \alpha \in \Sym_{5} \, \vert \, (34) \circ \alpha = \alpha \circ (34) \rbrace = C_{\Sym_{5}}(\langle (34) \rangle) $, we get $ \Aut(\Pi_{L,\rho}) \simeq \langle (12),(25),(34) \rangle \simeq \Sym_{3} \times \Z/2\Z $.
We explicitly construct the corresponding geometric actions by lifting well-chosen involutions on $ \mathcal{Q}^{L} $. We define
\begin{align*}
\alpha \, &: \, ([u_{0}:u_{1}],[v_{0}:v_{1}]) \longmapsto ([u_{0}:u_{0}-u_{1}],[v_{0}:v_{0}-v_{1}]) \, ;\\
\beta \, &: \, ([u_{0}:u_{1}],[v_{0}:v_{1}]) \longmapsto ([v_{0}:v_{1}],[u_{0}:u_{1}]) \, ; \, \text{and} \\
\gamma \, &: \, ([u_{0}:u_{1}],[v_{0}:v_{1}]) \longmapsto ([u_{1}:u_{0}],[v_{1}:v_{0}]).
\end{align*} 
They are involutions of $ \mathcal{Q}^{L} $, and their lifts $ \widehat{\alpha}, \widehat{\beta}, \widehat{\gamma} $ are automorphisms of $X$ defined over $\k$ whose actions on the diagram $\Pi_{L}$ correspond to $(25)$, $(34)$, $(12)$, respectively. To conclude,
$ \Aut_{\k}(X)=\langle \widehat{\alpha},\widehat{\gamma} \rangle \times \langle \widehat{\beta}\rangle \simeq \Sym_{3} \times \mathbb{Z}/2\mathbb{Z} $. 

(\ref{it:item_(4)_Proposition1}) From \ref{it:item_(1)_Proposition1} we have a contraction $\eta : X \rightarrow \mathcal{Q}^{L}$ onto three rational points, so thanks to Lemma \ref{lem:Lemme1_action_trans_Aut(Q^L)_triplets_k_points} we get $\rk \, \NS(X)^{\Aut_{\k}(X)}=2$. Now, note that the curve $D_{12}$ is stabilized by $\Aut_{\k}(X)$ (see point \ref{it:item_(3)_Proposition1}) and can then be contracted equivariantly.
The contraction of $D_{12}$ gives a birational morphism to the rational del Pezzo surface $Y$ of degree $6$ in \cite[Lemma 4.10]{sz21}, which is obtained by blowing up $\F_{0} \simeq \p^{1} \times \p^{1}$ in a point of degree $2$ with splitting field $L$ and whose $\k$-automorphism group acts transitively on the set of $(-1)$-curves, implying that $Y \rightarrow \ast$ is an $\Aut_{\k}(X)$-Mori fibre space.

\end{proof}

\begin{remark}
We note that via the birational map $\psi : \Ql^{L} \dashrightarrow \p^{2}$ constructed in the proof of Proposition \ref{Prop:Proposition1_Z/2Z_transposition},(\ref{it:item_(1)_Proposition1}) we could also have lifted $\alpha,\beta,\gamma$ to $\k$-birational involutions of $\p^{2}$.
\end{remark}

\begin{proposition}
Let $X$ be a del Pezzo surface of degree $5$ such that $ \rho(\Gal(\kb/\k)) = \langle (12)(34) \rangle \simeq \mathbb{Z}/2\mathbb{Z} $ in $\Sym_{5}$ as indicated in Figure \ref{fig:figure(b)_option_Gal(kbarre/k)-action_on_Pikbarre}. Then the following holds:
\begin{enumerate}
\item there exist a quadratic extension $L/\k$ and two points $ r=\lbrace r_{1},r_{2} \rbrace , s=\lbrace s_{1},s_{2} \rbrace $ in $\mathbb{P}^{2}$ of degree $2$ with splitting field $L$ such that $X$ is isomorphic to the blow-up of $ \mathbb{P}^{2} $ in $r$ and $s$.\label{it:item_(1)_Proposition_2}
\item Any two such surfaces are isomorphic if and only if the respective quadratic field extensions are $\k$-isomorphic.\label{it:item_(2)_Proposition_2}
\item $ \Aut_{\k}(X)=\langle \widehat{\alpha},\widehat{\beta} \rangle \simeq \D_{4} $, where $\widehat{\alpha}$ and $\widehat{\beta}$ are respectively the lifts of an involution $\alpha$ and an automorphism $\beta$ of order four of $\p^{2}$.\label{it:item_(3)_Proposition2_Z/2Z_2}
\item $ \rk \, \NS(X)^{\Aut_{\k}(X)} = 2 $, so in particular $ X \rightarrow \ast $ is not an $ \Aut_{\k}(X) $-Mori fibre space, but $X$ carries an $\Aut_{\k}(X)$-Mori fibre space structure $X \rightarrow \p^{1}$ given by the linear system of conics passing through the points $r$ and $s$ on $\p^{2}$.
\label{it:item_(4)_Proposition_2} 
\end{enumerate}
\label{Prop:Proposition2_Z/2Z_double_transposition}
\end{proposition}

\begin{proof}
(\ref{it:item_(1)_Proposition_2}) The Petersen graph $ \Pi_{\overline{\k},\rho} $ of $X$ presents three disjoint $(-1)$-curves: the $\k$-curve $D_{34}$ and the irreducible curve $E$ whose geometric components are $E_{1},E_{2}$. The contraction of $D_{34}$ and $E$ yields a birational morphism $ \eta : X \rightarrow Z $ to a rational del Pezzo surface $Z$ of degree $8$ such that $\rk \, \NS(Z) = 1$ and hence $ Z \simeq \mathcal{Q}^{L} $ for some quadratic extension $ L=\mathbf{k}(a_{1}) $ by \cite[Lemma 3.2(1)]{sz21}. Let $a_{2}=a_{1}^{g}$ be the conjugate of $a_{1}$ under the action of $\Gal(L/\k)=:\langle g \rangle$. The morphism $\eta$ contracts the curve $D_{34}$ onto a $\k$-rational point $q \in \mathcal{Q}^{L}(\k)$ and the curve $E$ onto a point $ p=\lbrace p_{1},p_{2} \rbrace $ of degree $2$ whose geometric components are exchanged by the $\Gal(L/\k)$-action.
Moreover, there is a birational map $ \psi : \mathcal{Q}^{L} \dashrightarrow \mathbb{P}^{2} $ that is the composition of the blow-up of $q$ with the contraction of a curve onto a point of degree $2$ in $ \mathbb{P}^{2} $ whose splitting field is $L$. The surface $X$ is therefore isomorphic to the blow-up of $\mathbb{P}^{2}$ in two points $ r,s $ of degree $2$ whose splitting field is $L$, and we can assume that $ r=\lbrace [a_{1}:0:1],[a_{2}:0:1] \rbrace $, $ s=\lbrace [a_{1}:1:0],[a_{2}:1:0] \rbrace $ 
by Lemma \ref{lem:Lemme_6.11_&_6.15_article_Julia},(\ref{it:item_(b)_Lemme_6.15_Julia}).

(\ref{it:item_(2)_Proposition_2}) Applying Lemma \ref{lem:Lemme_6.11_&_6.15_article_Julia},(\ref{it:item_(b)_Lemme_6.15_Julia}) to $ \lbrace [a_{1}:0:1],[a_{2}:0:1],[a_{1}:1:0],[a_{2}:1:0] \rbrace $ and $ \lbrace q_{1},q_{2},q_{3},q_{4} \rbrace $ for some $ q_{1},q_{2},q_{3},q_{4} \in \mathbb{P}^{2}(L) $, no three collinear, we see that we can send any two points $ t=\lbrace q_{1},q_{2}\rbrace $, $u=\lbrace q_{3},q_{4}\rbrace $ in $\mathbb{P}^{2} $ of degree $2$ whose splitting field is $L$ onto $ r=\lbrace [a_{1}:0:1],[a_{2}:0:1] \rbrace $, $ s=\lbrace [a_{1}:1:0],[a_{2}:1:0] \rbrace $ by an element from $ \text{PGL}_{3}(\k) $. It follows that any such two surfaces satisfying our hypothesis are isomorphic if and only if the respective quadratic extensions are $\k$-isomorphic.

(\ref{it:item_(3)_Proposition2_Z/2Z_2}) By Lemma \ref{lem:Lemma_faithful_action_Aut_k(X)_on_Pi_L_rho}, the action of $ \Aut_{\k}(X) $ on the set of $(-1)$-curves of $X$ yields an isomorphism $ \Psi : \Aut_{\k}(X) \rightarrow \Aut(\Pi_{L,\rho}) $. Let us determine $ \Aut(\Pi_{L,\rho}) $. Since $\Aut(\Pi_{L,\rho}) \simeq \lbrace \alpha \in \Sym_{5} \, \vert \, (12)(34) \circ \alpha = \alpha \circ (12)(34) \rbrace = C_{\Sym_{5}}(\langle (12)(34) \rangle)$, we get $\Aut(\Pi_{L,\rho}) \simeq \langle (34),(1324) \rangle \simeq \text{D}_{4}$.
Finally, we construct the corresponding geometric actions on $X$. We set
\[
\alpha \colon [x:y:z] \longmapsto [x-(a_{1}+a_{2})y:-y:z], \quad
\beta \colon [x:y:z] \longmapsto [x-(a_{1}+a_{2})y:z:-y].
\]
Then $\alpha$ is an involution of $\p^{2}$ and $\beta$ is an automorphism of $\p^{2}$ of order four, both defined over $\k$ and satisfying the relation $ \alpha \beta \alpha=\beta^{-1} $. The lifts of $\alpha$ and $\beta$ on $X$ give automorphisms $\widehat{\alpha}, \widehat{\beta}$ that act on the $(-1)$-curves as respectively $(34)$ and $(1324)$.
In conclusion, $\Aut_{\k}(X)=\langle \widehat{\alpha},\widehat{\beta} \rangle \simeq \text{D}_{4} $.

(\ref{it:item_(4)_Proposition_2}) The set $\{E_{1},E_{2},E_{3},E_{4}\}$ forms one $\Aut_{\k}(X)$-orbit and hence can be equivariantly contracted onto $\mathbb{P}^{2}$, giving $\rk \, \NS(X)^{\Aut_{\k}(X)}=2$. Moreover, the conic bundle induced by $2H-E_{1}-E_{2}-E_{3}-E_{4}$ on $X_{\kb}$ (see Figure \ref{fig:fig(d)_conic_bundle_point_of_view}), where $H$ denotes the pullback of a general line of $\p^{2}_{\kb}$, is defined over $\k$ and preserved by the action of $\Aut_{\k}(X)$. Since $\rk \, \NS(X)^{\Aut_{\k}(X)}=2$, the corresponding conic bundle structure $X \rightarrow \p^{1}$ given by the linear system of conics through $r$ and $s$ on $\p^{2}$ is an $\Aut_{\k}(X)$-Mori fibre space.
     
\end{proof}

\begin{remark}
We note that via the birational map $\psi^{-1} : \p^{2} \dashrightarrow \Ql^{L}$ given in the proof of Proposition \ref{Prop:Proposition2_Z/2Z_double_transposition},(\ref{it:item_(1)_Proposition_2}) we could also have lifted $\alpha, \beta$ to birational transformations of $\Ql^{L}$.
\end{remark}

\begin{proposition}
Let $X$ be a del Pezzo surface of degree $5$ such that $\rho(\Gal(\kb/\k)) = \langle(12)\rangle\times\langle(34)\rangle \simeq \Z/2\Z \times \Z/2\Z $ in $\Sym_{5}$ as indicated in Figure \ref{fig:figure(c)_option_Gal(kbarre/k)-action_on_Pikbarre}. Then there exist quadratic extensions $L=\k(a_{1})$ and $L'=\k(b_{1})$ that are not $\k$-isomorphic, with $ t^2+at+\tilde{a}=(t-a_{1})(t-a_{2}), \, t^2+bt+\tilde{b}=(t-b_{1})(t-b_{2}) \in \k[t] $ the minimal polynomials of $a_{1},b_{1}$, and such that the following holds:
\begin{enumerate}
\item there is a birational morphism $\varepsilon : X \longrightarrow \p^{2} $ contracting two irreducible curves $E, F$ onto two points $ \varepsilon(E)=\lbrace [b_{1}:0:1],[b_{2}:0:1] \rbrace, \, \varepsilon(F)=\lbrace [a_{1}:1:0],[a_{2}:1:0] \rbrace$ of degree $2$ whose splitting fields are respectively $ L'=\k(b_{1}), L=\k(a_{1}) $.\label{it:item_(1)_Proposition5_Z2xZ2_1}
\item Any two such surfaces $X$ and $\tilde{X}$ are isomorphic if and only if $\tilde{L}, \tilde{L'}$ are respectively $\k$-isomorphic to $L,L'$.\label{it:item_(2)_Proposition5_Z2xZ2_1}
\item $\Aut_{\k}(X)=\langle \widehat{\alpha} \rangle \times \langle \widehat{\beta} \rangle \simeq \Z/2\Z \times \Z/2\Z $, where $\widehat{\alpha}, \widehat{\beta}$ are the lifts of two involutions $\alpha, \beta$ of $\p^{2}$.\label{it:item_(3)_Proposition5_Z2xZ2_1}
\item $ \rk \, \NS(X)^{\Aut_{\k}(X)}=3 $ and $ \varepsilon \Aut_{\k}(X) \varepsilon^{-1} = \Aut_{\k}(\p^{2},\lbrace [b_{1}:0:1],[b_{2}:0:1] \rbrace,\lbrace [a_{1}:1:0],[a_{2}:1:0] \rbrace) $. In particular, $ X \longrightarrow \ast $ is not an $\Aut_{\k}(X)$-Mori fibre space and the $\Aut_{\k}(X)$-minimal models of $X$ are $\p^{2}$ and the quadric $\Ql^{L}$.\label{it:item_(4)_Proposition5_Z2xZ2_1}
\end{enumerate}
\label{Prop:Proposition5_Z2xZ2_1}
\end{proposition}

\begin{proof}
(\ref{it:item_(1)_Proposition5_Z2xZ2_1}) The Petersen graph $ \Pi_{\kb,\rho} $ of $X$ presents three disjoint $(-1)$-curves: the $\k$-curve $D_{34}$ and the curve denoted $E$ whose geometric components are $E_{1}, E_{2}$. Their contraction yields a birational morphism $ \eta : X \rightarrow Z $ onto a rational del Pezzo surface $Z$ of degree $8$ such that $\rk \, \NS(Z) = 1$ and hence $ Z \simeq \mathcal{Q}^{L} $ for some quadratic extension $ L/\k $ by \cite[Lemma 3.2(1)]{sz21}.
\begin{figure}[h]
\centering
\begin{tikzpicture}[x=0.6pt,y=0.6pt,yscale=-1,xscale=1, scale=0.8, every node/.style={scale=0.8}]

\draw    (61.6,213.4) -- (61.6,296.6) ;
\draw    (89.6,214) -- (89.6,296.6) ;
\draw    (130.6,213) -- (130.6,297.6) ;
\draw    (179.6,80) -- (179.6,160) ;
\draw    (210.6,61.6) -- (210.6,161) ;
\draw    (140.6,289.6) -- (50.6,289.6) ;
\draw    (259.6,150) -- (169.6,150) ;
\draw    (141.6,259) -- (50.6,259.6) ;
\draw    (140.6,220) -- (50.6,220) ;
\draw    (268.6,120.6) -- (169.6,120.6) ;
\draw    (169,93) -- (229.6,60.6) ;
\draw    (248,161) -- (269.6,90.6) ;
\draw    (219,60) -- (272.6,100.6) ;
\draw    (164.6,165.2) -- (133.07,194.24) ;
\draw [shift={(131.6,195.6)}, rotate = 317.35] [color={rgb, 255:red, 0; green, 0; blue, 0 }  ][line width=0.75]    (10.93,-3.29) .. controls (6.95,-1.4) and (3.31,-0.3) .. (0,0) .. controls (3.31,0.3) and (6.95,1.4) .. (10.93,3.29)   ;
\draw    (260.6,165.2) -- (293.13,195.24) ;
\draw [shift={(294.6,196.6)}, rotate = 222.72] [color={rgb, 255:red, 0; green, 0; blue, 0 }  ][line width=0.75]    (10.93,-3.29) .. controls (6.95,-1.4) and (3.31,-0.3) .. (0,0) .. controls (3.31,0.3) and (6.95,1.4) .. (10.93,3.29)   ;
\draw    (129.6,59.2) -- (156.08,81.9) ;
\draw [shift={(157.6,83.2)}, rotate = 220.6] [color={rgb, 255:red, 0; green, 0; blue, 0 }  ][line width=0.75]    (10.93,-3.29) .. controls (6.95,-1.4) and (3.31,-0.3) .. (0,0) .. controls (3.31,0.3) and (6.95,1.4) .. (10.93,3.29)   ;
\draw    (321.6,209.6) -- (278.6,297.6) ;
\draw    (305.6,210.6) -- (405.6,294.2) ;
\draw    (272.6,289.2) -- (405.6,286.2) ;
\draw    (317.6,208.2) -- (308.6,276.2) ;
\draw    (302.6,270) -- (409.6,289.2) ;
\draw    (100.6,62.2) -- (99.62,193.2) ;
\draw [shift={(99.6,195.2)}, rotate = 270.43] [color={rgb, 255:red, 0; green, 0; blue, 0 }  ][line width=0.75]    (10.93,-3.29) .. controls (6.95,-1.4) and (3.31,-0.3) .. (0,0) .. controls (3.31,0.3) and (6.95,1.4) .. (10.93,3.29)   ;
\draw  [fill={rgb, 255:red, 0; green, 0; blue, 0 }  ,fill opacity=1 ] (58.8,289.1) .. controls (58.8,287.39) and (60.19,286) .. (61.9,286) .. controls (63.61,286) and (65,287.39) .. (65,289.1) .. controls (65,290.81) and (63.61,292.2) .. (61.9,292.2) .. controls (60.19,292.2) and (58.8,290.81) .. (58.8,289.1) -- cycle ;
\draw  [fill={rgb, 255:red, 0; green, 0; blue, 0 }  ,fill opacity=1 ] (86.8,259.9) .. controls (86.8,258.19) and (88.19,256.8) .. (89.9,256.8) .. controls (91.61,256.8) and (93,258.19) .. (93,259.9) .. controls (93,261.61) and (91.61,263) .. (89.9,263) .. controls (88.19,263) and (86.8,261.61) .. (86.8,259.9) -- cycle ;
\draw  [fill={rgb, 255:red, 0; green, 0; blue, 0 }  ,fill opacity=1 ] (127.8,219.9) .. controls (127.8,218.19) and (129.19,216.8) .. (130.9,216.8) .. controls (132.61,216.8) and (134,218.19) .. (134,219.9) .. controls (134,221.61) and (132.61,223) .. (130.9,223) .. controls (129.19,223) and (127.8,221.61) .. (127.8,219.9) -- cycle ;
\draw  [fill={rgb, 255:red, 0; green, 0; blue, 0 }  ,fill opacity=1 ] (176.8,149.9) .. controls (176.8,148.19) and (178.19,146.8) .. (179.9,146.8) .. controls (181.61,146.8) and (183,148.19) .. (183,149.9) .. controls (183,151.61) and (181.61,153) .. (179.9,153) .. controls (178.19,153) and (176.8,151.61) .. (176.8,149.9) -- cycle ;
\draw  [fill={rgb, 255:red, 0; green, 0; blue, 0 }  ,fill opacity=1 ] (207.8,119.9) .. controls (207.8,118.19) and (209.19,116.8) .. (210.9,116.8) .. controls (212.61,116.8) and (214,118.19) .. (214,119.9) .. controls (214,121.61) and (212.61,123) .. (210.9,123) .. controls (209.19,123) and (207.8,121.61) .. (207.8,119.9) -- cycle ;
\draw  [fill={rgb, 255:red, 0; green, 0; blue, 0 }  ,fill opacity=1 ] (279.8,288.9) .. controls (279.8,287.19) and (281.19,285.8) .. (282.9,285.8) .. controls (284.61,285.8) and (286,287.19) .. (286,288.9) .. controls (286,290.61) and (284.61,292) .. (282.9,292) .. controls (281.19,292) and (279.8,290.61) .. (279.8,288.9) -- cycle ;
\draw  [fill={rgb, 255:red, 0; green, 0; blue, 0 }  ,fill opacity=1 ] (306.8,270.9) .. controls (306.8,269.19) and (308.19,267.8) .. (309.9,267.8) .. controls (311.61,267.8) and (313,269.19) .. (313,270.9) .. controls (313,272.61) and (311.61,274) .. (309.9,274) .. controls (308.19,274) and (306.8,272.61) .. (306.8,270.9) -- cycle ;
\draw  [fill={rgb, 255:red, 0; green, 0; blue, 0 }  ,fill opacity=1 ] (312.8,219.9) .. controls (312.8,218.19) and (314.19,216.8) .. (315.9,216.8) .. controls (317.61,216.8) and (319,218.19) .. (319,219.9) .. controls (319,221.61) and (317.61,223) .. (315.9,223) .. controls (314.19,223) and (312.8,221.61) .. (312.8,219.9) -- cycle ;
\draw  [fill={rgb, 255:red, 0; green, 0; blue, 0 }  ,fill opacity=1 ] (392.8,286.9) .. controls (392.8,285.19) and (394.19,283.8) .. (395.9,283.8) .. controls (397.61,283.8) and (399,285.19) .. (399,286.9) .. controls (399,288.61) and (397.61,290) .. (395.9,290) .. controls (394.19,290) and (392.8,288.61) .. (392.8,286.9) -- cycle ;
\draw  [dash pattern={on 4.5pt off 4.5pt}]  (155,260) -- (274.6,261.18) ;
\draw [shift={(276.6,261.2)}, rotate = 180.57] [color={rgb, 255:red, 0; green, 0; blue, 0 }  ][line width=0.75]    (10.93,-3.29) .. controls (6.95,-1.4) and (3.31,-0.3) .. (0,0) .. controls (3.31,0.3) and (6.95,1.4) .. (10.93,3.29)   ;
\draw    (274,294) -- (357.6,239.2) ;
\draw    (170.6,159.2) -- (254.6,77.2) ;
\draw [color={rgb, 255:red, 208; green, 2; blue, 27 }  ,draw opacity=1 ]   (63.53,253.12) .. controls (87.35,243.25) and (101.33,272.11) .. (98.09,287.72) ;
\draw [shift={(97.6,289.6)}, rotate = 287.99] [color={rgb, 255:red, 208; green, 2; blue, 27 }  ,draw opacity=1 ][line width=0.75]    (10.93,-4.9) .. controls (6.95,-2.3) and (3.31,-0.67) .. (0,0) .. controls (3.31,0.67) and (6.95,2.3) .. (10.93,4.9)   ;
\draw [shift={(61.6,254)}, rotate = 333.79] [color={rgb, 255:red, 208; green, 2; blue, 27 }  ,draw opacity=1 ][line width=0.75]    (10.93,-4.9) .. controls (6.95,-2.3) and (3.31,-0.67) .. (0,0) .. controls (3.31,0.67) and (6.95,2.3) .. (10.93,4.9)   ;
\draw [color={rgb, 255:red, 208; green, 2; blue, 27 }  ,draw opacity=1 ]   (106.39,222.32) .. controls (105.38,236.48) and (112.86,244.58) .. (128.83,244.27) ;
\draw [shift={(130.6,244.2)}, rotate = 176.82] [color={rgb, 255:red, 208; green, 2; blue, 27 }  ,draw opacity=1 ][line width=0.75]    (10.93,-4.9) .. controls (6.95,-2.3) and (3.31,-0.67) .. (0,0) .. controls (3.31,0.67) and (6.95,2.3) .. (10.93,4.9)   ;
\draw [shift={(106.6,220.2)}, rotate = 97.13] [color={rgb, 255:red, 208; green, 2; blue, 27 }  ,draw opacity=1 ][line width=0.75]    (10.93,-4.9) .. controls (6.95,-2.3) and (3.31,-0.67) .. (0,0) .. controls (3.31,0.67) and (6.95,2.3) .. (10.93,4.9)   ;
\draw [color={rgb, 255:red, 208; green, 2; blue, 27 }  ,draw opacity=1 ]   (190.59,84.45) .. controls (191.22,120.03) and (229.58,133.49) .. (257.13,126.27) ;
\draw [shift={(258.8,125.8)}, rotate = 163.41] [color={rgb, 255:red, 208; green, 2; blue, 27 }  ,draw opacity=1 ][line width=0.75]    (10.93,-4.9) .. controls (6.95,-2.3) and (3.31,-0.67) .. (0,0) .. controls (3.31,0.67) and (6.95,2.3) .. (10.93,4.9)   ;
\draw [shift={(190.6,82.2)}, rotate = 91.51] [color={rgb, 255:red, 208; green, 2; blue, 27 }  ,draw opacity=1 ][line width=0.75]    (10.93,-4.9) .. controls (6.95,-2.3) and (3.31,-0.67) .. (0,0) .. controls (3.31,0.67) and (6.95,2.3) .. (10.93,4.9)   ;
\draw [color={rgb, 255:red, 208; green, 2; blue, 27 }  ,draw opacity=1 ]   (321.4,219.82) .. controls (371.94,218.7) and (397.96,260.77) .. (396.15,281.9) ;
\draw [shift={(395.9,283.8)}, rotate = 280.13] [color={rgb, 255:red, 208; green, 2; blue, 27 }  ,draw opacity=1 ][line width=0.75]    (10.93,-4.9) .. controls (6.95,-2.3) and (3.31,-0.67) .. (0,0) .. controls (3.31,0.67) and (6.95,2.3) .. (10.93,4.9)   ;
\draw [shift={(319,219.9)}, rotate = 357.23] [color={rgb, 255:red, 208; green, 2; blue, 27 }  ,draw opacity=1 ][line width=0.75]    (10.93,-4.9) .. controls (6.95,-2.3) and (3.31,-0.67) .. (0,0) .. controls (3.31,0.67) and (6.95,2.3) .. (10.93,4.9)   ;
\draw [color={rgb, 255:red, 208; green, 2; blue, 27 }  ,draw opacity=1 ]   (276.9,285.8) .. controls (259.06,281.31) and (282.38,259.34) .. (282.91,283.82) ;
\draw [shift={(282.9,285.8)}, rotate = 271.45] [color={rgb, 255:red, 208; green, 2; blue, 27 }  ,draw opacity=1 ][line width=0.75]    (10.93,-4.9) .. controls (6.95,-2.3) and (3.31,-0.67) .. (0,0) .. controls (3.31,0.67) and (6.95,2.3) .. (10.93,4.9)   ;
\draw [color={rgb, 255:red, 208; green, 2; blue, 27 }  ,draw opacity=1 ]   (304.9,268) .. controls (287.06,263.51) and (310.38,241.54) .. (310.91,266.02) ;
\draw [shift={(310.9,268)}, rotate = 271.45] [color={rgb, 255:red, 208; green, 2; blue, 27 }  ,draw opacity=1 ][line width=0.75]    (10.93,-4.9) .. controls (6.95,-2.3) and (3.31,-0.67) .. (0,0) .. controls (3.31,0.67) and (6.95,2.3) .. (10.93,4.9)   ;
\draw    (135.6,214.2) -- (56,295) ;
\draw [color={rgb, 255:red, 24; green, 24; blue, 210 }  ,draw opacity=1 ]   (63,236.01) -- (86.6,236.19) ;
\draw [shift={(88.6,236.2)}, rotate = 180.42] [color={rgb, 255:red, 24; green, 24; blue, 210 }  ,draw opacity=1 ][line width=0.75]    (10.93,-4.9) .. controls (6.95,-2.3) and (3.31,-0.67) .. (0,0) .. controls (3.31,0.67) and (6.95,2.3) .. (10.93,4.9)   ;
\draw [shift={(61,236)}, rotate = 0.42] [color={rgb, 255:red, 24; green, 24; blue, 210 }  ,draw opacity=1 ][line width=0.75]    (10.93,-4.9) .. controls (6.95,-2.3) and (3.31,-0.67) .. (0,0) .. controls (3.31,0.67) and (6.95,2.3) .. (10.93,4.9)   ;
\draw [color={rgb, 255:red, 24; green, 24; blue, 210 }  ,draw opacity=1 ]   (112.66,261.2) -- (113.54,288.2) ;
\draw [shift={(113.6,290.2)}, rotate = 268.15] [color={rgb, 255:red, 24; green, 24; blue, 210 }  ,draw opacity=1 ][line width=0.75]    (10.93,-4.9) .. controls (6.95,-2.3) and (3.31,-0.67) .. (0,0) .. controls (3.31,0.67) and (6.95,2.3) .. (10.93,4.9)   ;
\draw [shift={(112.6,259.2)}, rotate = 88.15] [color={rgb, 255:red, 24; green, 24; blue, 210 }  ,draw opacity=1 ][line width=0.75]    (10.93,-4.9) .. controls (6.95,-2.3) and (3.31,-0.67) .. (0,0) .. controls (3.31,0.67) and (6.95,2.3) .. (10.93,4.9)   ;
\draw [color={rgb, 255:red, 24; green, 24; blue, 210 }  ,draw opacity=1 ]   (63.15,283.91) .. controls (65.02,269.92) and (70.19,262.55) .. (88.19,260.11) ;
\draw [shift={(89.9,259.9)}, rotate = 173.54] [color={rgb, 255:red, 24; green, 24; blue, 210 }  ,draw opacity=1 ][line width=0.75]    (10.93,-4.9) .. controls (6.95,-2.3) and (3.31,-0.67) .. (0,0) .. controls (3.31,0.67) and (6.95,2.3) .. (10.93,4.9)   ;
\draw [shift={(62.9,286)}, rotate = 276.14] [color={rgb, 255:red, 24; green, 24; blue, 210 }  ,draw opacity=1 ][line width=0.75]    (10.93,-4.9) .. controls (6.95,-2.3) and (3.31,-0.67) .. (0,0) .. controls (3.31,0.67) and (6.95,2.3) .. (10.93,4.9)   ;
\draw [color={rgb, 255:red, 24; green, 24; blue, 210 }  ,draw opacity=1 ]   (287.76,289.95) .. controls (301.49,297.65) and (315.61,293.49) .. (310.34,272.54) ;
\draw [shift={(309.9,270.9)}, rotate = 73.96] [color={rgb, 255:red, 24; green, 24; blue, 210 }  ,draw opacity=1 ][line width=0.75]    (10.93,-4.9) .. controls (6.95,-2.3) and (3.31,-0.67) .. (0,0) .. controls (3.31,0.67) and (6.95,2.3) .. (10.93,4.9)   ;
\draw [shift={(286,288.9)}, rotate = 32.5] [color={rgb, 255:red, 24; green, 24; blue, 210 }  ,draw opacity=1 ][line width=0.75]    (10.93,-4.9) .. controls (6.95,-2.3) and (3.31,-0.67) .. (0,0) .. controls (3.31,0.67) and (6.95,2.3) .. (10.93,4.9)   ;
\draw [color={rgb, 255:red, 24; green, 24; blue, 210 }  ,draw opacity=1 ]   (391.6,283.3) .. controls (373.96,270.56) and (370.73,304.88) .. (393.47,289.03) ;
\draw [shift={(394.9,288)}, rotate = 143.02] [color={rgb, 255:red, 24; green, 24; blue, 210 }  ,draw opacity=1 ][line width=0.75]    (10.93,-4.9) .. controls (6.95,-2.3) and (3.31,-0.67) .. (0,0) .. controls (3.31,0.67) and (6.95,2.3) .. (10.93,4.9)   ;
\draw [color={rgb, 255:red, 24; green, 24; blue, 210 }  ,draw opacity=1 ]   (310.6,216.2) .. controls (292.96,203.46) and (290.69,235.86) .. (313.47,219.94) ;
\draw [shift={(314.9,218.9)}, rotate = 143.02] [color={rgb, 255:red, 24; green, 24; blue, 210 }  ,draw opacity=1 ][line width=0.75]    (10.93,-4.9) .. controls (6.95,-2.3) and (3.31,-0.67) .. (0,0) .. controls (3.31,0.67) and (6.95,2.3) .. (10.93,4.9)   ;
\draw [color={rgb, 255:red, 24; green, 24; blue, 210 }  ,draw opacity=1 ]   (193.3,76.8) .. controls (175.46,72.31) and (198.78,50.34) .. (199.31,74.82) ;
\draw [shift={(199.3,76.8)}, rotate = 271.45] [color={rgb, 255:red, 24; green, 24; blue, 210 }  ,draw opacity=1 ][line width=0.75]    (10.93,-4.9) .. controls (6.95,-2.3) and (3.31,-0.67) .. (0,0) .. controls (3.31,0.67) and (6.95,2.3) .. (10.93,4.9)   ;
\draw [color={rgb, 255:red, 24; green, 24; blue, 210 }  ,draw opacity=1 ]   (258.6,130.8) .. controls (273.22,122.03) and (278.34,147.47) .. (257.27,137.63) ;
\draw [shift={(255.6,136.8)}, rotate = 27.55] [color={rgb, 255:red, 24; green, 24; blue, 210 }  ,draw opacity=1 ][line width=0.75]    (10.93,-4.9) .. controls (6.95,-2.3) and (3.31,-0.67) .. (0,0) .. controls (3.31,0.67) and (6.95,2.3) .. (10.93,4.9)   ;

\draw (92,28.4) node [anchor=north west][inner sep=0.75pt]  [font=\Large]  {$X$};
\draw (11,191.4) node [anchor=north west][inner sep=0.75pt]  [font=\normalsize]  {$Z\simeq Q^{L}$};
\draw (382,193.4) node [anchor=north west][inner sep=0.75pt]    {$\mathbb{P}^{2}$};
\draw (205,236.4) node [anchor=north west][inner sep=0.75pt]    {$\psi $};
\draw (82,108.4) node [anchor=north west][inner sep=0.75pt]    {$\eta $};
\end{tikzpicture}
\caption{Blow-up model for a del Pezzo surface as in Proposition \ref{Prop:Proposition5_Z2xZ2_1}.}
\label{Fig:Figure_blow-up_model_Z/2ZxZ/2Z_1}
\end{figure}
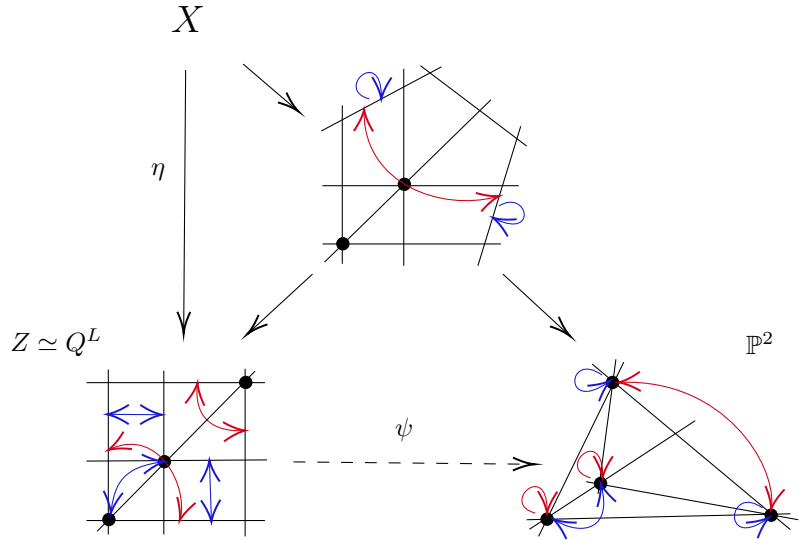
The Figure \ref{Fig:Figure_blow-up_model_Z/2ZxZ/2Z_1} shows the induced $\Gal(\kb/\k)$-action on the image by $\eta$ of $\Pi_{\kb}$. Then the image of the generator $g$ of $\Gal(L/\k)$ in $\Sym_{5}$ is either $(34)$ or $(12)(34)$. It follows that the splitting field of $ \eta(E)=\lbrace q_{1},q_{2} \rbrace=:q $ in $\mathcal{Q}^{L}$ is a quadratic extension $L'$ not $\k$-isomorphic to $L$, such that the generator $g'$ of $\Gal(L'/\k)$ induces the action of $(12)$ on $\Pi_{\kb}$. Let $a_{1} \in L, b_{1} \in L'$ such that $L=\k(a_{1})$ and $L'=\k(b_{1})$.
As in Propositions \ref{Prop:Proposition1_Z/2Z_transposition} and \ref{Prop:Proposition2_Z/2Z_double_transposition}, there is a birational map $ \psi : \mathcal{Q}^{L} \dashrightarrow \p^{2} $ that is the composition of the blow-up of $p:=\eta(D_{34})$ with the contraction of the curve $F$ whose geometric components are $E_{3},E_{4}$ onto a point $r=\lbrace r_{1},r_{2} \rbrace$ of degree $2$ in $\p^{2}$ whose splitting field is $L$ (see Figure \ref{Fig:Figure_blow-up_model_Z/2ZxZ/2Z_1}). The surface $X$ is therefore isomorphic to the blow-up $\varepsilon : X \rightarrow \p^{2}$ of $\p^{2}$ in the two points $s:=\psi(q)$ and $r$ of degree $2$ whose splitting fields are respectively $L'$ and $L$, and we can assume that $ s=\lbrace [b_{1}:0:1],[b_{2}:0:1] \rbrace $ and $ r=\lbrace [a_{1}:1:0],[a_{2}:1:0] \rbrace $
by Lemma \ref{lem:Lemme_6.11_&_6.15_article_Julia}.

(\ref{it:item_(2)_Proposition5_Z2xZ2_1}) Applying Lemma \ref{lem:Lemme_6.11_&_6.15_article_Julia},(\ref{it:item_(b)_Lemme_6.15_Julia}) to the sets $ \lbrace [a_{1}:1:0],[a_{2}:1:0],[b_{1}:0:1],[b_{2}:0:1] \rbrace $ and $ \lbrace w_{1},w_{2},w_{3},w_{4} \rbrace $ for some $ w_{1},w_{2},w_{3},w_{4} \in \mathbb{P}^{2}(LL') $, no three collinear, we see that we can send any two points $ t=\lbrace w_{1},w_{2}\rbrace $, $u=\lbrace w_{3},w_{4}\rbrace $ of degree $2$ in $\mathbb{P}^{2} $ whose splitting fields are $L, L'$ onto $ r=\lbrace [a_{1}:1:0],[a_{2}:1:0] \rbrace $, $ s=\lbrace [b_{1}:0:~1],[b_{2}:0:1] \rbrace $ by an element from $ \text{PGL}_{3}(\k) $. This gives the assertion.

(\ref{it:item_(3)_Proposition5_Z2xZ2_1}) The action of $ \Aut_{\k}(X) $ on the set of $(-1)$-curves of $X$ yields an isomorphism $\Psi : \Aut_{\k}(X) \rightarrow \Aut(\Pi_{LL',\rho})$ by Lemma \ref{lem:Lemma_faithful_action_Aut_k(X)_on_Pi_L_rho}. Let us determine $ \Aut(\Pi_{LL',\rho}) $. As before, since we have $\Aut(\Pi_{LL',\rho}) \simeq \lbrace \sigma \in \Sym_{5} \, \vert \, (12) \circ \sigma = \sigma \circ (12) \,\, \text{and} \,\, (34) \circ \sigma = \sigma \circ (34) \rbrace = C_{\Sym_{5}}(\langle (12),(34) \rangle)$, we get directly $\Aut(\Pi_{LL',\rho}) \simeq \langle (12),(34) \rangle \simeq \Z/2\Z \times \Z/2\Z$.
We then construct the geometric actions corresponding to the induced actions of the two generators of $\Gal(LL'/\k)$. 
Set
\[
\alpha \colon [x:y:z] \longmapsto [x-(b_{1}+b_{2})y:y:-z],\quad 
\beta \colon [x:y:z] \longmapsto [x-(a_{1}+a_{2})y:-y:z].
\]
Then $\alpha$ and $\beta$ are involutions of $\p^{2}$ defined over $\k$ and they commute. Their lifts $\widehat{\alpha}$ and $\widehat{\beta}$ are automorphisms of $X$ that commute and that respectively act as $(12)$ and $(34)$ on $\Pi_{LL'}$. In conclusion, $\Aut_{\k}(X) = \langle\widehat{\alpha}\rangle \times \langle\widehat{\beta}\rangle \simeq \Z/2\Z \times \Z/2\Z $.

(\ref{it:item_(4)_Proposition5_Z2xZ2_1}) The sets $\{E_{1},E_{2}\}$ and $\{E_{3},E_{4}\}$ form two $\Aut_{\k}(X)$-orbits and hence can be equivariantly contracted onto $\p^{2}$, giving $\rk \, \NS(X)^{\Aut_{\k}(X)}=3$. Note that on the other hand we can also equivariantly contract the fixed curve $D_{12}$ (resp. $D_{34}$) and the orbit $\lbrace E_{3},E_{4} \rbrace$ (resp. $\lbrace E_{1},E_{2} \rbrace$) onto $\Ql^{L}$.

\end{proof}

\subsection{Del Pezzo surfaces in Figures (\ref{fig:figure(0)_option_Gal(kbarre/k)-action_on_Pikbarre}),  (\ref{fig:figure(d)_option_Gal(kbarre/k)-action_on_Pikbarre}), (\ref{fig:figure(e)_option_Gal(kbarre/k)-action_on_Pikbarre}), (\ref{fig:figure(f)_option_Gal(kbarre/k)-action_on_Pikbarre}), (\ref{fig:figure(g)_option_Gal(kbarre/k)-action_on_Pikbarre}), (\ref{fig:figure(h)_option_Gal(kbarre/k)-action_on_Pikbarre}), (\ref{fig:figure(i)_option_Gal(kbarre/k)-action_on_Pikbarre}), (\ref{fig:figure(j)_option_Gal(kbarre/k)-action_on_Pikbarre})}
\label{subsec:subsection_2}
In this section we treat the cases where the Galois action on the Petersen diagram is as in (a), (e), (f), (g), (h), (i), (j), (k) in Figure \ref{Fig:Figure_options_for_rho(Gal(kbarre/k))_actions_on_Pikbarre}. The corresponding del Pezzo surfaces are obtained by blowing up $\p^{2}$.

\begin{proposition}
Let $X$ be a del Pezzo surface of degree $5$ such that $\rho(\Gal(\kb/\k)) = \lbrace \id \rbrace $ in $ \Sym_{5} $ as indicated in Figure \ref{fig:figure(0)_option_Gal(kbarre/k)-action_on_Pikbarre}. Then
\begin{enumerate}
\item there is one isomorphism class of such del Pezzo surfaces, obtained by blowing up $\p^{2}$ in four $\k$-rational points $p_{1}, p_{2}, p_{3}, p_{4}$ in general position.
\label{it:item_(1)_Proposition_0_Id}
\item $\Aut_{\k}(X) = \langle G , \widehat{\sigma} \rangle \simeq \Sym_{5} $, where $G$ is the lift of the subgroup $\Aut_{\k}(\p^{2},\lbrace p_{1},p_{2},p_{3},p_{4} \rbrace) $ and $\widehat{\sigma}$ is the lift of the standard quadratic involution $\sigma : [x:y:z] \dashmapsto [yz:xz:xy]$.
\label{it:item_(2)_Proposition_0_Id}
\item $ \rk \, \NS(X)^{\Aut_{\k}(X)} = 1 $ and $ X \longrightarrow \ast $ is an $ \Aut_{\k}(X) $-Mori fibre space.
\label{it_item_(3)_Proposition_0_Id} 
\end{enumerate}
\label{Prop:Proposition_0_Id}
\end{proposition}

\begin{proof}
Since $\rho(\Gal(\kb/\k))=\lbrace\id\rbrace$, all the $(-1)$-curves on the Petersen diagram of $X_{\kb}$ are defined over $\k$. Then the statement is proven analogously to the classical statement over algebraically closed fields and can be found in \cite[Proposition 6.3.7]{bla06} or in \cite[§8.5.4]{dol12} for instance.
\end{proof}

\begin{proposition}
Let $X$ be a del Pezzo surface of degree $5$ such that $\rho(\Gal(\kb/\k)) = \langle (123) \rangle \simeq \Z/3\Z $ in $ \Sym_{5} $ as indicated in Figure \ref{fig:figure(e)_option_Gal(kbarre/k)-action_on_Pikbarre}. Then the following holds:
\begin{enumerate}
\item there exist a point $q=\lbrace q_{1},q_{2},q_{3} \rbrace$
in $\p^{2}$ of degree $3$ with splitting field $L$ such that $ \Gal(L/\k) \simeq \Z/3\Z $ and a $\k$-rational point $r \in \p^{2}(\k) $, such that $X$ is isomorphic to the blow-up of $\p^{2}$ in $r$ and $q$.
\label{it:item_(1)_Proposition_3_Z3}
\item Any two such surfaces are isomorphic if and only if the corresponding field extensions are $\k$-isomorphic.\label{it:item_(2)_Proposition_3_Z3}
\item $\Aut_{\k}(X)=\langle \widehat{\alpha} \rangle \times \langle \widehat{\beta} \rangle \simeq \Z/6\Z $, where $\widehat{\alpha}$ is the lift of an automorphism of $\p^{2}$ of order three and $\widehat{\beta}$ is the lift of a birational quadratic involution of $\p^{2}$ with base-point $q$. 
\label{it:item_(3)_Proposition_3_Z3}
\item $ \rk \, \NS(X)^{\Aut_{\k}(X)} = 2 $, so in particular $ X \longrightarrow \ast $ is not an $ \Aut_{\k}(X) $-Mori fibre space, and the $\Aut_{\k}(X)$-minimal models of $X$ are the del Pezzo surface $Y$ of degree $6$ obtained by blowing up $\p^{2}$ in $q$ and $\mathbb{F}_{0} \simeq \p^{1} \times \p^{1}$.
\label{it:item_(4)_Proposition_3_Z3}
\end{enumerate}
\label{Prop:Proposition_3_Z/3Z}
\end{proposition}

\begin{proof}
(\ref{it:item_(1)_Proposition_3_Z3}) The Petersen diagram $\Pi_{\kb,\rho}$ of $X$ contains only one $(-1)$-curve defined over $\k$ namely $E_{4}$. Its contraction yields a birational morphism $\eta : X \rightarrow Y $ to a rational del Pezzo surface $Y$ of degree $6$ as in \cite[Lemma 4.2]{sz21}, where the set of $(-1)$-curves of $Y$ is the union of two curves $C_{1}$ and $C_{2}$ whose geometric components are respectively $\lbrace D_{13},D_{12},D_{23} \rbrace$ and $\lbrace E_{1},E_{2},E_{3} \rbrace$, on which $\Gal(\kb/\k)$ acts cyclically (see also Figure \ref{Fig:Figure_blow-up_model_Z/3Z_case}). Then there exists a point $q=\lbrace q_{1},q_{2},q_{3} \rbrace$ in $\p^{2}$ of degree $3$ with splitting field $L$ such that $\Gal(L/\k)\simeq\Z/3\Z$ and such that $Y$ is isomorphic to the blow-up of $\p^{2}$ in $q$. Therefore $X$ is isomorphic to the blow-up of $\p^{2}$ in a point $q$ of degree $3$ and a $\k$-rational point $r\in\p^{2}(\k)$ (see Figure \ref{Fig:Figure_blow-up_model_Z/3Z_case}). 
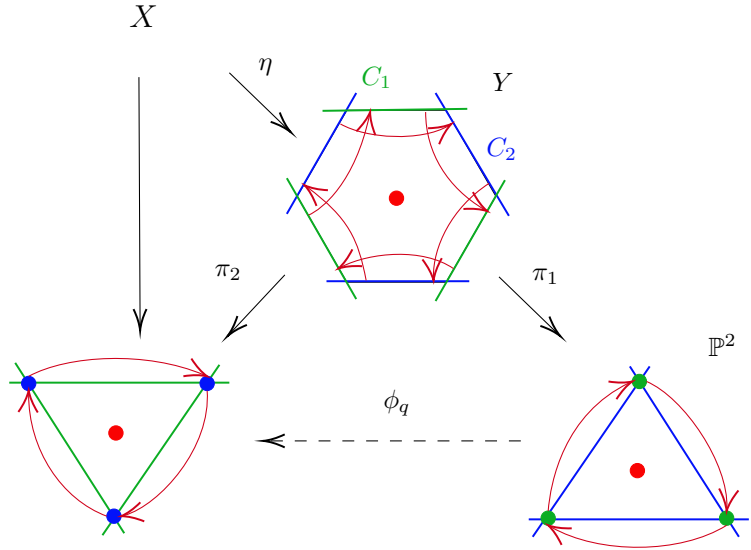
\begin{figure}[h]
\centering
\begin{tikzpicture}[x=0.6pt,y=0.6pt,yscale=-1,xscale=1, scale=0.8, every node/.style={scale=0.8}]

\draw [color={rgb, 255:red, 11; green, 173; blue, 35 }  ,draw opacity=1 ][line width=0.75]    (27.6,221.2) -- (138.6,221.2) ;
\draw [color={rgb, 255:red, 11; green, 173; blue, 35 }  ,draw opacity=1 ][line width=0.75]    (30.6,211.2) -- (85.6,297.2) ;
\draw [color={rgb, 255:red, 11; green, 173; blue, 35 }  ,draw opacity=1 ][line width=0.75]    (133.6,212.2) -- (74.6,298.2) ;
\draw [color={rgb, 255:red, 3; green, 19; blue, 249 }  ,draw opacity=1 ][line width=0.75]    (289.6,290.2) -- (400.6,290.2) ;
\draw [color={rgb, 255:red, 3; green, 19; blue, 249 }  ,draw opacity=1 ][line width=0.75]    (350.6,213) -- (291.6,299) ;
\draw [color={rgb, 255:red, 3; green, 19; blue, 249 }  ,draw opacity=1 ][line width=0.75]    (340.6,213) -- (395.6,299) ;
\draw  [color={rgb, 255:red, 3; green, 19; blue, 249 }  ,draw opacity=1 ][fill={rgb, 255:red, 3; green, 19; blue, 249 }  ,fill opacity=1 ] (33.8,221.6) .. controls (33.8,219.61) and (35.41,218) .. (37.4,218) .. controls (39.39,218) and (41,219.61) .. (41,221.6) .. controls (41,223.59) and (39.39,225.2) .. (37.4,225.2) .. controls (35.41,225.2) and (33.8,223.59) .. (33.8,221.6) -- cycle ;
\draw  [color={rgb, 255:red, 3; green, 19; blue, 249 }  ,draw opacity=1 ][fill={rgb, 255:red, 3; green, 19; blue, 249 }  ,fill opacity=1 ] (123.8,221.6) .. controls (123.8,219.61) and (125.41,218) .. (127.4,218) .. controls (129.39,218) and (131,219.61) .. (131,221.6) .. controls (131,223.59) and (129.39,225.2) .. (127.4,225.2) .. controls (125.41,225.2) and (123.8,223.59) .. (123.8,221.6) -- cycle ;
\draw  [color={rgb, 255:red, 3; green, 19; blue, 249 }  ,draw opacity=1 ][fill={rgb, 255:red, 3; green, 19; blue, 249 }  ,fill opacity=1 ] (76.8,288.6) .. controls (76.8,286.61) and (78.41,285) .. (80.4,285) .. controls (82.39,285) and (84,286.61) .. (84,288.6) .. controls (84,290.59) and (82.39,292.2) .. (80.4,292.2) .. controls (78.41,292.2) and (76.8,290.59) .. (76.8,288.6) -- cycle ;
\draw  [color={rgb, 255:red, 11; green, 173; blue, 35 }  ,draw opacity=1 ][fill={rgb, 255:red, 11; green, 173; blue, 35 }  ,fill opacity=1 ] (341.8,220.6) .. controls (341.8,218.61) and (343.41,217) .. (345.4,217) .. controls (347.39,217) and (349,218.61) .. (349,220.6) .. controls (349,222.59) and (347.39,224.2) .. (345.4,224.2) .. controls (343.41,224.2) and (341.8,222.59) .. (341.8,220.6) -- cycle ;
\draw  [color={rgb, 255:red, 11; green, 173; blue, 35 }  ,draw opacity=1 ][fill={rgb, 255:red, 11; green, 173; blue, 35 }  ,fill opacity=1 ] (385.8,289.6) .. controls (385.8,287.61) and (387.41,286) .. (389.4,286) .. controls (391.39,286) and (393,287.61) .. (393,289.6) .. controls (393,291.59) and (391.39,293.2) .. (389.4,293.2) .. controls (387.41,293.2) and (385.8,291.59) .. (385.8,289.6) -- cycle ;
\draw  [color={rgb, 255:red, 11; green, 173; blue, 35 }  ,draw opacity=1 ][fill={rgb, 255:red, 11; green, 173; blue, 35 }  ,fill opacity=1 ] (295.8,289.6) .. controls (295.8,287.61) and (297.41,286) .. (299.4,286) .. controls (301.39,286) and (303,287.61) .. (303,289.6) .. controls (303,291.59) and (301.39,293.2) .. (299.4,293.2) .. controls (297.41,293.2) and (295.8,291.59) .. (295.8,289.6) -- cycle ;
\draw   (273,127.1) -- (247.95,170.49) -- (197.85,170.49) -- (172.8,127.1) -- (197.85,83.71) -- (247.95,83.71) -- cycle ;
\draw [color={rgb, 255:red, 3; green, 19; blue, 249 }  ,draw opacity=1 ][line width=0.75]    (202.6,75) -- (167.6,136) ;
\draw [color={rgb, 255:red, 3; green, 19; blue, 249 }  ,draw opacity=1 ][line width=0.75]    (259.6,170) -- (187.6,170) ;
\draw [color={rgb, 255:red, 3; green, 19; blue, 249 }  ,draw opacity=1 ][line width=0.75]    (242.6,75) -- (279.6,137) ;
\draw [color={rgb, 255:red, 11; green, 173; blue, 35 }  ,draw opacity=1 ][line width=0.75]    (167.6,118) -- (202.6,179) ;
\draw [color={rgb, 255:red, 11; green, 173; blue, 35 }  ,draw opacity=1 ][line width=0.75]    (185.6,84) -- (258.6,83) ;
\draw [color={rgb, 255:red, 11; green, 173; blue, 35 }  ,draw opacity=1 ][line width=0.75]    (277.6,119) -- (242.6,180) ;
\draw  [color={rgb, 255:red, 249; green, 3; blue, 3 }  ,draw opacity=1 ][fill={rgb, 255:red, 249; green, 3; blue, 3 }  ,fill opacity=1 ] (77.8,246.6) .. controls (77.8,244.61) and (79.41,243) .. (81.4,243) .. controls (83.39,243) and (85,244.61) .. (85,246.6) .. controls (85,248.59) and (83.39,250.2) .. (81.4,250.2) .. controls (79.41,250.2) and (77.8,248.59) .. (77.8,246.6) -- cycle ;
\draw  [color={rgb, 255:red, 249; green, 3; blue, 3 }  ,draw opacity=1 ][fill={rgb, 255:red, 249; green, 3; blue, 3 }  ,fill opacity=1 ] (340.8,265.6) .. controls (340.8,263.61) and (342.41,262) .. (344.4,262) .. controls (346.39,262) and (348,263.61) .. (348,265.6) .. controls (348,267.59) and (346.39,269.2) .. (344.4,269.2) .. controls (342.41,269.2) and (340.8,267.59) .. (340.8,265.6) -- cycle ;
\draw  [color={rgb, 255:red, 249; green, 3; blue, 3 }  ,draw opacity=1 ][fill={rgb, 255:red, 249; green, 3; blue, 3 }  ,fill opacity=1 ] (219.3,128.1) .. controls (219.3,126.11) and (220.91,124.5) .. (222.9,124.5) .. controls (224.89,124.5) and (226.5,126.11) .. (226.5,128.1) .. controls (226.5,130.09) and (224.89,131.7) .. (222.9,131.7) .. controls (220.91,131.7) and (219.3,130.09) .. (219.3,128.1) -- cycle ;
\draw    (166.6,166.8) -- (138.97,196.34) ;
\draw [shift={(137.6,197.8)}, rotate = 313.09] [color={rgb, 255:red, 0; green, 0; blue, 0 }  ][line width=0.75]    (10.93,-3.29) .. controls (6.95,-1.4) and (3.31,-0.3) .. (0,0) .. controls (3.31,0.3) and (6.95,1.4) .. (10.93,3.29)   ;
\draw    (274.6,167.8) -- (304.16,196.41) ;
\draw [shift={(305.6,197.8)}, rotate = 224.06] [color={rgb, 255:red, 0; green, 0; blue, 0 }  ][line width=0.75]    (10.93,-3.29) .. controls (6.95,-1.4) and (3.31,-0.3) .. (0,0) .. controls (3.31,0.3) and (6.95,1.4) .. (10.93,3.29)   ;
\draw    (138.6,64.8) -- (168.16,93.41) ;
\draw [shift={(169.6,94.8)}, rotate = 224.06] [color={rgb, 255:red, 0; green, 0; blue, 0 }  ][line width=0.75]    (10.93,-3.29) .. controls (6.95,-1.4) and (3.31,-0.3) .. (0,0) .. controls (3.31,0.3) and (6.95,1.4) .. (10.93,3.29)   ;
\draw    (93,67) -- (93.59,193) ;
\draw [shift={(93.6,195)}, rotate = 269.73] [color={rgb, 255:red, 0; green, 0; blue, 0 }  ][line width=0.75]    (10.93,-3.29) .. controls (6.95,-1.4) and (3.31,-0.3) .. (0,0) .. controls (3.31,0.3) and (6.95,1.4) .. (10.93,3.29)   ;
\draw  [dash pattern={on 4.5pt off 4.5pt}]  (285.6,250.6) -- (157.6,250.6) ;
\draw [shift={(155.6,250.6)}, rotate = 360] [color={rgb, 255:red, 0; green, 0; blue, 0 }  ][line width=0.75]    (10.93,-3.29) .. controls (6.95,-1.4) and (3.31,-0.3) .. (0,0) .. controls (3.31,0.3) and (6.95,1.4) .. (10.93,3.29)   ;
\draw [color={rgb, 255:red, 208; green, 2; blue, 27 }  ,draw opacity=1 ]   (207.6,169.6) .. controls (203.68,149.02) and (198.8,138.04) .. (177.9,122.55) ;
\draw [shift={(176.6,121.6)}, rotate = 36.03] [color={rgb, 255:red, 208; green, 2; blue, 27 }  ,draw opacity=1 ][line width=0.75]    (10.93,-4.9) .. controls (6.95,-2.3) and (3.31,-0.67) .. (0,0) .. controls (3.31,0.67) and (6.95,2.3) .. (10.93,4.9)   ;
\draw [color={rgb, 255:red, 208; green, 2; blue, 27 }  ,draw opacity=1 ]   (194.6,90.6) .. controls (208.04,98.28) and (235.3,99.51) .. (248.97,91.63) ;
\draw [shift={(250.6,90.6)}, rotate = 145.3] [color={rgb, 255:red, 208; green, 2; blue, 27 }  ,draw opacity=1 ][line width=0.75]    (10.93,-4.9) .. controls (6.95,-2.3) and (3.31,-0.67) .. (0,0) .. controls (3.31,0.67) and (6.95,2.3) .. (10.93,4.9)   ;
\draw [color={rgb, 255:red, 208; green, 2; blue, 27 }  ,draw opacity=1 ]   (269.6,120.6) .. controls (250.4,133.08) and (241.34,152.93) .. (241.53,167.77) ;
\draw [shift={(241.6,169.6)}, rotate = 266.19] [color={rgb, 255:red, 208; green, 2; blue, 27 }  ,draw opacity=1 ][line width=0.75]    (10.93,-4.9) .. controls (6.95,-2.3) and (3.31,-0.67) .. (0,0) .. controls (3.31,0.67) and (6.95,2.3) .. (10.93,4.9)   ;
\draw [color={rgb, 255:red, 208; green, 2; blue, 27 }  ,draw opacity=1 ]   (178.6,136.6) .. controls (185.49,133.65) and (202.09,120.02) .. (209.28,86.16) ;
\draw [shift={(209.6,84.6)}, rotate = 101.31] [color={rgb, 255:red, 208; green, 2; blue, 27 }  ,draw opacity=1 ][line width=0.75]    (10.93,-4.9) .. controls (6.95,-2.3) and (3.31,-0.67) .. (0,0) .. controls (3.31,0.67) and (6.95,2.3) .. (10.93,4.9)   ;
\draw [color={rgb, 255:red, 208; green, 2; blue, 27 }  ,draw opacity=1 ]   (237.6,84.6) .. controls (237.6,105.36) and (248.89,122.92) .. (266.91,133.63) ;
\draw [shift={(268.6,134.6)}, rotate = 209.16] [color={rgb, 255:red, 208; green, 2; blue, 27 }  ,draw opacity=1 ][line width=0.75]    (10.93,-4.9) .. controls (6.95,-2.3) and (3.31,-0.67) .. (0,0) .. controls (3.31,0.67) and (6.95,2.3) .. (10.93,4.9)   ;
\draw [color={rgb, 255:red, 208; green, 2; blue, 27 }  ,draw opacity=1 ]   (251.6,163.6) .. controls (236.24,153.04) and (209.82,155.38) .. (195.35,162.66) ;
\draw [shift={(193.6,163.6)}, rotate = 330.26] [color={rgb, 255:red, 208; green, 2; blue, 27 }  ,draw opacity=1 ][line width=0.75]    (10.93,-4.9) .. controls (6.95,-2.3) and (3.31,-0.67) .. (0,0) .. controls (3.31,0.67) and (6.95,2.3) .. (10.93,4.9)   ;
\draw [color={rgb, 255:red, 208; green, 2; blue, 27 }  ,draw opacity=1 ]   (299.4,286) .. controls (300.58,271.22) and (307.39,237.05) .. (340.28,221.31) ;
\draw [shift={(341.8,220.6)}, rotate = 155.76] [color={rgb, 255:red, 208; green, 2; blue, 27 }  ,draw opacity=1 ][line width=0.75]    (10.93,-4.9) .. controls (6.95,-2.3) and (3.31,-0.67) .. (0,0) .. controls (3.31,0.67) and (6.95,2.3) .. (10.93,4.9)   ;
\draw [color={rgb, 255:red, 208; green, 2; blue, 27 }  ,draw opacity=1 ]   (349,220.6) .. controls (368.11,228.79) and (390.45,262.27) .. (389.52,284.33) ;
\draw [shift={(389.4,286)}, rotate = 275.71] [color={rgb, 255:red, 208; green, 2; blue, 27 }  ,draw opacity=1 ][line width=0.75]    (10.93,-4.9) .. controls (6.95,-2.3) and (3.31,-0.67) .. (0,0) .. controls (3.31,0.67) and (6.95,2.3) .. (10.93,4.9)   ;
\draw [color={rgb, 255:red, 208; green, 2; blue, 27 }  ,draw opacity=1 ]   (389.4,293.2) .. controls (366.08,307.7) and (325.27,307.99) .. (300.87,294.07) ;
\draw [shift={(299.4,293.2)}, rotate = 31.45] [color={rgb, 255:red, 208; green, 2; blue, 27 }  ,draw opacity=1 ][line width=0.75]    (10.93,-4.9) .. controls (6.95,-2.3) and (3.31,-0.67) .. (0,0) .. controls (3.31,0.67) and (6.95,2.3) .. (10.93,4.9)   ;
\draw [color={rgb, 255:red, 208; green, 2; blue, 27 }  ,draw opacity=1 ]   (76.8,288.6) .. controls (54.06,280.17) and (37.28,255.03) .. (37.37,226.92) ;
\draw [shift={(37.4,225.2)}, rotate = 91.59] [color={rgb, 255:red, 208; green, 2; blue, 27 }  ,draw opacity=1 ][line width=0.75]    (10.93,-4.9) .. controls (6.95,-2.3) and (3.31,-0.67) .. (0,0) .. controls (3.31,0.67) and (6.95,2.3) .. (10.93,4.9)   ;
\draw [color={rgb, 255:red, 208; green, 2; blue, 27 }  ,draw opacity=1 ]   (37.4,218) .. controls (57.2,207.22) and (96.97,205.08) .. (125.66,217.24) ;
\draw [shift={(127.4,218)}, rotate = 204.29] [color={rgb, 255:red, 208; green, 2; blue, 27 }  ,draw opacity=1 ][line width=0.75]    (10.93,-4.9) .. controls (6.95,-2.3) and (3.31,-0.67) .. (0,0) .. controls (3.31,0.67) and (6.95,2.3) .. (10.93,4.9)   ;
\draw [color={rgb, 255:red, 208; green, 2; blue, 27 }  ,draw opacity=1 ]   (127.4,225.2) .. controls (129.52,253.96) and (103.52,282.9) .. (85.88,288.13) ;
\draw [shift={(84,288.6)}, rotate = 348.44] [color={rgb, 255:red, 208; green, 2; blue, 27 }  ,draw opacity=1 ][line width=0.75]    (10.93,-4.9) .. controls (6.95,-2.3) and (3.31,-0.67) .. (0,0) .. controls (3.31,0.67) and (6.95,2.3) .. (10.93,4.9)   ;

\draw (379,195.4) node [anchor=north west][inner sep=0.75pt]    {$\mathbb{P}^{2}$};
\draw (270,63.4) node [anchor=north west][inner sep=0.75pt]    {$Y$};
\draw (87,29.4) node [anchor=north west][inner sep=0.75pt]  [font=\large]  {$X$};
\draw (204,60.4) node [anchor=north west][inner sep=0.75pt]  [color={rgb, 255:red, 11; green, 173; blue, 35 }  ,opacity=1 ]  {$C_{1}$};
\draw (267,96.4) node [anchor=north west][inner sep=0.75pt]  [color={rgb, 255:red, 3; green, 19; blue, 249 }  ,opacity=1 ]  {$C_{2}$};
\draw (152,56.4) node [anchor=north west][inner sep=0.75pt]    {$\eta $};
\draw (290,162.4) node [anchor=north west][inner sep=0.75pt]    {$\pi _{1}$};
\draw (130,160.4) node [anchor=north west][inner sep=0.75pt]    {$\pi _{2}$};
\draw (215,223.4) node [anchor=north west][inner sep=0.75pt]    {$\phi _{q}$};

\end{tikzpicture}
\caption{Blow-up model for a del Pezzo surface as in Proposition \ref{Prop:Proposition_3_Z/3Z}.}
\label{Fig:Figure_blow-up_model_Z/3Z_case}
\end{figure}

(\ref{it:item_(2)_Proposition_3_Z3}) The fact that any two del Pezzo surfaces satisfying our hypothesis are isomorphic if and only if the respective splitting fields are $\k$-isomorphic follows from Lemma \ref{lem:Lemme_6.11_&_6.15_article_Julia}.

(\ref{it:item_(3)_Proposition_3_Z3}) The action of $\Aut_{\k}(X)$ on the set of $(-1)$-curves of $X$ gives an isomorphism $\Psi : \Aut_{\k}(X) \overset{\sim}{\longrightarrow} \Aut(\Pi_{L,\rho}) $ by Lemma \ref{lem:Lemma_faithful_action_Aut_k(X)_on_Pi_L_rho}. In this case, since $\Aut(\Pi_{L,\rho}) \simeq \lbrace \sigma \in \Sym_{5} \, \vert \, (123) \circ \sigma = \sigma \circ (123) \rbrace = C_{\Sym_{5}}(\langle (123) \rangle)$, we directly get $\Aut(\Pi_{L,\rho}) \simeq \langle (123),(45) \rangle \simeq \Z/3\Z \times \Z/2\Z$.
We now construct the corresponding geometric actions on $X$. Let us define the map $\phi_{q}:=\pi_{2}\circ\pi_{1}^{-1} \in \Bir_{\k}(\p^{2})$, where $\pi_{1},\pi_{2} : Y \rightarrow \p^{2} $ are respectively the contractions of the curves $C_{1}, C_{2}$ (see Figure \ref{Fig:Figure_blow-up_model_Z/3Z_case}). This is a birational map of degree $2$. By Lemma \ref{lem:Lemme_6.11_&_6.15_article_Julia} we can assume that $\phi_{q}$ fixes the rational point $r$ and that it contracts the line through $q_{i},q_{j}$ onto $q_{k}$, where $\lbrace i,j,k \rbrace = \lbrace 1,2,3 \rbrace$. These conditions imply that $\phi_{q}$ is an involution and that it lifts to an automorphism $\widehat{\beta}$ of X defined over $\k$ inducing a rotation of order two on the hexagon of $Y$ and acting as $(45)$ on the $(-1)$-curves of $X$. Since $\langle \sigma=(123) \rangle = \Z/3\Z \simeq \Gal(L/\k) $, Lemma \ref{lem:Lemme_6.11_&_6.15_article_Julia} guarantees that there exists $\alpha' \in \Aut_{\k}(\p^{2})$ such that $\alpha'(q_{i})=q_{\sigma(i)}$, $i=1,2,3$, and $\alpha'(r)=r$. Then $\alpha'^{3}$ and $\alpha'\circ\phi_{q}\circ\alpha'^{-1}\circ\phi_{q}$ are linear and fix $ r,q_{1},q_{2},q_{3} $, and hence $\alpha'$ is of order $3$ and $\alpha'$ and $\phi_{q}$ commute. The lift $\widehat{\alpha}$ of $\alpha'$ is an automorphism of $X$ commuting with $\widehat{\beta}$; it induces a rotation of order $3$ on the hexagon of $Y$ and it thus acts as $(123)$ on $\Pi_{L,\rho}$. In conclusion, $\Aut_{\k}(X)=\langle\widehat{\alpha}\rangle \times \langle\widehat{\beta}\rangle \simeq \Z/3\Z \times \Z/2\Z $.

(\ref{it:item_(4)_Proposition_3_Z3}) The curve $E_{4}$ is stabilized by $\Aut_{\k}(X)$ (see point \ref{it:item_(3)_Proposition_3_Z3}) and can then be contracted equivariantly. The contraction of $E_{4}$ gives a birational morphism to the rational del Pezzo surface $Y$ of degree $6$ whose $\k$-automorphism group acts transitively on the hexagon, implying that $Y \rightarrow \ast$ is an $\Aut_{\k}(X)$-Mori fibre space.
It follows that $\rk \, \NS(X)^{\Aut_{\k}(X)}=2$. Note that we can also equivariantly contract the $\Aut_{\k}(X)$-orbit $\lbrace D_{24},D_{34},D_{14} \rbrace$ onto the surface $\mathbb{F}_{0} \simeq \p^{1} \times \p^{1}$, which finally yields the claim.
\end{proof}

\begin{example}
We construct a del Pezzo surface as in Proposition \ref{Prop:Proposition_3_Z/3Z}. From \cite[Example 4.4]{sz21}, let $\k = \mathbb{F}_{2}$ and $L/\k$ be the splitting field of $P(X)=X^{3}+X+1$, i.e.\ $\vert L \vert = 8$. Then  $ 3=\vert L:\k \vert=\vert \Gal(L/\k) \vert$ and $ \sigma : a \mapsto a^{2} $ generates $\Gal(L/\k)$ (see for instance \cite[Theorem 6.5]{mor96}). If $\zeta$ is a root of $P(X)$, then $\sigma(\zeta^{4})=\zeta$ and hence the point $\lbrace [1:\zeta:\zeta^{4}],[1:\zeta^{2}:\zeta],[1:\zeta^{4}:\zeta^{2}] \rbrace$ is of degree $3$, its geometric components are not collinear and they are cyclically permuted by $\sigma$. We take $ r=[1:1:1] \in \p^{2}(\k) $ as a $\k$-rational point.
\label{ex:example_Z/3Z}
\end{example}

\begin{proposition}
Let $X$ be a del Pezzo surface of degree $5$ such that $ \rho(\Gal(\kb/\k)) = \langle (123),(12) \rangle \simeq \Sym_{3} $ in $ \Sym_{5} $ as indicated in Figure \ref{fig:figure(j)_option_Gal(kbarre/k)-action_on_Pikbarre}. Then the following holds: 
\begin{enumerate}
\item there exist a point $ q = \lbrace q_{1},q_{2},q_{3} \rbrace $ in $ \mathbb{P}^{2} $ of degree $3$ with splitting field $L$ such that $ \Gal(L/\k) \simeq \Sym_{3} $ and a rational point $ r \in \mathbb{P}^{2}(\k) $ such that $X$ is isomorphic to the blow-up of $ \mathbb{P}^{2} $ in $q$ and $r$.\label{it:item_(1)_Proposition_8_Sym3_1}
\item Any two such surfaces are isomorphic if and only if the corresponding field extensions are $\k$-isomorphic.\label{it:item_(2)_Proposition_8_Sym3_1}
\item $ \Aut_{\k}(X)=\langle \widehat{\alpha} \rangle \simeq \mathbb{Z}/2\mathbb{Z} $, where $\widehat{\alpha}$ is the lift of a birational quadratic involution $\phi_{q} $ with base-point $q$.\label{it:item_(3)_Proposition_8_Sym3_1}
\item $ \rk \, \NS(X)^{\Aut_{\k}(X)} = 2 $, so in particular $ X \longrightarrow \ast $ is not an $ \Aut_{\k}(X) $-Mori fibre space, and the $\Aut_{\k}(X)$-minimal models of $X$ are the rational del Pezzo surface $Y$ of degree $6$ obtained by blowing up $\mathbb{P}^{2}$ in $q$ and $\mathbb{F}_{0} \simeq \mathbb{P}^{1} \times \mathbb{P}^{1}$.\label{it:item_(4)_Proposition_8_Sym3_1}
\end{enumerate}
\label{Prop:Proposition_8_Sym3_case_1}
\end{proposition}

\begin{proof}
(\ref{it:item_(1)_Proposition_8_Sym3_1}) The incidence diagram $ \Pi_{\kb,\rho} $ of $X$ contains only one $(-1)$-curve that is defined over $ \k $ namely $E_{4}$. Its contraction yields a birational morphism $ \eta : X \rightarrow Y $ onto a rational del Pezzo surface $Y$ of degree $6$ as in \cite[Lemma 4.3]{sz21}. Then there exists a point $ q=\lbrace q_{1},q_{2},q_{3} \rbrace $ in $ \mathbb{P}^{2} $ of degree $3$ with splitting field $L$ such that $ \Gal(L/\k) \simeq \Sym_{3} $ and whose geometric components are in general position in $\p^{2}_{L}$, such that $Y$ is isomorphic to the blow-up of $ \mathbb{P}^{2} $ in $q$. Finally, $X$ is isomorphic to the blow-up of $\p^{2}$ in a point $q$ of degree $3$ and in a $\k$-rational point $r$.

(\ref{it:item_(2)_Proposition_8_Sym3_1}) The fact that any such two del Pezzo surfaces are isomorphic if and only if the respective splitting fields are $\k$-isomorphic follows from Lemma \ref{lem:Lemme_6.11_&_6.15_article_Julia}.

(\ref{it:item_(3)_Proposition_8_Sym3_1}) By Lemma \ref{lem:Lemma_faithful_action_Aut_k(X)_on_Pi_L_rho}, the action of $ \Aut_{\k}(X) $ on the set of $(-1)$-curves of $X$ yields an isomorphism $ \Psi : \Aut_{\k}(X) \overset{\sim}{\longrightarrow} \Aut(\Pi_{L,\rho}) $. Since $\Aut(\Pi_{L,\rho}) \simeq \lbrace \sigma \in \Sym_{5} \, \vert \, (123) \circ \sigma = \sigma \circ (123) \, \text{and} \, (12) \circ \sigma = \sigma \circ (12) \rbrace = C_{\Sym_{5}}(\langle (123),(12) \rangle)$, we get $\Aut(\Pi_{L,\rho}) \simeq \langle (45) \rangle \simeq \Z/2\Z$.
To finish, we construct the corresponding geometric action. From Proposition \ref{Prop:Proposition_3_Z/3Z}(\ref{it:item_(3)_Proposition_3_Z3}) there exists a birational quadratic involution $\phi_{q} \in \Bir_{\k}(\p^{2})$ that lifts to an automorphism $\widehat{\alpha}$ of $X$ inducing the same action as $(45)$ on the $(-1)$-curves of $X$. In conclusion, $ \Aut_{\k}(X) = \langle \widehat{\alpha} \rangle \simeq \mathbb{Z}/2\mathbb{Z} $.

(\ref{it:item_(4)_Proposition_8_Sym3_1}) The proof is analogous to the one of  Proposition \ref{Prop:Proposition_3_Z/3Z}(\ref{it:item_(4)_Proposition_3_Z3}).

\end{proof}

\begin{example}
We construct a del Pezzo surface as in Proposition \ref{Prop:Proposition_8_Sym3_case_1}. Indeed, from \cite[Example 4.5]{sz21}, let $ \k := \mathbb{Q} $, let $ \zeta := \sqrt[3]{2} $ and $ \omega := e^{\frac{2 \pi i}{3}} $. Then $ L := \k(\zeta,\omega) $ is a Galois extension of $\k$ of degree $6$ and $ \Gal(L/\k) \simeq \Sym_{3} $ is the group of $ \k $-isomorphisms $ [ (\zeta,\omega) \overset{\sigma_{1}}{\longmapsto} (\zeta,\omega), (\zeta,\omega) \overset{\sigma_{2}}{\longmapsto} (\omega \zeta,\omega), (\zeta,\omega) \overset{\sigma_{3}}{\longmapsto} (\zeta,\omega^{2}), (\zeta,\omega) \overset{\sigma_{4}}{\longmapsto} (\omega \zeta,\omega^{2}), (\zeta,\omega) \overset{\sigma_{5}}{\longmapsto} (\omega^{2} \zeta,\omega), (\zeta,\omega) \overset{\sigma_{6}}{\longmapsto} (\omega^{2} \zeta,\omega^{2}) ] $ (see \cite[Example 2.21]{mor96}). The point $ q = \displaystyle\lbrace [\zeta:\zeta^{2}:1] , [\omega\zeta:\omega^{2}\zeta^{2}:1] , [\omega^{2}\zeta:\omega\zeta^{2}:1] \rbrace $ is of degree $3$, its geometric components are not collinear and any non-trivial element of $ \Gal(L/\k) $ permutes them non-trivially. We take $ r=[1:1:1] \in \p^{2}(\k) $ as a $\k$-rational point.
\label{ex:example_Sym3_case_1}
\end{example}

\begin{remark}
We note that in Proposition \ref{Prop:Proposition_3_Z/3Z} (\ref{it:item_(1)_Proposition_3_Z3}) (resp. \ref{Prop:Proposition_8_Sym3_case_1} (\ref{it:item_(1)_Proposition_8_Sym3_1})), we could also have contracted the curve of $X$ whose geometric components are $\lbrace D_{14},D_{24},D_{34} \rbrace$ onto a rational del Pezzo surface $Z$ of degree $8$ with $\rk \, \NS(Z) = 2$, namely $Z \simeq \mathbb{F}_{0} \simeq \p^{1} \times \p^{1}$ by \cite[Lemma 3.2]{sz21}; that is $X$ is also obtained by blowing up $\mathbb{F}_{0}$ in a point of degree $3$ whose splitting field $L$ is such that $\Gal(L/\k) \simeq \Z/3\Z$ (resp. $\Gal(L/\k) \simeq \Sym_{3}$).
\end{remark}

\begin{proposition}
Let $X$ be a del Pezzo surface of degree $5$ such that $\rho(\Gal(\kb/\k))$ is one of the subgroups of $\Sym_{5}$ occuring in Table \ref{Tab:Table1}, and corresponding to Figures \ref{fig:figure(d)_option_Gal(kbarre/k)-action_on_Pikbarre}, \ref{fig:figure(f)_option_Gal(kbarre/k)-action_on_Pikbarre}, \ref{fig:figure(h)_option_Gal(kbarre/k)-action_on_Pikbarre}, \ref{fig:figure(g)_option_Gal(kbarre/k)-action_on_Pikbarre}, \ref{fig:figure(i)_option_Gal(kbarre/k)-action_on_Pikbarre}, respectively.\\

Then the following holds:
\begin{enumerate}
\item there exists a point $p = \lbrace p_{1},p_{2},p_{3},p_{4} \rbrace$ in $\p^{2}$ of degree $4$ with splitting field $L$ such that $\Gal(L/\k) \simeq \Gamma$, where $\Gamma$ is given in Table \ref{Tab:Table1}, and such that $X$ is isomorphic to the blow-up of $\p^{2}$ in $p$.\label{it:item_(1)_Proposition_regroupement}
\item Any two such surfaces are isomorphic if and only if the corresponding residue fields are $\k$-isomorphic, which is equivalent to having $\k$-isomorphic splitting fields except when $\Gal(L/\k) \simeq \D_{4}$.\label{it:item_(2)_Proposition_regroupement}
\item $\Aut_{\k}(X) \simeq G$, where $G$ is described in Table \ref{Tab:Table1}, and it is generated by the lifts of some automorphisms of $\p^{2}$.\label{it:item_(3)_Proposition_regroupement}
\item $\rk \, \NS(X)^{\Aut_{\k}(X)} = 2$, so in particular $X \rightarrow \ast$ is not an $\Aut_{\k}(X)$-Mori fibre space, but $X$ carries an $\Aut_{\k}(X)$-Mori fibre space structure $X \rightarrow \p^{1}$ given by the linear system of conics passing through $p$.\label{it:item_(4)_Proposition_regroupement}
\end{enumerate}
\begin{table}[!h]
\begin{center}
\renewcommand*{\arraystretch}{1.9}
\begin{tabular}{|>{\centering\arraybackslash}m{1cm}|>{\centering\arraybackslash}m{3.5cm}|>{\centering\arraybackslash}m{3.2cm}|>{\centering\arraybackslash}m{4.2cm}|} 
\hline  
 & $\rho(\Gal(\kb/\k)) \subset \Sym_{5}$ & $\Gamma \simeq \Gal(L/\k)$ & $G=\Aut_{\k}(X)$ \\ 
\hline 
\ref{fig:figure(d)_option_Gal(kbarre/k)-action_on_Pikbarre} & $\langle (12)(34) \rangle \times \langle (13)(24) \rangle$ & $\Z/2\Z \times \Z/2\Z$ & $\langle \widehat{\alpha} \rangle \times \langle \widehat{\beta} \rangle \simeq \Z/2\Z \times \Z/2\Z$, where $\widehat{\alpha}$ and $\widehat{\beta}$ are the lifts of involutions of $\p^{2}$ \\  
\hline 
\ref{fig:figure(f)_option_Gal(kbarre/k)-action_on_Pikbarre} & $\langle (1234) \rangle$ & $\Z/4\Z$ & $\langle \widehat{\alpha} \rangle \simeq \Z/4\Z$, where $\widehat{\alpha}$ is the lift of an automorphism of $\p^{2}$ of order four \\ 
\hline
\ref{fig:figure(h)_option_Gal(kbarre/k)-action_on_Pikbarre} & $\langle (1234),(13) \rangle$ & $\D_{4}$ & $\langle \widehat{\alpha} \rangle \simeq \Z/2\Z$, where $\widehat{\alpha}$ is the lift of an involution of $\p^{2}$ \\
\hline 
\ref{fig:figure(g)_option_Gal(kbarre/k)-action_on_Pikbarre} & $\langle (12)(34),(123) \rangle$ & $\mathcal{A}_{4}$ & $\lbrace \id \rbrace$ \\
\hline 
\ref{fig:figure(i)_option_Gal(kbarre/k)-action_on_Pikbarre} & $\langle (1234),(12) \rangle$ & $\Sym_{4}$ & $\lbrace \id \rbrace$ \\
\hline
\end{tabular}
\end{center}
\caption{Summary of the results of Proposition \ref{Prop:Proposition_essai_regroupement}.}
\label{Tab:Table1}
\end{table} 
\label{Prop:Proposition_essai_regroupement}
\end{proposition}

\begin{proof}
Let us give the proof for a del Pezzo surface $X$ as in Figure \ref{fig:figure(f)_option_Gal(kbarre/k)-action_on_Pikbarre}.\\
(\ref{it:item_(1)_Proposition_regroupement}) The union $E=E_{1} \cup E_{2} \cup E_{3} \cup E_{4}$ makes up a $\Gal(\kb/\k)$-orbit and hence defines an irreducible $\k$-curve. Its contraction yields a birational morphism $\eta : X \longrightarrow \p^{2}$ onto a point $p=\lbrace p_{1},p_{2},p_{3},p_{4} \rbrace$ of degree $4$ whose geometric components are in general position in $\p^2_{\kb}$ (see Figure \ref{Fig:Figure_blow-up_model_Z/4Z_case}).
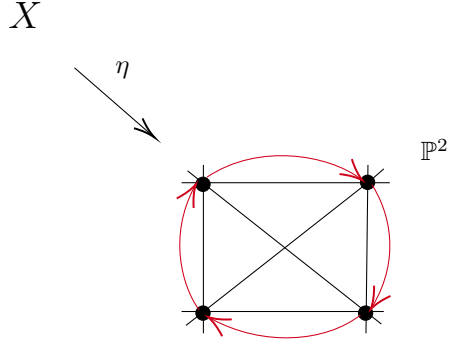
\begin{figure}[]
\centering
\begin{tikzpicture}[x=0.6pt,y=0.6pt,yscale=-1,xscale=1, scale=0.9, every node/.style={scale=0.8}]

\draw    (269.6,148.8) -- (269.6,233.8) ;
\draw    (352.6,148.8) -- (351.6,233.8) ;
\draw    (258.6,157.8) -- (363.6,157.8) ;
\draw    (258.6,222.8) -- (360.6,222.8) ;
\draw  [fill={rgb, 255:red, 0; green, 0; blue, 0 }  ,fill opacity=1 ] (265.8,158.6) .. controls (265.8,156.61) and (267.41,155) .. (269.4,155) .. controls (271.39,155) and (273,156.61) .. (273,158.6) .. controls (273,160.59) and (271.39,162.2) .. (269.4,162.2) .. controls (267.41,162.2) and (265.8,160.59) .. (265.8,158.6) -- cycle ;
\draw  [fill={rgb, 255:red, 0; green, 0; blue, 0 }  ,fill opacity=1 ] (348.8,157.6) .. controls (348.8,155.61) and (350.41,154) .. (352.4,154) .. controls (354.39,154) and (356,155.61) .. (356,157.6) .. controls (356,159.59) and (354.39,161.2) .. (352.4,161.2) .. controls (350.41,161.2) and (348.8,159.59) .. (348.8,157.6) -- cycle ;
\draw  [fill={rgb, 255:red, 0; green, 0; blue, 0 }  ,fill opacity=1 ] (265.8,223.6) .. controls (265.8,221.61) and (267.41,220) .. (269.4,220) .. controls (271.39,220) and (273,221.61) .. (273,223.6) .. controls (273,225.59) and (271.39,227.2) .. (269.4,227.2) .. controls (267.41,227.2) and (265.8,225.59) .. (265.8,223.6) -- cycle ;
\draw  [fill={rgb, 255:red, 0; green, 0; blue, 0 }  ,fill opacity=1 ] (347.8,223.6) .. controls (347.8,221.61) and (349.41,220) .. (351.4,220) .. controls (353.39,220) and (355,221.61) .. (355,223.6) .. controls (355,225.59) and (353.39,227.2) .. (351.4,227.2) .. controls (349.41,227.2) and (347.8,225.59) .. (347.8,223.6) -- cycle ;
\draw    (261.6,151.8) -- (359.35,229.12) ;
\draw    (360.6,150.8) -- (353.77,156.75) -- (260.84,229.83) ;
\draw    (204.6,99.8) -- (243.49,133.69) ;
\draw [shift={(245,135)}, rotate = 221.07] [color={rgb, 255:red, 0; green, 0; blue, 0 }  ][line width=0.75]    (10.93,-3.29) .. controls (6.95,-1.4) and (3.31,-0.3) .. (0,0) .. controls (3.31,0.3) and (6.95,1.4) .. (10.93,3.29)   ;
\draw [color={rgb, 255:red, 208; green, 2; blue, 27 }  ,draw opacity=1 ]   (267.4,222) .. controls (258.78,210.04) and (251.88,184.83) .. (264.97,160.11) ;
\draw [shift={(265.8,158.6)}, rotate = 119.4] [color={rgb, 255:red, 208; green, 2; blue, 27 }  ,draw opacity=1 ][line width=0.75]    (10.93,-4.9) .. controls (6.95,-2.3) and (3.31,-0.67) .. (0,0) .. controls (3.31,0.67) and (6.95,2.3) .. (10.93,4.9)   ;
\draw [color={rgb, 255:red, 208; green, 2; blue, 27 }  ,draw opacity=1 ]   (271.4,155) .. controls (297.78,137.35) and (332.82,143.97) .. (348.41,155.51) ;
\draw [shift={(349.8,156.6)}, rotate = 219.5] [color={rgb, 255:red, 208; green, 2; blue, 27 }  ,draw opacity=1 ][line width=0.75]    (10.93,-4.9) .. controls (6.95,-2.3) and (3.31,-0.67) .. (0,0) .. controls (3.31,0.67) and (6.95,2.3) .. (10.93,4.9)   ;
\draw [color={rgb, 255:red, 208; green, 2; blue, 27 }  ,draw opacity=1 ]   (354.6,158.8) .. controls (363.42,172.52) and (369.36,195.84) .. (354.91,220.11) ;
\draw [shift={(354,221.6)}, rotate = 302.17] [color={rgb, 255:red, 208; green, 2; blue, 27 }  ,draw opacity=1 ][line width=0.75]    (10.93,-4.9) .. controls (6.95,-2.3) and (3.31,-0.67) .. (0,0) .. controls (3.31,0.67) and (6.95,2.3) .. (10.93,4.9)   ;
\draw [color={rgb, 255:red, 208; green, 2; blue, 27 }  ,draw opacity=1 ]   (348.8,224.6) .. controls (331.16,238.32) and (299.88,240.71) .. (273.6,226.49) ;
\draw [shift={(272,225.6)}, rotate = 29.74] [color={rgb, 255:red, 208; green, 2; blue, 27 }  ,draw opacity=1 ][line width=0.75]    (10.93,-4.9) .. controls (6.95,-2.3) and (3.31,-0.67) .. (0,0) .. controls (3.31,0.67) and (6.95,2.3) .. (10.93,4.9)   ;

\draw (378,134.4) node [anchor=north west][inner sep=0.75pt]    {$\mathbb{P}^{2}$};
\draw (170,65.4) node [anchor=north west][inner sep=0.75pt]  [font=\Large]  {$X$};
\draw (224,95.4) node [anchor=north west][inner sep=0.75pt]    {$\eta $};

\end{tikzpicture}
\caption{Blow-up model for a del Pezzo surface as in Proposition \ref{Prop:Proposition_essai_regroupement},(\ref{fig:figure(f)_option_Gal(kbarre/k)-action_on_Pikbarre}).}
\label{Fig:Figure_blow-up_model_Z/4Z_case}
\end{figure}
Let $L$ be any splitting field of $p$. Since $\rho(\Gal(\kb/\k))=\langle(1234)\rangle =\Z/4\Z $, the action of $\Gal(L/\k)$ on $\lbrace p_{1},p_{2},p_{3},p_{4} \rbrace $ induces an exact sequence $ 1 \longrightarrow H \longrightarrow \Gal(L/\k) \longrightarrow \Z/4\Z \longrightarrow 1 $. The field $ L':=\lbrace a \in L \mid h(a)=a \,\, \forall \, h \in H\rbrace $ is an intermediate field between $L$ and $\k$, over which $p_{1},p_{2},p_{3},p_{4} $ are rational. The minimality of $L$ implies that $L'=L$ and hence $H=\Gal(L/L')=\lbrace \text{id} \rbrace$ \cite[Corollary 2.10]{mor96}, and so $\Gal(L/\k) \simeq \Z/4\Z $.

(\ref{it:item_(2)_Proposition_regroupement}) The claim follows from Lemma \ref{lem:Lemme_6.11_&_6.15_article_Julia},(\ref{it:item_(a)_Lemme_6.15_Julia}) and we can even assume that the blown up point $p$ is of the form $\lbrace [1:a_{1}:a_{1}^{2}],[1:a_{2}:a_{2}^{2}],[1:a_{3}:a_{3}^{2}],[1:a_{4}:a_{4}^{2}] \rbrace$, with $a_{i} \in L$ forming an orbit $\lbrace a_{1},a_{2},a_{3},a_{4} \rbrace \subset L $ of size four under the Galois action.

(\ref{it:item_(3)_Proposition_regroupement}) By Lemma \ref{lem:Lemma_faithful_action_Aut_k(X)_on_Pi_L_rho} the action of $\Aut_{\k}(X)$ on the set of $(-1)$-curves  of $X$ yields an isomorphism $\Psi : \Aut_{\k}(X) \overset{\sim}{\longrightarrow} \Aut(\Pi_{L,\rho})$. Since $\Aut(\Pi_{L,\rho}) \simeq \lbrace \sigma \in \Sym_{5} \, \vert \, (1234) \circ \sigma = \sigma \circ (1234) \rbrace = C_{\Sym_{5}}(\langle (1234) \rangle)$, we then get $\Aut(\Pi_{L,\rho}) \simeq \langle (1234) \rangle \simeq \Z/4\Z$. To finish we construct explicitly the corresponding geometric action: if $\langle\sigma=(1234)\rangle=\Z/4\Z \simeq \Gal(L/\k)$, Lemma \ref{lem:Lemme_6.11_&_6.15_article_Julia} guarantees that there exists $\alpha \in \Aut_{\k}(\p^{2})$ such that $\alpha(p_{i})=p_{\sigma(i)}$, for $i=1,2,3,4$. Then $\alpha^{4}$ is linear and fixes $p_{1},p_{2},p_{3},p_{4}$, and hence $\alpha$ is an automorphism of order four. The lift $\widehat{\alpha}$ of $\alpha$ is an automorphism of $X$ defined over $\k$ that acts as $(1234)$ on $\Pi_{L}$. In conclusion, $\Aut_{\k}(X) = \langle\widehat{\alpha}\rangle \simeq \Z/4\Z $.

(\ref{it:item_(4)_Proposition_regroupement}) The set $\{ E_{1},E_{2},E_{3},E_{4} \}$ forms one $\Aut_{\k}(X)$-orbit and hence can be equivariantly contracted onto $\p^{2}$, giving $\rk \,  \NS(X)^{\Aut_{\k}(X)} = 2$. Moreover, the conic bundle induced by $2H-E_{1}-E_{2}-E_{3}-E_{4}$ on $X$ (see Figure \ref{fig:fig(d)_conic_bundle_point_of_view}), where $H$ denotes the pullback of a general line of $\p^{2}$, is preserved by the action of $\Aut_{\k}(X)$. Since $\rk \NS(X)^{\Aut_{\k}(X)} = 2$, the corresponding conic bundle structure $X \rightarrow \p^{1}$ given by the linear system of conics through $p$ on $\p^{2}$ is an $\Aut_{\k}(X)$-Mori fibre space.

\end{proof}

\begin{examples}
\begin{enumerate}
\item We construct a del Pezzo surface as in Proposition \ref{Prop:Proposition_essai_regroupement},(\ref{fig:figure(d)_option_Gal(kbarre/k)-action_on_Pikbarre}): let $P(X)=X^{4}+aX^{2}+b$ be an irreducible polynomial over $\Q$. We denote $\pm \alpha$ and $\pm \beta$ its roots in a splitting field $L$. Then  $\Gal(L/\Q)$ is isomorphic to $\Z/4\Z$ or to $\Z/2\Z \times \Z/2\Z$ or to $\D_{4}$. Note that $ \alpha^{2}-\beta^{2} \notin \Q$ otherwise $P$ would be reducible, and we have the following characterization: $\Gal(L/\Q)=\langle \tau \circ \sigma \rangle \times \langle \sigma^{2} \rangle \simeq \Z/2\Z \times \Z/2\Z \Longleftrightarrow \alpha \beta \in \Q $, where $ \tau : \alpha \mapsto -\alpha, \beta \mapsto \beta $ and $ \sigma : \alpha \mapsto \beta, \beta \mapsto -\alpha $. In the latter case,$$p=\lbrace [1:\alpha:\alpha^{2}],[1:-\alpha:\alpha^{2}],[1:\beta:\beta^{2}],[1:-\beta:\beta^{2}] \rbrace$$is a point of degree $4$ in $\p^{2}$ whose geometric components are defined over $L$ and form an orbit of size four under the $\Gal(L/\Q)$-action. One can consider $P(X)=X^{4}-6X^{2}+4$ over $\Q$ with roots $ \pm \alpha = \pm \sqrt{3+\sqrt{5}} $ and $ \pm \beta = \pm \sqrt{3-\sqrt{5}} $ for a specific example.

\item We construct a del Pezzo surface as in Proposition \ref{Prop:Proposition_essai_regroupement},(\ref{fig:figure(f)_option_Gal(kbarre/k)-action_on_Pikbarre}). Indeed, let $\k = \mathbb{F}_{2}$ and $L/\k$ be the splitting field of $P(X)=X^{4}+X+1$, i.e.\ $\vert L \vert = 16$. Then $4=\vert L:\k \vert=\vert \Gal(L/\k) \vert$ and $\sigma : s \mapsto s^{2} $ generates $\Gal(L/\k)$. If $\zeta$ is a root of $P(X)$, then $\sigma(\zeta^{8})=\zeta$ and hence the point $\lbrace [1:\zeta:\zeta^{2}],[1:\zeta^{2}:\zeta^{4}],[1:\zeta^{4}:\zeta^{8}],[1:\zeta^{8}:~\zeta^{16}] \rbrace$ is of degree $4$, its geometric components are not collinear and they are cyclically permuted by $\sigma$.

\item We construct a del Pezzo surface as in Proposition \ref{Prop:Proposition_essai_regroupement},(\ref{fig:figure(h)_option_Gal(kbarre/k)-action_on_Pikbarre}): let $\k=\Q$ and let $L$ be a splitting field of the polynomial $X^{4}-6 \in \k[X]$. Then $L=\k(\sqrt[4]{6},i)$ and $ \vert L:\k \vert=8$. The Galois group $\Gamma:=\Gal(L/\k)$ contains the $\Q$-isomorphisms $r$ and $s$ being such that $r(\sqrt[4]{6})=i\sqrt[4]{6}$, $r(i)=i$, $s(\sqrt[4]{6})= \sqrt[4]{6}$ and $s(i)=-i$, with $\vert \Gal(L/\k) \vert = \vert L:\k \vert = 8$ because $L/\k$ is Galois. Thus, $\Gal(L/\k) = \langle r,s \, \vert \, r^{4}=s^{2}=\id \, , \, s \circ r \circ s=r^{-1} \rangle \simeq \D_{4} $ and $ p=\lbrace [1:\sqrt[4]{6}:\sqrt{6}],[1:-\sqrt[4]{6}:\sqrt{6}],[1:i\sqrt[4]{6}:-\sqrt{6}],[1:-i\sqrt[4]{6}:-\sqrt{6}] \rbrace$ is a point of degree $4$ in $\p^{2}$ whose geometric components are non-trivially permuted by any non-trivial element in $\Gamma$.

\item We construct a del Pezzo surface as in Proposition \ref{Prop:Proposition_essai_regroupement},(\ref{fig:figure(g)_option_Gal(kbarre/k)-action_on_Pikbarre}). Indeed, let $\k=\Q$, let $P(X)=X^{4}+8X+12 \in \k[X]$ and $L$ be a splitting field of $P$ over $\k$. Then $P$ is separable and $\text{disc}(P)=576^{2} $ so $\Gamma:=\Gal(L/\k)$ is a subgroup of $\mathcal{A}_{4}$, and $P$ is irreducible so $\Gamma$ acts transitively on the set of roots of $P$. We have $P(X) \equiv X^{4}+3X+7=(X+1)(X^{3}+4X^{2}+X+2) \modulo \, 5$, where $X^{3}+4X^{2}+X+2$ is irreducible over $\F_{5}$ with Galois group $\Gal(\F_{5^{3}}/\F_{5}) \simeq \Z/3\Z $ generated by $x \mapsto x^{5}$. It follows that $\Gamma$ contains $\Z/3\Z$ as a subgroup. Moreover, $P(X) \equiv (X^{2}+4X+7)(X^{2}+13X+9) \modulo \, 17$, where $X^{2}+4X+7$ is irreducible over $\F_{17}$ and its Galois group is $\Gal(\F_{17^{2}}/\F_{17}) \simeq \Z/2\Z $ generated by $x \mapsto x^{17} $. Then $\Gamma$ contains $\Z/2\Z$ as a subgroup as well. In conclusion, $\Gamma$ is a subgroup of $\mathcal{A}_{4}$ acting transitively on the roots of $P$, and both $2$ and $3$ divide $\vert \Gamma \vert$. This implies $\Gamma \simeq \mathcal{A}_{4}$. Let us denote by $a_{1},a_{2},a_{3},a_{4}$ the roots of $P$ in $L$. Then $p=\lbrace [1:a_{1}:a_{1}^{2}],[1:a_{2}:a_{2}^{2}],[1:a_{3}:a_{3}^{2}],[1:a_{4}:a_{4}^{2}] \rbrace$ defines a point of degree $4$ whose geometric components are non-trivially permuted by any non-trivial element of $\Gamma$.

\end{enumerate}

\end{examples}

\begin{remark}
Note that if $p, p'$ are two points of degree $4$ in general position on $\p^{2}$, then $p$ and $p'$ are projectively equivalent, i.e. there exists $\alpha \in \PGL_{3}(\k)$ with $\alpha(p) = p'$, exactly if $p$ and $p'$ have $\k$-isomorphic residue fields thanks to Lemma \ref{lem:Lemme_6.11_&_6.15_article_Julia}. This is equivalent to having $\k$-isomorphic splitting fields, say $L/\k$, $L'/\k$, except when $\Gal(L/\k) \simeq \text{D}_{4}$ (see Proposition \ref{Prop:Proposition_essai_regroupement},\ref{fig:figure(h)_option_Gal(kbarre/k)-action_on_Pikbarre},(\ref{it:item_(2)_Proposition_regroupement})).
\end{remark}

\subsection{Del Pezzo surfaces in Figures (\ref{fig:figure(k)_option_Gal(kbarre/k)-action_on_Pikbarre}), (\ref{fig:figure(l)_option_Gal(kbarre/k)-action_on_Pikbarre}), (\ref{fig:figure(m)_option_Gal(kbarre/k)-action_on_Pikbarre})}
\label{subsec:subsection_3}

In this section we treat the cases where the Galois action on the Petersen diagram is as in (l), (m), (n) in Figure \ref{Fig:Figure_options_for_rho(Gal(kbarre/k))_actions_on_Pikbarre}. The corresponding del Pezzo surfaces of degree $5$ are obtained by blowing up the quadric $\Ql^{L}$ or a rational del Pezzo surface of degree $6$.

\begin{proposition}
Let $X$ be a del Pezzo surface of degree $5$ such that $ \rho(\Gal(\kb/\k)) = \langle (123)(45) \rangle \simeq \mathbb{Z}/6\mathbb{Z} $ in $ \text{Sym}_{5} $ as indicated in Figure \ref{fig:figure(k)_option_Gal(kbarre/k)-action_on_Pikbarre}. Then the following holds.
\begin{enumerate}
\item There exist a quadratic extension $ L/\k $ and a birational morphism $ \eta : X \longrightarrow \mathcal{Q}^{L} $ contracting an irreducible curve onto a point $ p=\lbrace p_{1},p_{2},p_{3} \rbrace $ of degree $3$ whose geometric components are contained in pairwise distinct rulings of $ \mathcal{Q}^{L}_{L} $, and with splitting field $ F $ such that $ \Gal(F/\k) \simeq \Gal(FL/L) \simeq \mathbb{Z}/3\mathbb{Z} $, where $FL$ is a splitting field of $p$ over $L$.\label{it:item_(1)_Proposition_7_Z/6Z}
\item Any two such del Pezzo surfaces are isomorphic if and only if the corresponding field extensions of degree two and three are $\k$-isomorphic.\label{it:item_(2)_Proposition_7_Z/6Z}
\item $ \Aut_{\k}(X)=\langle \widehat{\alpha} \rangle \times \langle \widehat{\Phi_{q}} \rangle \simeq \mathbb{Z}/3\mathbb{Z} \times \mathbb{Z}/2\mathbb{Z} $, where $\widehat{\alpha}$ is the lift of an automorphism of $\p^{2}_{L}$ of order three and $\widehat{\Phi_{q}}$ is the lift of a quadratic birational involution of $\p^{2}_{L}$ with base-point $q$ of degree $3$.\label{it:item_(3)_Proposition_7_Z/6Z}
\item $ \rk \, \NS(X)^{\Aut_{\k}(X)} = 2 $, so in particular $ X \rightarrow \ast $ is not an $ \Aut_{\k}(X) $-Mori fibre space, and the $\Aut_{\k}(X)$-minimal models of $X$ are the quadric $\Ql^{L}$ and the rational del Pezzo surface $Y$ of degree $6$ such that $Y_{L}$ is isomorphic to the blow-up of $\p^{2}_{L}$ in a point $q$ of degree $3$ with splitting field $E$ such that $\Gal(E/L) \simeq \Z/3\Z$.\label{it:item_(4)_Proposition_7_Z/6Z}
\end{enumerate}
\label{Prop:Proposition_7_Z/6Z}
\end{proposition}

\begin{proof}
(\ref{it:item_(1)_Proposition_7_Z/6Z}) There is only one $(-1)$-curve defined over $ \k $ on $ X_{\kb} $ namely $ E_{4} $. Its contraction yields a birational morphism $ \varepsilon : X \rightarrow Y $ onto a rational del Pezzo surface $Y$ of degree $6$ as in \cite[Lemma 4.6 and Remark B.3]{sz21}, where the $\Gal(\kb/\k)$-action on the hexagon of $Y_{\kb}$ is indicated in Figure \ref{Fig:Figures_Gal(k_barre/k)-action_on_dP6_Y_in_Z/6Z_case}. On the other hand, $ \Pi_{\kb} $ contains an orbit of three disjoint $(-1)$-curves $ D_{34},D_{24},D_{14} $ (see Figure \ref{fig:figure(k)_option_Gal(kbarre/k)-action_on_Pikbarre}). The contraction of the irreducible curve $D$ with geometric components $D_{34},D_{24},D_{14}$ yields a birational morphism $ \eta : X \rightarrow Z $ onto a rational del Pezzo surface of degree $8$. It contracts $D$ onto a point $p=\lbrace p_{1},p_{2},p_{3} \rbrace$ of degree $3$ and it conjugates the $ \Gal(\kb/\k) $-action on $X_{\kb}$ to an action that exchanges the fibrations of $Z_{\kb}$, so $ \rk \, \NS(Z) = 1 $. Hence, $ Z \simeq \mathcal{Q}^{L} $ for some quadratic extension $ L/\k $ \cite[Lemma 3.2(1)]{sz21}. Figure \ref{Fig:Figure_Blow-up_model_Z/6Z_case} shows the action of $\rho(\Gal(\kb/\k))=\langle (123)(45) \rangle$ on the image by $ \eta $ of the incidence diagram $ \Pi_{\kb} $. Then $ \eta $ conjugates the $ \Gal(\kb/L) $-action on $ \mathcal{Q}^{L}_{\kb} $ to an action on $ \Pi_{\kb} $  with $ \rho(\Gal(\kb/L)) = \mathbb{Z}/3\mathbb{Z} $. 
Let $ F/\k $ be any splitting field of $ p \in \mathcal{Q}^{L} $ and let us show that $\Gal(F/\k)\simeq\Z/3\Z$. Let $ M/\k $ be a minimal normal extension containing $F$ such that $ \Gal(\kb/M) $ acts trivially on $ \Pi_{\kb} $. The $ \Gal(\kb/\k) $-action on $ \Pi_{\kb} $ induces a short exact sequence $ 1 \rightarrow \Gal(\kb/M) \rightarrow \Gal(\kb/\k) \rightarrow \mathbb{Z}/6\mathbb{Z} \rightarrow 1 $, which implies $ \Gal(M/\k) \simeq \Gal(\kb/\k) / \Gal(\kb/M) \simeq \mathbb{Z}/6\mathbb{Z} $. The extension $F/\k$ is normal, so $\Gal(F/\k) \subseteq \Gal(M/\k)$. Then $ \mathbb{Z}/3\mathbb{Z} \subseteq \Gal(F/\k) $ implies $ \Gal(F/\k) \simeq \mathbb{Z}/3\mathbb{Z} $ since $ F/\k $ is a splitting field of $ p=\lbrace p_{1},p_{2},p_{3} \rbrace $.
Then $ FL/L $ is a splitting field of $p$ over $L$, $ FL/L $ is Galois and $ \Gal(FL/L) \simeq \Gal(F/F \cap L) = \Gal(F/\k) \simeq \mathbb{Z}/3\mathbb{Z} $ (see \cite[Theorem 5.5]{mor96}).
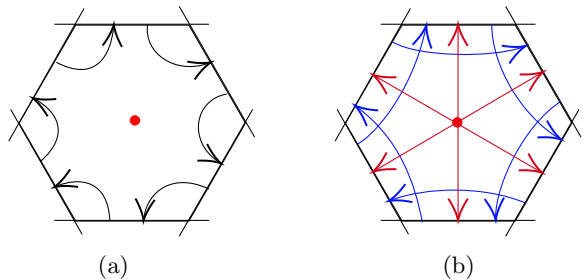
\begin{figure}[h]
\centering
\begin{subfigure}[b]{0.2\textwidth}
\begin{tikzpicture}[x=0.6pt,y=0.6pt,yscale=-1,xscale=1, scale=0.8, every node/.style={scale=0.8}]

\draw   (171,160.1) -- (142.45,209.55) -- (85.35,209.55) -- (56.8,160.1) -- (85.35,110.65) -- (142.45,110.65) -- cycle ;
\draw    (51.6,169.2) -- (89.6,102.2) ;
\draw    (88.6,216.2) -- (52.6,153.2) ;
\draw    (152.6,209.2) -- (74.6,209.2) ;
\draw    (174.6,153.2) -- (137.6,217.2) ;
\draw    (137.6,101.2) -- (175.6,167.2) ;
\draw    (152.6,110.2) -- (74.6,110.2) ;
\draw    (102.6,209.4) .. controls (101.65,188.5) and (87.16,187.45) .. (76.28,191.69) ;
\draw [shift={(74.6,192.4)}, rotate = 335.56] [color={rgb, 255:red, 0; green, 0; blue, 0 }  ][line width=0.75]    (10.93,-4.9) .. controls (6.95,-2.3) and (3.31,-0.67) .. (0,0) .. controls (3.31,0.67) and (6.95,2.3) .. (10.93,4.9)   ;
\draw    (67.6,179.4) .. controls (82.88,166.03) and (76.27,154.48) .. (65.19,148.25) ;
\draw [shift={(63.6,147.4)}, rotate = 26.57] [color={rgb, 255:red, 0; green, 0; blue, 0 }  ][line width=0.75]    (10.93,-4.9) .. controls (6.95,-2.3) and (3.31,-0.67) .. (0,0) .. controls (3.31,0.67) and (6.95,2.3) .. (10.93,4.9)   ;
\draw    (75.6,129.4) .. controls (95.65,139.9) and (102.94,125.79) .. (104.43,112.3) ;
\draw [shift={(104.6,110.4)}, rotate = 94.09] [color={rgb, 255:red, 0; green, 0; blue, 0 }  ][line width=0.75]    (10.93,-4.9) .. controls (6.95,-2.3) and (3.31,-0.67) .. (0,0) .. controls (3.31,0.67) and (6.95,2.3) .. (10.93,4.9)   ;
\draw    (121.6,110.4) .. controls (119.67,124.87) and (135.43,135.63) .. (151.82,130.07) ;
\draw [shift={(153.6,129.4)}, rotate = 157.62] [color={rgb, 255:red, 0; green, 0; blue, 0 }  ][line width=0.75]    (10.93,-4.9) .. controls (6.95,-2.3) and (3.31,-0.67) .. (0,0) .. controls (3.31,0.67) and (6.95,2.3) .. (10.93,4.9)   ;
\draw    (161.6,145.4) .. controls (145.11,148.31) and (141.79,169.1) .. (157.12,179.47) ;
\draw [shift={(158.6,180.4)}, rotate = 210.47] [color={rgb, 255:red, 0; green, 0; blue, 0 }  ][line width=0.75]    (10.93,-4.9) .. controls (6.95,-2.3) and (3.31,-0.67) .. (0,0) .. controls (3.31,0.67) and (6.95,2.3) .. (10.93,4.9)   ;
\draw    (152.6,193.4) .. controls (134.45,184.8) and (121.78,194.45) .. (119.81,207.53) ;
\draw [shift={(119.6,209.4)}, rotate = 274.09] [color={rgb, 255:red, 0; green, 0; blue, 0 }  ][line width=0.75]    (10.93,-4.9) .. controls (6.95,-2.3) and (3.31,-0.67) .. (0,0) .. controls (3.31,0.67) and (6.95,2.3) .. (10.93,4.9)   ;
\draw  [color={rgb, 255:red, 248; green, 7; blue, 7 }  ,draw opacity=1 ][fill={rgb, 255:red, 248; green, 7; blue, 7 }  ,fill opacity=1 ] (112.9,158.75) .. controls (112.9,157.45) and (113.95,156.4) .. (115.25,156.4) .. controls (116.55,156.4) and (117.6,157.45) .. (117.6,158.75) .. controls (117.6,160.05) and (116.55,161.1) .. (115.25,161.1) .. controls (113.95,161.1) and (112.9,160.05) .. (112.9,158.75) -- cycle ;

\draw (42,68.4) node [anchor=north west][inner sep=0.75pt]    {$Y$};
\end{tikzpicture}
\caption{}
\end{subfigure} $\quad\quad\quad$ \begin{subfigure}[b]{0.2\textwidth}
\begin{tikzpicture}[x=0.6pt,y=0.6pt,yscale=-1,xscale=1, scale=0.8, every node/.style={scale=0.8}]

\draw   (171,160.1) -- (142.45,209.55) -- (85.35,209.55) -- (56.8,160.1) -- (85.35,110.65) -- (142.45,110.65) -- cycle ;
\draw    (51.6,169.2) -- (89.6,102.2) ;
\draw    (88.6,216.2) -- (52.6,153.2) ;
\draw    (152.6,209.2) -- (74.6,209.2) ;
\draw    (174.6,153.2) -- (137.6,217.2) ;
\draw    (137.6,101.2) -- (175.6,167.2) ;
\draw    (152.6,110.2) -- (74.6,110.2) ;
\draw  [color={rgb, 255:red, 248; green, 7; blue, 7 }  ,draw opacity=1 ][fill={rgb, 255:red, 248; green, 7; blue, 7 }  ,fill opacity=1 ] (110.9,159.75) .. controls (110.9,158.45) and (111.95,157.4) .. (113.25,157.4) .. controls (114.55,157.4) and (115.6,158.45) .. (115.6,159.75) .. controls (115.6,161.05) and (114.55,162.1) .. (113.25,162.1) .. controls (111.95,162.1) and (110.9,161.05) .. (110.9,159.75) -- cycle ;
\draw [color={rgb, 255:red, 3; green, 19; blue, 249 }  ,draw opacity=1 ]   (95.6,210.2) .. controls (91.68,189.62) and (83.92,168.08) .. (62.9,152.17) ;
\draw [shift={(61.6,151.2)}, rotate = 36.03] [color={rgb, 255:red, 3; green, 19; blue, 249 }  ,draw opacity=1 ][line width=0.75]    (10.93,-4.9) .. controls (6.95,-2.3) and (3.31,-0.67) .. (0,0) .. controls (3.31,0.67) and (6.95,2.3) .. (10.93,4.9)   ;
\draw [color={rgb, 255:red, 3; green, 19; blue, 249 }  ,draw opacity=1 ]   (62.6,170.8) .. controls (69.49,167.85) and (89.97,146.46) .. (97.28,112.37) ;
\draw [shift={(97.6,110.8)}, rotate = 101.31] [color={rgb, 255:red, 3; green, 19; blue, 249 }  ,draw opacity=1 ][line width=0.75]    (10.93,-4.9) .. controls (6.95,-2.3) and (3.31,-0.67) .. (0,0) .. controls (3.31,0.67) and (6.95,2.3) .. (10.93,4.9)   ;
\draw [color={rgb, 255:red, 3; green, 19; blue, 249 }  ,draw opacity=1 ]   (80.6,119.2) .. controls (94.04,126.88) and (130.52,127.19) .. (145.8,121.89) ;
\draw [shift={(147.6,121.2)}, rotate = 156.8] [color={rgb, 255:red, 3; green, 19; blue, 249 }  ,draw opacity=1 ][line width=0.75]    (10.93,-4.9) .. controls (6.95,-2.3) and (3.31,-0.67) .. (0,0) .. controls (3.31,0.67) and (6.95,2.3) .. (10.93,4.9)   ;
\draw [color={rgb, 255:red, 3; green, 19; blue, 249 }  ,draw opacity=1 ]   (130.6,110.4) .. controls (130.6,131.16) and (145.65,156.24) .. (163.9,167.4) ;
\draw [shift={(165.6,168.4)}, rotate = 209.16] [color={rgb, 255:red, 3; green, 19; blue, 249 }  ,draw opacity=1 ][line width=0.75]    (10.93,-4.9) .. controls (6.95,-2.3) and (3.31,-0.67) .. (0,0) .. controls (3.31,0.67) and (6.95,2.3) .. (10.93,4.9)   ;
\draw [color={rgb, 255:red, 3; green, 19; blue, 249 }  ,draw opacity=1 ]   (165.6,149) .. controls (146.4,161.48) and (132.73,191.47) .. (132.55,207.12) ;
\draw [shift={(132.6,209)}, rotate = 266.19] [color={rgb, 255:red, 3; green, 19; blue, 249 }  ,draw opacity=1 ][line width=0.75]    (10.93,-4.9) .. controls (6.95,-2.3) and (3.31,-0.67) .. (0,0) .. controls (3.31,0.67) and (6.95,2.3) .. (10.93,4.9)   ;
\draw [color={rgb, 255:red, 3; green, 19; blue, 249 }  ,draw opacity=1 ]   (147.6,200.4) .. controls (126.26,191.67) and (98.33,192.35) .. (80.25,198.79) ;
\draw [shift={(78.6,199.4)}, rotate = 338.75] [color={rgb, 255:red, 3; green, 19; blue, 249 }  ,draw opacity=1 ][line width=0.75]    (10.93,-4.9) .. controls (6.95,-2.3) and (3.31,-0.67) .. (0,0) .. controls (3.31,0.67) and (6.95,2.3) .. (10.93,4.9)   ;
\draw [color={rgb, 255:red, 208; green, 2; blue, 27 }  ,draw opacity=1 ]   (113.6,112.2) -- (113.6,158.2) -- (113.6,207.2) ;
\draw [shift={(113.6,209.2)}, rotate = 270] [color={rgb, 255:red, 208; green, 2; blue, 27 }  ,draw opacity=1 ][line width=0.75]    (10.93,-4.9) .. controls (6.95,-2.3) and (3.31,-0.67) .. (0,0) .. controls (3.31,0.67) and (6.95,2.3) .. (10.93,4.9)   ;
\draw [shift={(113.6,110.2)}, rotate = 90] [color={rgb, 255:red, 208; green, 2; blue, 27 }  ,draw opacity=1 ][line width=0.75]    (10.93,-4.9) .. controls (6.95,-2.3) and (3.31,-0.67) .. (0,0) .. controls (3.31,0.67) and (6.95,2.3) .. (10.93,4.9)   ;
\draw [color={rgb, 255:red, 208; green, 2; blue, 27 }  ,draw opacity=1 ]   (72.32,183.69) -- (116.2,157.92) -- (154.88,135.21) ;
\draw [shift={(156.6,134.2)}, rotate = 149.58] [color={rgb, 255:red, 208; green, 2; blue, 27 }  ,draw opacity=1 ][line width=0.75]    (10.93,-4.9) .. controls (6.95,-2.3) and (3.31,-0.67) .. (0,0) .. controls (3.31,0.67) and (6.95,2.3) .. (10.93,4.9)   ;
\draw [shift={(70.6,184.7)}, rotate = 329.58] [color={rgb, 255:red, 208; green, 2; blue, 27 }  ,draw opacity=1 ][line width=0.75]    (10.93,-4.9) .. controls (6.95,-2.3) and (3.31,-0.67) .. (0,0) .. controls (3.31,0.67) and (6.95,2.3) .. (10.93,4.9)   ;
\draw [color={rgb, 255:red, 208; green, 2; blue, 27 }  ,draw opacity=1 ]   (72.33,136.7) -- (113,160.24) -- (154.37,184.2) ;
\draw [shift={(156.1,185.2)}, rotate = 210.07] [color={rgb, 255:red, 208; green, 2; blue, 27 }  ,draw opacity=1 ][line width=0.75]    (10.93,-4.9) .. controls (6.95,-2.3) and (3.31,-0.67) .. (0,0) .. controls (3.31,0.67) and (6.95,2.3) .. (10.93,4.9)   ;
\draw [shift={(70.6,135.7)}, rotate = 30.07] [color={rgb, 255:red, 208; green, 2; blue, 27 }  ,draw opacity=1 ][line width=0.75]    (10.93,-4.9) .. controls (6.95,-2.3) and (3.31,-0.67) .. (0,0) .. controls (3.31,0.67) and (6.95,2.3) .. (10.93,4.9)   ;
\end{tikzpicture}
\caption{}
\end{subfigure}
\caption{The $ \Gal(\kb/\k) $-action on $Y_{\kb}$.}
\label{Fig:Figures_Gal(k_barre/k)-action_on_dP6_Y_in_Z/6Z_case}
\end{figure}
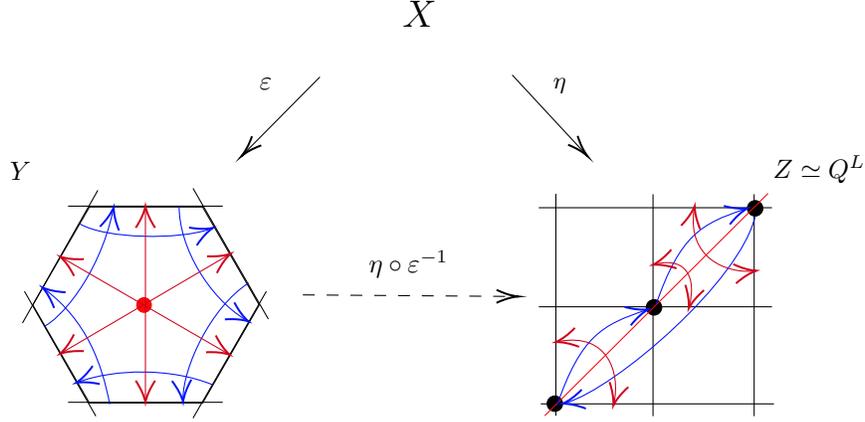
\begin{figure}[h]
\centering
\begin{tikzpicture}[x=0.6pt,y=0.6pt,yscale=-1,xscale=1,scale=0.8, every node/.style={scale=0.8}]

\draw   (171,170.1) -- (142.45,219.55) -- (85.35,219.55) -- (56.8,170.1) -- (85.35,120.65) -- (142.45,120.65) -- cycle ;
\draw    (51.6,179.2) -- (89.6,112.2) ;
\draw    (88.6,226.2) -- (52.6,163.2) ;
\draw    (152.6,219.2) -- (74.6,219.2) ;
\draw    (174.6,163.2) -- (137.6,227.2) ;
\draw    (137.6,111.2) -- (175.6,177.2) ;
\draw    (152.6,120.2) -- (74.6,120.2) ;
\draw  [color={rgb, 255:red, 248; green, 7; blue, 7 }  ,draw opacity=1 ][fill={rgb, 255:red, 248; green, 7; blue, 7 }  ,fill opacity=1 ] (109.2,170.1) .. controls (109.2,168.06) and (110.86,166.4) .. (112.9,166.4) .. controls (114.94,166.4) and (116.6,168.06) .. (116.6,170.1) .. controls (116.6,172.14) and (114.94,173.8) .. (112.9,173.8) .. controls (110.86,173.8) and (109.2,172.14) .. (109.2,170.1) -- cycle ;
\draw [color={rgb, 255:red, 3; green, 19; blue, 249 }  ,draw opacity=1 ]   (95.6,220.2) .. controls (91.68,199.62) and (83.92,178.08) .. (62.9,162.17) ;
\draw [shift={(61.6,161.2)}, rotate = 36.03] [color={rgb, 255:red, 3; green, 19; blue, 249 }  ,draw opacity=1 ][line width=0.75]    (10.93,-4.9) .. controls (6.95,-2.3) and (3.31,-0.67) .. (0,0) .. controls (3.31,0.67) and (6.95,2.3) .. (10.93,4.9)   ;
\draw [color={rgb, 255:red, 3; green, 19; blue, 249 }  ,draw opacity=1 ]   (62.6,180.8) .. controls (69.49,177.85) and (89.97,156.46) .. (97.28,122.37) ;
\draw [shift={(97.6,120.8)}, rotate = 101.31] [color={rgb, 255:red, 3; green, 19; blue, 249 }  ,draw opacity=1 ][line width=0.75]    (10.93,-4.9) .. controls (6.95,-2.3) and (3.31,-0.67) .. (0,0) .. controls (3.31,0.67) and (6.95,2.3) .. (10.93,4.9)   ;
\draw [color={rgb, 255:red, 3; green, 19; blue, 249 }  ,draw opacity=1 ]   (80.6,129.2) .. controls (94.04,136.88) and (130.52,137.19) .. (145.8,131.89) ;
\draw [shift={(147.6,131.2)}, rotate = 156.8] [color={rgb, 255:red, 3; green, 19; blue, 249 }  ,draw opacity=1 ][line width=0.75]    (10.93,-4.9) .. controls (6.95,-2.3) and (3.31,-0.67) .. (0,0) .. controls (3.31,0.67) and (6.95,2.3) .. (10.93,4.9)   ;
\draw [color={rgb, 255:red, 3; green, 19; blue, 249 }  ,draw opacity=1 ]   (130.6,120.4) .. controls (130.6,141.16) and (145.65,166.24) .. (163.9,177.4) ;
\draw [shift={(165.6,178.4)}, rotate = 209.16] [color={rgb, 255:red, 3; green, 19; blue, 249 }  ,draw opacity=1 ][line width=0.75]    (10.93,-4.9) .. controls (6.95,-2.3) and (3.31,-0.67) .. (0,0) .. controls (3.31,0.67) and (6.95,2.3) .. (10.93,4.9)   ;
\draw [color={rgb, 255:red, 3; green, 19; blue, 249 }  ,draw opacity=1 ]   (165.6,159) .. controls (146.4,171.48) and (132.73,201.47) .. (132.55,217.12) ;
\draw [shift={(132.6,219)}, rotate = 266.19] [color={rgb, 255:red, 3; green, 19; blue, 249 }  ,draw opacity=1 ][line width=0.75]    (10.93,-4.9) .. controls (6.95,-2.3) and (3.31,-0.67) .. (0,0) .. controls (3.31,0.67) and (6.95,2.3) .. (10.93,4.9)   ;
\draw [color={rgb, 255:red, 3; green, 19; blue, 249 }  ,draw opacity=1 ]   (147.6,210.4) .. controls (126.26,201.67) and (98.33,202.35) .. (80.25,208.79) ;
\draw [shift={(78.6,209.4)}, rotate = 338.75] [color={rgb, 255:red, 3; green, 19; blue, 249 }  ,draw opacity=1 ][line width=0.75]    (10.93,-4.9) .. controls (6.95,-2.3) and (3.31,-0.67) .. (0,0) .. controls (3.31,0.67) and (6.95,2.3) .. (10.93,4.9)   ;
\draw [color={rgb, 255:red, 208; green, 2; blue, 27 }  ,draw opacity=1 ]   (113.6,122.2) -- (113.6,168.2) -- (113.6,217.2) ;
\draw [shift={(113.6,219.2)}, rotate = 270] [color={rgb, 255:red, 208; green, 2; blue, 27 }  ,draw opacity=1 ][line width=0.75]    (10.93,-4.9) .. controls (6.95,-2.3) and (3.31,-0.67) .. (0,0) .. controls (3.31,0.67) and (6.95,2.3) .. (10.93,4.9)   ;
\draw [shift={(113.6,120.2)}, rotate = 90] [color={rgb, 255:red, 208; green, 2; blue, 27 }  ,draw opacity=1 ][line width=0.75]    (10.93,-4.9) .. controls (6.95,-2.3) and (3.31,-0.67) .. (0,0) .. controls (3.31,0.67) and (6.95,2.3) .. (10.93,4.9)   ;
\draw [color={rgb, 255:red, 208; green, 2; blue, 27 }  ,draw opacity=1 ]   (72.32,193.69) -- (116.2,167.92) -- (154.88,145.21) ;
\draw [shift={(156.6,144.2)}, rotate = 149.58] [color={rgb, 255:red, 208; green, 2; blue, 27 }  ,draw opacity=1 ][line width=0.75]    (10.93,-4.9) .. controls (6.95,-2.3) and (3.31,-0.67) .. (0,0) .. controls (3.31,0.67) and (6.95,2.3) .. (10.93,4.9)   ;
\draw [shift={(70.6,194.7)}, rotate = 329.58] [color={rgb, 255:red, 208; green, 2; blue, 27 }  ,draw opacity=1 ][line width=0.75]    (10.93,-4.9) .. controls (6.95,-2.3) and (3.31,-0.67) .. (0,0) .. controls (3.31,0.67) and (6.95,2.3) .. (10.93,4.9)   ;
\draw [color={rgb, 255:red, 208; green, 2; blue, 27 }  ,draw opacity=1 ]   (72.33,146.7) -- (113,170.24) -- (154.37,194.2) ;
\draw [shift={(156.1,195.2)}, rotate = 210.07] [color={rgb, 255:red, 208; green, 2; blue, 27 }  ,draw opacity=1 ][line width=0.75]    (10.93,-4.9) .. controls (6.95,-2.3) and (3.31,-0.67) .. (0,0) .. controls (3.31,0.67) and (6.95,2.3) .. (10.93,4.9)   ;
\draw [shift={(70.6,145.7)}, rotate = 30.07] [color={rgb, 255:red, 208; green, 2; blue, 27 }  ,draw opacity=1 ][line width=0.75]    (10.93,-4.9) .. controls (6.95,-2.3) and (3.31,-0.67) .. (0,0) .. controls (3.31,0.67) and (6.95,2.3) .. (10.93,4.9)   ;
\draw    (201.6,55) -- (163.41,93.58) ;
\draw [shift={(162,95)}, rotate = 314.71] [color={rgb, 255:red, 0; green, 0; blue, 0 }  ][line width=0.75]    (10.93,-3.29) .. controls (6.95,-1.4) and (3.31,-0.3) .. (0,0) .. controls (3.31,0.3) and (6.95,1.4) .. (10.93,3.29)   ;
\draw    (298.6,54) -- (335.24,93.53) ;
\draw [shift={(336.6,95)}, rotate = 227.17] [color={rgb, 255:red, 0; green, 0; blue, 0 }  ][line width=0.75]    (10.93,-3.29) .. controls (6.95,-1.4) and (3.31,-0.3) .. (0,0) .. controls (3.31,0.3) and (6.95,1.4) .. (10.93,3.29)   ;
\draw    (320.6,114) -- (320.6,226) ;
\draw    (369.6,114.2) -- (369.6,226.2) ;
\draw    (420.6,114.2) -- (420.6,226.2) ;
\draw    (312,121) -- (429.6,121) ;
\draw    (312,171) -- (429.6,171) ;
\draw    (313,220) -- (430.6,220) ;
\draw  [fill={rgb, 255:red, 0; green, 0; blue, 0 }  ,fill opacity=1 ] (316.2,219.9) .. controls (316.2,217.75) and (317.95,216) .. (320.1,216) .. controls (322.25,216) and (324,217.75) .. (324,219.9) .. controls (324,222.05) and (322.25,223.8) .. (320.1,223.8) .. controls (317.95,223.8) and (316.2,222.05) .. (316.2,219.9) -- cycle ;
\draw  [fill={rgb, 255:red, 0; green, 0; blue, 0 }  ,fill opacity=1 ] (366.2,171.3) .. controls (366.2,169.15) and (367.95,167.4) .. (370.1,167.4) .. controls (372.25,167.4) and (374,169.15) .. (374,171.3) .. controls (374,173.45) and (372.25,175.2) .. (370.1,175.2) .. controls (367.95,175.2) and (366.2,173.45) .. (366.2,171.3) -- cycle ;
\draw  [fill={rgb, 255:red, 0; green, 0; blue, 0 }  ,fill opacity=1 ] (417.2,121.3) .. controls (417.2,119.15) and (418.95,117.4) .. (421.1,117.4) .. controls (423.25,117.4) and (425,119.15) .. (425,121.3) .. controls (425,123.45) and (423.25,125.2) .. (421.1,125.2) .. controls (418.95,125.2) and (417.2,123.45) .. (417.2,121.3) -- cycle ;
\draw [color={rgb, 255:red, 248; green, 7; blue, 7 }  ,draw opacity=1 ]   (427.6,113.8) -- (315,226) ;
\draw [color={rgb, 255:red, 3; green, 19; blue, 249 }  ,draw opacity=1 ]   (321.1,220.9) .. controls (333.29,185.7) and (344.05,180.07) .. (365.53,172.86) ;
\draw [shift={(367.2,172.3)}, rotate = 161.64] [color={rgb, 255:red, 3; green, 19; blue, 249 }  ,draw opacity=1 ][line width=0.75]    (10.93,-4.9) .. controls (6.95,-2.3) and (3.31,-0.67) .. (0,0) .. controls (3.31,0.67) and (6.95,2.3) .. (10.93,4.9)   ;
\draw [color={rgb, 255:red, 3; green, 19; blue, 249 }  ,draw opacity=1 ]   (371.1,169.9) .. controls (383.29,134.7) and (394.05,129.07) .. (415.53,121.86) ;
\draw [shift={(417.2,121.3)}, rotate = 161.64] [color={rgb, 255:red, 3; green, 19; blue, 249 }  ,draw opacity=1 ][line width=0.75]    (10.93,-4.9) .. controls (6.95,-2.3) and (3.31,-0.67) .. (0,0) .. controls (3.31,0.67) and (6.95,2.3) .. (10.93,4.9)   ;
\draw [color={rgb, 255:red, 3; green, 19; blue, 249 }  ,draw opacity=1 ]   (421.1,125.2) .. controls (413.79,157.96) and (350.86,212.01) .. (325.85,219.42) ;
\draw [shift={(324,219.9)}, rotate = 347.81] [color={rgb, 255:red, 3; green, 19; blue, 249 }  ,draw opacity=1 ][line width=0.75]    (10.93,-4.9) .. controls (6.95,-2.3) and (3.31,-0.67) .. (0,0) .. controls (3.31,0.67) and (6.95,2.3) .. (10.93,4.9)   ;
\draw [color={rgb, 255:red, 208; green, 2; blue, 27 }  ,draw opacity=1 ]   (322.59,188.4) .. controls (346.88,184.12) and (353.02,204.44) .. (350.02,219.01) ;
\draw [shift={(349.6,220.8)}, rotate = 284.93] [color={rgb, 255:red, 208; green, 2; blue, 27 }  ,draw opacity=1 ][line width=0.75]    (10.93,-4.9) .. controls (6.95,-2.3) and (3.31,-0.67) .. (0,0) .. controls (3.31,0.67) and (6.95,2.3) .. (10.93,4.9)   ;
\draw [shift={(320.6,188.8)}, rotate = 347.47] [color={rgb, 255:red, 208; green, 2; blue, 27 }  ,draw opacity=1 ][line width=0.75]    (10.93,-4.9) .. controls (6.95,-2.3) and (3.31,-0.67) .. (0,0) .. controls (3.31,0.67) and (6.95,2.3) .. (10.93,4.9)   ;
\draw [color={rgb, 255:red, 208; green, 2; blue, 27 }  ,draw opacity=1 ]   (371.69,149.29) .. controls (387.18,146) and (391.1,157.27) .. (388.07,169.99) ;
\draw [shift={(387.6,171.8)}, rotate = 285.95] [color={rgb, 255:red, 208; green, 2; blue, 27 }  ,draw opacity=1 ][line width=0.75]    (10.93,-4.9) .. controls (6.95,-2.3) and (3.31,-0.67) .. (0,0) .. controls (3.31,0.67) and (6.95,2.3) .. (10.93,4.9)   ;
\draw [shift={(369.6,149.8)}, rotate = 344.48] [color={rgb, 255:red, 208; green, 2; blue, 27 }  ,draw opacity=1 ][line width=0.75]    (10.93,-4.9) .. controls (6.95,-2.3) and (3.31,-0.67) .. (0,0) .. controls (3.31,0.67) and (6.95,2.3) .. (10.93,4.9)   ;
\draw [color={rgb, 255:red, 208; green, 2; blue, 27 }  ,draw opacity=1 ]   (390.22,122.54) .. controls (387.06,146.16) and (406.17,154.03) .. (419.72,153) ;
\draw [shift={(421.6,152.8)}, rotate = 171.87] [color={rgb, 255:red, 208; green, 2; blue, 27 }  ,draw opacity=1 ][line width=0.75]    (10.93,-4.9) .. controls (6.95,-2.3) and (3.31,-0.67) .. (0,0) .. controls (3.31,0.67) and (6.95,2.3) .. (10.93,4.9)   ;
\draw [shift={(390.6,120.2)}, rotate = 100.65] [color={rgb, 255:red, 208; green, 2; blue, 27 }  ,draw opacity=1 ][line width=0.75]    (10.93,-4.9) .. controls (6.95,-2.3) and (3.31,-0.67) .. (0,0) .. controls (3.31,0.67) and (6.95,2.3) .. (10.93,4.9)   ;
\draw  [dash pattern={on 4.5pt off 4.5pt}]  (193,165) -- (299.6,165.79) ;
\draw [shift={(301.6,165.8)}, rotate = 180.42] [color={rgb, 255:red, 0; green, 0; blue, 0 }  ][line width=0.75]    (10.93,-3.29) .. controls (6.95,-1.4) and (3.31,-0.3) .. (0,0) .. controls (3.31,0.3) and (6.95,1.4) .. (10.93,3.29)   ;

\draw (44,96.4) node [anchor=north west][inner sep=0.75pt]    {$Y$};
\draw (242,15.4) node [anchor=north west][inner sep=0.75pt]  [font=\Large]  {$X$};
\draw (429,93.4) node [anchor=north west][inner sep=0.75pt]    {$Z\simeq Q^{L}$};
\draw (170,54.4) node [anchor=north west][inner sep=0.75pt]    {$\varepsilon $};
\draw (318,54.4) node [anchor=north west][inner sep=0.75pt]  [font=\small]  {$\eta $};
\draw (225,141.4) node [anchor=north west][inner sep=0.75pt]    {$\eta \circ \varepsilon ^{-1}$};
\end{tikzpicture}
\caption{Blow-up model for a del Pezzo surface as in Proposition \ref{Prop:Proposition_7_Z/6Z}.}
\label{Fig:Figure_Blow-up_model_Z/6Z_case}
\end{figure}

(\ref{it:item_(2)_Proposition_7_Z/6Z}) Two del Pezzo surfaces $ \mathcal{Q}^{L} $ and $ \mathcal{Q}^{L'} $ are isomorphic if and only if the corresponding quadratic extensions $L$ and $L'$ are $\k$-isomorphic \cite[Lemma 3.2(3)]{sz21}. By Lemma \ref{lem:Lemma_transitive_action_Aut_k(Q^L)_points_of_degree_3}, $ \Aut_{\k}(\mathcal{Q}^{L}) $ acts transitively on the set of points of degree $3$ in $ \mathcal{Q}^{L} $ with $\k$-isomorphic splitting fields and whose geometric components are in general position. This yields the claim.

(\ref{it:item_(3)_Proposition_7_Z/6Z}\&\ref{it:item_(4)_Proposition_7_Z/6Z}) The action of $ \Aut_{\k}(X) $ on the set of $ (-1) $-curves of $X$ yields an isomorphism $ \Psi : \Aut_{\k}(X) \overset{\sim}{\longrightarrow} \Aut(\Pi_{\kb,\rho}) $ by Lemma \ref{lem:Lemma_faithful_action_Aut_k(X)_on_Pi_L_rho}. Since $\Aut(\Pi_{\kb,\rho}) \simeq \lbrace \sigma \in \Sym_{5} \, \vert \, (123)(45) \circ \sigma = \sigma \circ (123)(45) \rbrace = C_{\Sym_{5}}(\langle (123)(45) \rangle)$, we get $\Aut(\Pi_{\kb,\rho}) \simeq \langle (123) \rangle \times \langle (45) \rangle \simeq \Z/3\Z \times \Z/2\Z$.
We now construct the corresponding geometric actions. From part (\ref{it:item_(1)_Proposition_7_Z/6Z}) we know that $Y_{L}$ is isomorphic to the del Pezzo surface of degree $6$ from \cite[Lemma 4.2]{sz21}, which is the blow-up $ \pi : Y_{L} \rightarrow \p^{2}_{L} $ of a point $q=\lbrace q_{1},q_{2},q_{3} \rbrace$ of degree $3$, and so that $X_{L}$ is isomorphic to the del Pezzo surface from Proposition \ref{Prop:Proposition_3_Z/3Z}. Since $ Y $ is rational, $ \Gal(L/\k) = \langle g \rangle $ has a fixed point $ r \in Y(\k) $. Let $ \phi_{q} \in \Bir_{L}(\mathbb{P}^{2}) $ be the quadratic involution from Proposition \ref{Prop:Proposition_3_Z/3Z}(\ref{it:item_(3)_Proposition_3_Z3}), such that $ \phi_{q} $ fixes $ \pi(r) \in \mathbb{P}^{2}_{L}(L) \simeq \mathbb{P}^{2}(L) $ and $ \Phi_{q} := \pi^{-1} \phi_{q} \pi \in \Aut_{L}(Y) $ induces a rotation of order two on the hexagon of $ Y_{L} $. Then $ \Phi_{q}(g\Phi_{q}g) \in \Aut_{L}(Y) $ preserves the edges of the hexagon and fixes $r$. It therefore descends to an element of $ \Aut_{L}(\mathbb{P}^{2},q_{1},q_{2},q_{3}) $ fixing $ \pi(r) $ and is hence equal to the identity. It follows that $ \Phi_{q} \in \Aut_{\k}(Y) $ and it induces a rotation of order two on the hexagon of $ Y $. Thus, it lifts by $ \varepsilon $ to an automorphism $ \widehat{\Phi_{q}} $ of $X$ defined over $ \k $ that acts as $(45)$ on $ \Pi_{\kb} $. Now let $ \tilde{\alpha} \in \Aut_{L}(\mathbb{P}^{2},\lbrace q_{1},q_{2},q_{3} \rbrace) $ be the automorphism of order three from Proposition \ref{Prop:Proposition_3_Z/3Z}(\ref{it:item_(3)_Proposition_3_Z3}) such that $ \tilde{\alpha} $ fixes $ \pi(r) \in \mathbb{P}^{2}(L) $, $\tilde{\alpha} : q_{1} \mapsto q_{2} \mapsto q_{3}$ and $ \alpha := \pi^{-1} \tilde{\alpha} \pi \in \text{Aut}_{L}(Y) $ induces a rotation of order three on the hexagon of $ Y_{L} $. We can argue as above that $ \alpha \in \Aut_{\k}(Y) $ and thus it lifts to an automorphism $ \widehat{\alpha} $ of $X$ defined over $\k$ that acts as $(123)$ on $ \Pi_{\kb} $. In conclusion, $ \Aut_{\k}(X)=\langle \widehat{\alpha} \rangle  \times \langle \widehat{\Phi_{q}} \rangle \simeq \mathbb{Z}/3\mathbb{Z} \times \mathbb{Z}/2\mathbb{Z} $. For the last assertion, we have a contraction $\eta : X \rightarrow \Ql^{L}$ onto a point of degree $3$ from (\ref{it:item_(1)_Proposition_7_Z/6Z}), so thanks to Lemma \ref{lem:Lemma_transitive_action_Aut_k(Q^L)_points_of_degree_3}, we get $ \rk \, \NS(X)^{\Aut_{\k}(X)} = 2 $. Now, note that the curve $E_{4}$ is stabilized by $\Aut_{\k}(X)$ and can then be contracted equivariantly. It gives an equivariant birational morphism to the rational del Pezzo surface $Y$ with $\rk \, \NS(Y) = 1$, which yields the claim.  
\end{proof}

\begin{example} We construct a del Pezzo surface as in Proposition \ref{Prop:Proposition_7_Z/6Z}. Let $\k = \mathbb{F}_{2}$, let $L=\k(a_{1})$ be a quadratic extension of $\k$ and $F/\k$ be the splitting field of $P(X)=X^{3}+X+1$ as in Example \ref{ex:example_Z/3Z}. Then $3=\vert F:\k \vert=\vert \Gal(F/\k) \vert$ and $ \sigma : s \mapsto s^{2} $ generates $\Gal(F/\k)$. If $\zeta$ is a root of $P(X)$, $\zeta \neq a_{1}$, then the point$$p=\lbrace ([1-\zeta:\zeta^{4}],[1-\zeta:\zeta^{4}]);([1-\zeta^{2}:\zeta],[1-\zeta^{2}:\zeta]);([1-\zeta^{4}:\zeta^{2}],[1-\zeta^{4}:\zeta^{2}]) \rbrace$$is of degree $3$, its geometric components are in general position in $\mathcal{Q}^{L}_{\kb}$ and any non-trivial element of $\Gal(FL/L) \simeq \Gal(F/F \cap L) = \Gal(F/\k) \simeq \Z/3\Z$ permutes them non-trivially.
\end{example}

\begin{proposition}
Let $X$ be a del Pezzo surface of degree $5$ such that $ \rho(\Gal(\kb/\k)) = \langle (123),(12),(45) \rangle \simeq \Sym_{3} \times \Z/2\Z $ in $ \text{Sym}_{5} $ as indicated in Figure \ref{fig:figure(l)_option_Gal(kbarre/k)-action_on_Pikbarre}. Then the following holds.
\begin{enumerate}
\item There exist a quadratic extension $ L/\k $ and a birational morphism $ \eta : X \longrightarrow \mathcal{Q}^{L} $ contracting an irreducible curve onto a point $ p=\lbrace p_{1},p_{2},p_{3} \rbrace $ of degree $3$ whose geometric components are contained in pairwise distinct rulings of $ \mathcal{Q}^{L}_{L} $, and with splitting field $ F $ such that $ \Gal(F/\k) \simeq \Gal(FL/L) \simeq \Sym_{3} $, where $FL$ is a splitting field of $p$ over $L$.
\label{it:item_(1)_Proposition_?_S3xZ2}
\item Any two such del Pezzo surfaces are isomorphic if and only if the corresponding field extensions of degree two and six are $\k$-isomorphic.\label{it:item_(2)_Proposition_?_S3xZ2}
\item $ \Aut_{\k}(X)=\langle \widehat{\Phi_{q}} \rangle \simeq \mathbb{Z}/2\mathbb{Z} $, where $\widehat{\Phi_{q}}$ is the lift of a quadratic birational involution of $\p^{2}_{L}$ with base-point $q$ of degree $3$.\label{it:item_(3)_Proposition_?_S3xZ2}
\item $ \rk \, \NS(X)^{\Aut_{\k}(X)} = 2 $, so in particular $ X \rightarrow \ast $ is not an $ \Aut_{\k}(X) $-Mori fibre space, and the $\Aut_{\k}(X)$-minimal models of $X$ are the quadric $\Ql^{L}$ and the rational del Pezzo surface $Y$ of degree $6$ such that $Y_{L}$ is isomorphic to the blow-up of $\p^{2}_{L}$ in a point $q$ of degree $3$ with splitting field $E$ such that $\Gal(E/L) \simeq \Sym_{3}$.\label{it:item_(4)_Proposition_?_S3xZ2}
\end{enumerate}
\label{Prop:Proposition_?_S3xZ2}
\end{proposition}

\begin{proof}
The proof of points (\ref{it:item_(1)_Proposition_?_S3xZ2}), (\ref{it:item_(2)_Proposition_?_S3xZ2}), (\ref{it:item_(4)_Proposition_?_S3xZ2}) is analogous to that of Proposition \ref{Prop:Proposition_7_Z/6Z}.\\ (\ref{it:item_(3)_Proposition_?_S3xZ2}) In this case, we have $\Aut(\Pi_{\kb,\rho}) \simeq C_{\Sym_{5}}(\langle (123),(12),(45) \rangle) = \langle (45) \rangle $, and from Proposition \ref{Prop:Proposition_7_Z/6Z}(\ref{it:item_(3)_Proposition_7_Z/6Z}) there exists a birational quadratic involution $\phi_{q} \in \text{Bir}_{L}(\p^{2})$ that lifts to an automorphism $\widehat{\Phi_{q}}$ of $X$ defined over $\k$ and that acts as $(45)$ on the $(-1)$-curves, so $\Aut_{\k}(X) = \langle \widehat{\Phi_{q}} \rangle \simeq \Z/2\Z$.
\end{proof}

\begin{example} We construct a del Pezzo surface as in Proposition \ref{Prop:Proposition_?_S3xZ2}. Let $\k=\Q$, $\zeta=\sqrt[3]{2}$ and $\omega=e^{\frac{i2\pi}{3}} $. Let $L=\k(a_{1})$, $a_{1} \neq \zeta, \omega$, be a quadratic extension and let $F=\k(\zeta,\omega)$ be the Galois extension of degree $6$ in Example \ref{ex:example_Sym3_case_1}, with $\Gal(F/\k) \simeq \Sym_{3} $. Then the point$$p=\lbrace ([\zeta-\zeta^{2}:1],[\zeta-\zeta^{2}:1]);([\omega\zeta-\omega^{2}\zeta^{2}:1],[\omega\zeta-\omega^{2}\zeta^{2}:1]);([\omega^{2}\zeta-\omega\zeta^{2}:1],[\omega^{2}\zeta-\omega\zeta^{2}:1]) \rbrace$$is of degree $3$, its geometric components are in general position in $\mathcal{Q}^{L}_{\kb}$ and any non-trivial element of $\Gal(FL/L) \simeq \Gal(F/F \cap L) = \Gal(F/\k) \simeq \Sym_{3} $ permutes them non-trivially.
\end{example}

\begin{proposition}
Let $X$ be a del Pezzo surface of degree $5$ such that $ \rho(\Gal(\kb/\k)) = \langle (123),(12)(45) \rangle \simeq \Sym_{3} $ in $ \Sym_{5} $ as indicated in Figure \ref{fig:figure(m)_option_Gal(kbarre/k)-action_on_Pikbarre}. Then the following holds.
\begin{enumerate}
\item There exist a quadratic extension $ L/\k $ and a birational morphism $ \eta : X \longrightarrow \mathcal{Q}^{L} $ contracting an irreducible curve onto a point $ p=\lbrace p_{1},p_{2},p_{3} \rbrace $ of degree $3$ whose geometric components are contained in pairwise distinct rulings of $ \mathcal{Q}^{L}_{L} $, and with splitting field $F$ such that $ \Gal(F/\k) \simeq \Sym_{3} $ and $ \Gal(FL/L) \simeq \mathbb{Z}/3\mathbb{Z} $, where $FL$ is a splitting field of $p$ over $L$.\label{it:item_(1):Proposition_9_Sym3_2}
\item Any two such del Pezzo surfaces $ X,X' $ are isomorphic if and only if the corresponding field extensions $ L,L' $ and $ F,F' $ are $\k$-isomorphic.\label{it:item_(2):Proposition_9_Sym3_2}
\item $ \Aut_{\k}(X)=\langle \widehat{\Phi_{q}} \rangle \simeq \mathbb{Z}/2\mathbb{Z} $, where $\widehat{\Phi_{q}}$ is the lift of a quadratic birational involution of $\p^{2}_{L}$ with base-point $q$ of degree $3$.\label{it:item_(3):Proposition_9_Sym3_2}
\item $ \rk \, \NS(X)^{\Aut_{\k}(X)} = 2 $, so in particular $ X \rightarrow \ast $ is not an $ \Aut_{\k}(X) $-Mori fibre space, and the $\Aut_{\k}(X)$-minimal models of $X$ are the quadric $\Ql^{L}$ and the rational del Pezzo surface $Y$ of degree $6$ such that $Y_{L}$ is isomorphic to the blow-up of $\p^{2}_{L}$ in a point $q$ of degree $3$ with splitting field $E$ such that $\Gal(E/L) \simeq \Z/3\Z$.
\label{it:item_(4):Proposition_9_Sym3_2}
\end{enumerate}
\label{Prop:Proposition9_Sym3_2}
\end{proposition}

\begin{proof}
(\ref{it:item_(1):Proposition_9_Sym3_2}) There is only one $(-1)$-curve defined over $\k$ on $X_{\kb}$ namely $ E_{4} $. Its contraction yields a birational morphism $ \varepsilon : X \rightarrow Y $ onto a rational del Pezzo surface of degree $6$ as in Figure \ref{fig:Figure_Gal(k_barre/k)-action_on_Y_and_Q^L_in_Sym_3_2_case}, where the $\Gal(\kb/\k)$-action on the hexagon of $Y_{\kb}$ is indicated by arrows. On the other hand, the contraction of the irreducible curve denoted $D$ with geometric components $D_{34},D_{24},D_{14}$ yields a birational morphism $ \eta : X \rightarrow Z $ onto a rational del Pezzo surface of degree $8$. It contracts $D$ onto a point $p = \lbrace p_{1},p_{2},p_{3} \rbrace$ of degree $3$ and it conjugates the $ \Gal(\kb/\k) $-action on $X_{\kb}$ to an action that exchanges the fibrations of $ Z_{\kb} $ (see Figure \ref{fig:Figure_Gal(k_barre/k)-action_on_Y_and_Q^L_in_Sym_3_2_case}). Thus, $ Z \simeq \mathcal{Q}^{L} $ for some quadratic extension $ L/\k $ \cite[Lemma 3.2(1)]{sz21}. Then $ \eta $ conjugates the $ \Gal(\kb/L) $-action on $ \mathcal{Q}^{L}_{\kb} $ to an action on $ \Pi_{\kb} $ with $ \rho(\Gal(\kb/L)) = \mathbb{Z}/3\mathbb{Z} $. Let $ F/\k $ be any splitting field of $ p \in \mathcal{Q}^{L} $. Let us show that $\Gal(F/\k) \simeq \Sym_{3}$. Let $ M/\k $ be a minimal normal extension containing $F$ such that $ \Gal(\kb/M) $ acts trivially on $ \Pi_{\kb} $. The $ \Gal(\kb/\k) $-action on $ \Pi_{\kb} $ induces a short exact sequence $ 1 \rightarrow \Gal(\kb/M) \rightarrow \Gal(\kb/\k) \rightarrow \Sym_{3} \rightarrow 1 $, which implies that $ \Gal(M/\k) \simeq \Gal(\kb/\k) / \Gal(\kb/M) \simeq \text{Sym}_{3} $. Then $F$ is an intermediate field between $M$ and $\k$, and since $ F/\k $ is normal, we get $ \Sym_{3} \subseteq \Gal(F/\k) \subseteq \Gal(M/\k) \simeq \Sym_{3} $. This implies $ \Gal(F/\k) \simeq \Sym_{3} $. 
Then $ FL/L $ is a splitting field of $p$ over $L$ as well, $ FL/L $ is Galois and $ \Gal(FL/L) \simeq \Gal(F/F \cap L) $ (see \cite[Theorem 5.5]{mor96}). Since $ \rho(\Gal(\kb/L)) = \langle(123)\rangle = \mathbb{Z}/3\mathbb{Z} $, the action of $ \Gal(FL/L) $ on $ p = \lbrace p_{1},p_{2},p_{3} \rbrace  $ induces an exact sequence $ 1 \rightarrow H \rightarrow \Gal(FL/L) \rightarrow \mathbb{Z}/3\mathbb{Z} \rightarrow 1 $, and $H=\lbrace \id \rbrace $ using the same argument as in Proposition \ref{Prop:Proposition_essai_regroupement}(\ref{it:item_(1)_Proposition_regroupement}) (see \cite[Corollary 2.10]{mor96}). Then we get $\Gal(FL/L) \simeq \Z/3\Z$.
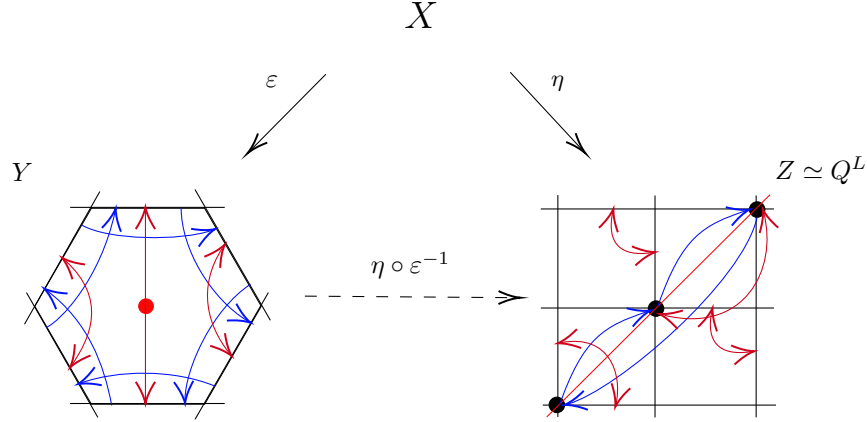
\begin{figure}[h]
\centering
\begin{tikzpicture}[x=0.6pt,y=0.6pt,yscale=-1,xscale=1, scale=0.8, every node/.style={scale=0.8}]

\draw   (171,170.1) -- (142.45,219.55) -- (85.35,219.55) -- (56.8,170.1) -- (85.35,120.65) -- (142.45,120.65) -- cycle ;
\draw    (51.6,179.2) -- (89.6,112.2) ;
\draw    (88.6,226.2) -- (52.6,163.2) ;
\draw    (152.6,219.2) -- (74.6,219.2) ;
\draw    (174.6,163.2) -- (137.6,227.2) ;
\draw    (137.6,111.2) -- (175.6,177.2) ;
\draw    (152.6,120.2) -- (74.6,120.2) ;
\draw  [color={rgb, 255:red, 248; green, 7; blue, 7 }  ,draw opacity=1 ][fill={rgb, 255:red, 248; green, 7; blue, 7 }  ,fill opacity=1 ] (109.2,170.1) .. controls (109.2,168.06) and (110.86,166.4) .. (112.9,166.4) .. controls (114.94,166.4) and (116.6,168.06) .. (116.6,170.1) .. controls (116.6,172.14) and (114.94,173.8) .. (112.9,173.8) .. controls (110.86,173.8) and (109.2,172.14) .. (109.2,170.1) -- cycle ;
\draw [color={rgb, 255:red, 3; green, 19; blue, 249 }  ,draw opacity=1 ]   (95.6,220.2) .. controls (91.68,199.62) and (83.92,178.08) .. (62.9,162.17) ;
\draw [shift={(61.6,161.2)}, rotate = 36.03] [color={rgb, 255:red, 3; green, 19; blue, 249 }  ,draw opacity=1 ][line width=0.75]    (10.93,-4.9) .. controls (6.95,-2.3) and (3.31,-0.67) .. (0,0) .. controls (3.31,0.67) and (6.95,2.3) .. (10.93,4.9)   ;
\draw [color={rgb, 255:red, 3; green, 19; blue, 249 }  ,draw opacity=1 ]   (62.6,180.8) .. controls (69.49,177.85) and (89.97,156.46) .. (97.28,122.37) ;
\draw [shift={(97.6,120.8)}, rotate = 101.31] [color={rgb, 255:red, 3; green, 19; blue, 249 }  ,draw opacity=1 ][line width=0.75]    (10.93,-4.9) .. controls (6.95,-2.3) and (3.31,-0.67) .. (0,0) .. controls (3.31,0.67) and (6.95,2.3) .. (10.93,4.9)   ;
\draw [color={rgb, 255:red, 3; green, 19; blue, 249 }  ,draw opacity=1 ]   (80.6,129.2) .. controls (94.04,136.88) and (130.52,137.19) .. (145.8,131.89) ;
\draw [shift={(147.6,131.2)}, rotate = 156.8] [color={rgb, 255:red, 3; green, 19; blue, 249 }  ,draw opacity=1 ][line width=0.75]    (10.93,-4.9) .. controls (6.95,-2.3) and (3.31,-0.67) .. (0,0) .. controls (3.31,0.67) and (6.95,2.3) .. (10.93,4.9)   ;
\draw [color={rgb, 255:red, 3; green, 19; blue, 249 }  ,draw opacity=1 ]   (130.6,120.4) .. controls (130.6,141.16) and (145.65,166.24) .. (163.9,177.4) ;
\draw [shift={(165.6,178.4)}, rotate = 209.16] [color={rgb, 255:red, 3; green, 19; blue, 249 }  ,draw opacity=1 ][line width=0.75]    (10.93,-4.9) .. controls (6.95,-2.3) and (3.31,-0.67) .. (0,0) .. controls (3.31,0.67) and (6.95,2.3) .. (10.93,4.9)   ;
\draw [color={rgb, 255:red, 3; green, 19; blue, 249 }  ,draw opacity=1 ]   (165.6,159) .. controls (146.4,171.48) and (132.73,201.47) .. (132.55,217.12) ;
\draw [shift={(132.6,219)}, rotate = 266.19] [color={rgb, 255:red, 3; green, 19; blue, 249 }  ,draw opacity=1 ][line width=0.75]    (10.93,-4.9) .. controls (6.95,-2.3) and (3.31,-0.67) .. (0,0) .. controls (3.31,0.67) and (6.95,2.3) .. (10.93,4.9)   ;
\draw [color={rgb, 255:red, 3; green, 19; blue, 249 }  ,draw opacity=1 ]   (147.6,210.4) .. controls (126.26,201.67) and (98.33,202.35) .. (80.25,208.79) ;
\draw [shift={(78.6,209.4)}, rotate = 338.75] [color={rgb, 255:red, 3; green, 19; blue, 249 }  ,draw opacity=1 ][line width=0.75]    (10.93,-4.9) .. controls (6.95,-2.3) and (3.31,-0.67) .. (0,0) .. controls (3.31,0.67) and (6.95,2.3) .. (10.93,4.9)   ;
\draw [color={rgb, 255:red, 208; green, 2; blue, 27 }  ,draw opacity=1 ]   (112.6,122.2) -- (112.6,168.2) -- (112.6,217.2) ;
\draw [shift={(112.6,219.2)}, rotate = 270] [color={rgb, 255:red, 208; green, 2; blue, 27 }  ,draw opacity=1 ][line width=0.75]    (10.93,-4.9) .. controls (6.95,-2.3) and (3.31,-0.67) .. (0,0) .. controls (3.31,0.67) and (6.95,2.3) .. (10.93,4.9)   ;
\draw [shift={(112.6,120.2)}, rotate = 90] [color={rgb, 255:red, 208; green, 2; blue, 27 }  ,draw opacity=1 ][line width=0.75]    (10.93,-4.9) .. controls (6.95,-2.3) and (3.31,-0.67) .. (0,0) .. controls (3.31,0.67) and (6.95,2.3) .. (10.93,4.9)   ;
\draw    (203.6,53) -- (165.41,91.58) ;
\draw [shift={(164,93)}, rotate = 314.71] [color={rgb, 255:red, 0; green, 0; blue, 0 }  ][line width=0.75]    (10.93,-3.29) .. controls (6.95,-1.4) and (3.31,-0.3) .. (0,0) .. controls (3.31,0.3) and (6.95,1.4) .. (10.93,3.29)   ;
\draw    (296.6,52) -- (333.24,91.53) ;
\draw [shift={(334.6,93)}, rotate = 227.17] [color={rgb, 255:red, 0; green, 0; blue, 0 }  ][line width=0.75]    (10.93,-3.29) .. controls (6.95,-1.4) and (3.31,-0.3) .. (0,0) .. controls (3.31,0.3) and (6.95,1.4) .. (10.93,3.29)   ;
\draw    (320.6,114) -- (320.6,226) ;
\draw    (369.6,114.2) -- (369.6,226.2) ;
\draw    (420.6,114.2) -- (420.6,226.2) ;
\draw    (312,121) -- (429.6,121) ;
\draw    (312,171) -- (429.6,171) ;
\draw    (313,220) -- (430.6,220) ;
\draw  [fill={rgb, 255:red, 0; green, 0; blue, 0 }  ,fill opacity=1 ] (316.2,219.9) .. controls (316.2,217.75) and (317.95,216) .. (320.1,216) .. controls (322.25,216) and (324,217.75) .. (324,219.9) .. controls (324,222.05) and (322.25,223.8) .. (320.1,223.8) .. controls (317.95,223.8) and (316.2,222.05) .. (316.2,219.9) -- cycle ;
\draw  [fill={rgb, 255:red, 0; green, 0; blue, 0 }  ,fill opacity=1 ] (366.2,171.3) .. controls (366.2,169.15) and (367.95,167.4) .. (370.1,167.4) .. controls (372.25,167.4) and (374,169.15) .. (374,171.3) .. controls (374,173.45) and (372.25,175.2) .. (370.1,175.2) .. controls (367.95,175.2) and (366.2,173.45) .. (366.2,171.3) -- cycle ;
\draw  [fill={rgb, 255:red, 0; green, 0; blue, 0 }  ,fill opacity=1 ] (417.2,121.3) .. controls (417.2,119.15) and (418.95,117.4) .. (421.1,117.4) .. controls (423.25,117.4) and (425,119.15) .. (425,121.3) .. controls (425,123.45) and (423.25,125.2) .. (421.1,125.2) .. controls (418.95,125.2) and (417.2,123.45) .. (417.2,121.3) -- cycle ;
\draw [color={rgb, 255:red, 248; green, 7; blue, 7 }  ,draw opacity=1 ]   (427.6,113.8) -- (315,226) ;
\draw [color={rgb, 255:red, 3; green, 19; blue, 249 }  ,draw opacity=1 ]   (321.1,220.9) .. controls (333.29,185.7) and (344.05,180.07) .. (365.53,172.86) ;
\draw [shift={(367.2,172.3)}, rotate = 161.64] [color={rgb, 255:red, 3; green, 19; blue, 249 }  ,draw opacity=1 ][line width=0.75]    (10.93,-4.9) .. controls (6.95,-2.3) and (3.31,-0.67) .. (0,0) .. controls (3.31,0.67) and (6.95,2.3) .. (10.93,4.9)   ;
\draw [color={rgb, 255:red, 3; green, 19; blue, 249 }  ,draw opacity=1 ]   (371.1,169.9) .. controls (383.29,134.7) and (394.05,129.07) .. (415.53,121.86) ;
\draw [shift={(417.2,121.3)}, rotate = 161.64] [color={rgb, 255:red, 3; green, 19; blue, 249 }  ,draw opacity=1 ][line width=0.75]    (10.93,-4.9) .. controls (6.95,-2.3) and (3.31,-0.67) .. (0,0) .. controls (3.31,0.67) and (6.95,2.3) .. (10.93,4.9)   ;
\draw [color={rgb, 255:red, 3; green, 19; blue, 249 }  ,draw opacity=1 ]   (421.1,125.2) .. controls (413.79,157.96) and (350.86,212.01) .. (325.85,219.42) ;
\draw [shift={(324,219.9)}, rotate = 347.81] [color={rgb, 255:red, 3; green, 19; blue, 249 }  ,draw opacity=1 ][line width=0.75]    (10.93,-4.9) .. controls (6.95,-2.3) and (3.31,-0.67) .. (0,0) .. controls (3.31,0.67) and (6.95,2.3) .. (10.93,4.9)   ;
\draw [color={rgb, 255:red, 208; green, 2; blue, 27 }  ,draw opacity=1 ]   (322.59,188.4) .. controls (346.88,184.12) and (353.02,204.44) .. (350.02,219.01) ;
\draw [shift={(349.6,220.8)}, rotate = 284.93] [color={rgb, 255:red, 208; green, 2; blue, 27 }  ,draw opacity=1 ][line width=0.75]    (10.93,-4.9) .. controls (6.95,-2.3) and (3.31,-0.67) .. (0,0) .. controls (3.31,0.67) and (6.95,2.3) .. (10.93,4.9)   ;
\draw [shift={(320.6,188.8)}, rotate = 347.47] [color={rgb, 255:red, 208; green, 2; blue, 27 }  ,draw opacity=1 ][line width=0.75]    (10.93,-4.9) .. controls (6.95,-2.3) and (3.31,-0.67) .. (0,0) .. controls (3.31,0.67) and (6.95,2.3) .. (10.93,4.9)   ;
\draw [color={rgb, 255:red, 208; green, 2; blue, 27 }  ,draw opacity=1 ]   (348.17,123) .. controls (345.28,138.08) and (351.83,145.35) .. (367.82,143.08) ;
\draw [shift={(369.6,142.8)}, rotate = 169.92] [color={rgb, 255:red, 208; green, 2; blue, 27 }  ,draw opacity=1 ][line width=0.75]    (10.93,-4.9) .. controls (6.95,-2.3) and (3.31,-0.67) .. (0,0) .. controls (3.31,0.67) and (6.95,2.3) .. (10.93,4.9)   ;
\draw [shift={(348.6,121)}, rotate = 103.24] [color={rgb, 255:red, 208; green, 2; blue, 27 }  ,draw opacity=1 ][line width=0.75]    (10.93,-4.9) .. controls (6.95,-2.3) and (3.31,-0.67) .. (0,0) .. controls (3.31,0.67) and (6.95,2.3) .. (10.93,4.9)   ;
\draw  [dash pattern={on 4.5pt off 4.5pt}]  (193,165) -- (299.6,165.79) ;
\draw [shift={(301.6,165.8)}, rotate = 180.42] [color={rgb, 255:red, 0; green, 0; blue, 0 }  ][line width=0.75]    (10.93,-3.29) .. controls (6.95,-1.4) and (3.31,-0.3) .. (0,0) .. controls (3.31,0.3) and (6.95,1.4) .. (10.93,3.29)   ;
\draw [color={rgb, 255:red, 208; green, 2; blue, 27 }  ,draw opacity=1 ]   (72.34,147.22) .. controls (93.34,166.35) and (84.99,185.93) .. (75.62,199.54) ;
\draw [shift={(74.6,201)}, rotate = 305.54] [color={rgb, 255:red, 208; green, 2; blue, 27 }  ,draw opacity=1 ][line width=0.75]    (10.93,-4.9) .. controls (6.95,-2.3) and (3.31,-0.67) .. (0,0) .. controls (3.31,0.67) and (6.95,2.3) .. (10.93,4.9)   ;
\draw [shift={(70.6,145.7)}, rotate = 40.23] [color={rgb, 255:red, 208; green, 2; blue, 27 }  ,draw opacity=1 ][line width=0.75]    (10.93,-4.9) .. controls (6.95,-2.3) and (3.31,-0.67) .. (0,0) .. controls (3.31,0.67) and (6.95,2.3) .. (10.93,4.9)   ;
\draw [color={rgb, 255:red, 208; green, 2; blue, 27 }  ,draw opacity=1 ]   (151.31,140.92) .. controls (138.46,160.83) and (144.03,180.21) .. (154.9,193.75) ;
\draw [shift={(156.1,195.2)}, rotate = 229.4] [color={rgb, 255:red, 208; green, 2; blue, 27 }  ,draw opacity=1 ][line width=0.75]    (10.93,-4.9) .. controls (6.95,-2.3) and (3.31,-0.67) .. (0,0) .. controls (3.31,0.67) and (6.95,2.3) .. (10.93,4.9)   ;
\draw [shift={(152.6,139)}, rotate = 125.15] [color={rgb, 255:red, 208; green, 2; blue, 27 }  ,draw opacity=1 ][line width=0.75]    (10.93,-4.9) .. controls (6.95,-2.3) and (3.31,-0.67) .. (0,0) .. controls (3.31,0.67) and (6.95,2.3) .. (10.93,4.9)   ;
\draw [color={rgb, 255:red, 208; green, 2; blue, 27 }  ,draw opacity=1 ]   (374.34,174.66) .. controls (414.49,184.4) and (431.73,151.76) .. (424.48,123.05) ;
\draw [shift={(424,121.3)}, rotate = 73.85] [color={rgb, 255:red, 208; green, 2; blue, 27 }  ,draw opacity=1 ][line width=0.75]    (10.93,-4.9) .. controls (6.95,-2.3) and (3.31,-0.67) .. (0,0) .. controls (3.31,0.67) and (6.95,2.3) .. (10.93,4.9)   ;
\draw [shift={(371.8,174)}, rotate = 15.66] [color={rgb, 255:red, 208; green, 2; blue, 27 }  ,draw opacity=1 ][line width=0.75]    (10.93,-4.9) .. controls (6.95,-2.3) and (3.31,-0.67) .. (0,0) .. controls (3.31,0.67) and (6.95,2.3) .. (10.93,4.9)   ;
\draw [color={rgb, 255:red, 208; green, 2; blue, 27 }  ,draw opacity=1 ]   (398.17,173) .. controls (395.28,188.08) and (401.83,195.35) .. (417.82,193.08) ;
\draw [shift={(419.6,192.8)}, rotate = 169.92] [color={rgb, 255:red, 208; green, 2; blue, 27 }  ,draw opacity=1 ][line width=0.75]    (10.93,-4.9) .. controls (6.95,-2.3) and (3.31,-0.67) .. (0,0) .. controls (3.31,0.67) and (6.95,2.3) .. (10.93,4.9)   ;
\draw [shift={(398.6,171)}, rotate = 103.24] [color={rgb, 255:red, 208; green, 2; blue, 27 }  ,draw opacity=1 ][line width=0.75]    (10.93,-4.9) .. controls (6.95,-2.3) and (3.31,-0.67) .. (0,0) .. controls (3.31,0.67) and (6.95,2.3) .. (10.93,4.9)   ;

\draw (44,96.4) node [anchor=north west][inner sep=0.75pt]    {$Y$};
\draw (242,15.4) node [anchor=north west][inner sep=0.75pt]  [font=\Large]  {$X$};
\draw (429,93.4) node [anchor=north west][inner sep=0.75pt]    {$Z\simeq Q^{L}$};
\draw (172,52.4) node [anchor=north west][inner sep=0.75pt]    {$\varepsilon $};
\draw (316,52.4) node [anchor=north west][inner sep=0.75pt]  [font=\small]  {$\eta $};
\draw (225,141.4) node [anchor=north west][inner sep=0.75pt]    {$\eta \circ \varepsilon ^{-1}$};
\end{tikzpicture}
\caption{The $ \Gal(\kb/\k) $-action on $ Y_{\kb} $ and $ Z_{\kb} \simeq \mathcal{Q}^{L}_{\kb} $.}
\label{fig:Figure_Gal(k_barre/k)-action_on_Y_and_Q^L_in_Sym_3_2_case}
\end{figure}

(\ref{it:item_(2):Proposition_9_Sym3_2}) Two del Pezzo surfaces $ \mathcal{Q}^{L} $ and $ \mathcal{Q}^{L'} $ are isomorphic if and only if the corresponding quadratic extensions $L$ and $L'$ are $\k$-isomorphic \cite[Lemma 3.2(3)]{sz21}. Moreover, $ \Aut_{\k}(\mathcal{Q}^{L}) $ acts transitively on the set of points of degree $3$ with $\k$-isomorphic splitting fields and whose geometric components are in general position in $\Ql^{L}_{\kb}$ by Lemma \ref{lem:Lemma_transitive_action_Aut_k(Q^L)_points_of_degree_3}. This yields the claim.

(\ref{it:item_(3):Proposition_9_Sym3_2}\&\ref{it:item_(4):Proposition_9_Sym3_2}) By Lemma \ref{lem:Lemma_faithful_action_Aut_k(X)_on_Pi_L_rho}, the action of $ \Aut_{\k}(X) $ on the set of $ (-1) $-curves of $X$ gives an isomorphism $ \Psi : \Aut_{\k}(X) \overset{\sim}{\longrightarrow} \Aut(\Pi_{\kb,\rho}) $. As before, we have $\Aut(\Pi_{\kb,\rho}) \simeq C_{\Sym_{5}}(\langle (123),(12)(45) \rangle) = \langle (45) \rangle$, and we now construct the corresponding geometric action.
From part (\ref{it:item_(1):Proposition_9_Sym3_2}) we know that $Y_{L}$ is isomorphic to the blow-up $ \pi : Y_{L} \rightarrow \p^{2}_{L} $ of a point $q=\lbrace q_{1},q_{2},q_{3} \rbrace$ of degree $3$, so that $X_{L}$ is isomorphic to the del Pezzo surface of degree $5$ from Proposition \ref{Prop:Proposition_3_Z/3Z}. From the proof of Proposition \ref{Prop:Proposition_7_Z/6Z}(\ref{it:item_(3)_Proposition_7_Z/6Z}) there exists $ \Phi_{q} \in \Aut_{\k}(Y) $ that induces a rotation of order two on the hexagon of $ Y $ and that lifts by $ \varepsilon $ to an automorphism $ \widehat{\Phi_{q}} $ of $X$ defined over $ \k $ that acts exactly as $(45)$ on $ \Pi_{\kb} $. In conclusion, $ \Aut_{\k}(X)=\langle \widehat{\Phi_{q}} \rangle \simeq \mathbb{Z}/2\mathbb{Z} $. To finish, note that from (\ref{it:item_(1):Proposition_9_Sym3_2}) we have a contraction $\eta : X \rightarrow \Ql^{L}$ onto a point of degree $3$, so thanks to Lemma \ref{lem:Lemma_transitive_action_Aut_k(Q^L)_points_of_degree_3}, we get $\rk \, \NS(X)^{\Aut_{\k}(X)} = 2$. Now, note that the curve $E_{4}$ is stabilized by $\Aut_{\k}(X)$. It can then be contracted equivariantly onto the rational del Pezzo surface $Y$ of degree $6$ in Figure \ref{fig:Figure_Gal(k_barre/k)-action_on_Y_and_Q^L_in_Sym_3_2_case} with $\rk \, \NS(Y) = 1$, which gives the last assertion.
\end{proof}

\begin{example}
We construct a del Pezzo surface of degree $5$ as in Proposition \ref{Prop:Proposition9_Sym3_2}: let $ \k := \mathbb{Q} $, let $ \zeta := \sqrt[3]{2} $ and $ \omega := e^{\frac{2 \pi i}{3}} $ as in Example \ref{ex:example_Sym3_case_1}. Let $ L := \k(\omega) $. It is a Galois extension of $ \k $ of degree $2$ and $ \Gal(L/\k) = \langle \tau : \omega \mapsto \omega^{2} \rangle \simeq \mathbb{Z}/2\mathbb{Z} $. Then $ F := \k(\zeta,\omega) $ is a Galois extension of $\k$ of degree $6$ and $ \Gal(F/\k) \simeq \Sym_{3} $ is the group of $ \k $-isomorphisms of $F$ composed with$$ (\zeta,\omega) \overset{\sigma_{1}}{\mapsto} (\zeta,\omega), (\zeta,\omega) \overset{\sigma_{2}}{\mapsto} (\omega \zeta,\omega), (\zeta,\omega) \overset{\sigma_{3}}{\mapsto} (\zeta,\omega^{2}), (\zeta,\omega) \overset{\sigma_{4}}{\mapsto} (\omega \zeta,\omega^{2}), (\zeta,\omega) \overset{\sigma_{5}}{\mapsto} (\omega^{2} \zeta,\omega), (\zeta,\omega) \overset{\sigma_{6}}{\mapsto} (\omega^{2} \zeta,\omega^{2}). $$The point$$ p = \displaystyle\lbrace \big([\zeta-\omega\zeta^{2}:1],[\zeta-\omega^{2}\zeta^{2}:1]\big) ;\big([\omega\zeta-\zeta^{2}:1],[\omega\zeta-\omega\zeta^{2}:1]\big) ; \big([\omega^{2}\zeta-\omega^{2}\zeta^{2}:1],[\omega^{2}\zeta-\zeta^{2}:1]\big) \rbrace $$is of degree $3$, its geometric components are contained in pairwise distinct rulings of $ \mathcal{Q}^{L}_{L} $ and any non-trivial element of $ \Gal(F/\k) $ permutes them non-trivially. Note that $ F := \k(\zeta,\omega) $ is the splitting field of $ X^3-2 $ over $ L := \k(\omega) $. Then $ FL/L = F/L $ is a Galois extension of degree $3$, $ \Gal(FL/L) = \Gal(F/L) \simeq \mathbb{Z}/3\mathbb{Z} $ and it permutes cyclically the geometric components of $p$. 
\end{example}

\subsection{Del Pezzo surfaces in Figures (\ref{fig:figure(n)_option_Gal(kbarre/k)-action_on_Pikbarre}), (\ref{fig:figure(o)_option_Gal(kbarre/k)-action_on_Pikbarre}), (\ref{fig:figure(p)_option_Gal(kbarre/k)-action_on_Pikbarre}), (\ref{fig:figure(q)_option_Gal(kbarre/k)-action_on_Pikbarre}), (\ref{fig:figure(r)_option_Gal(kbarre/k)-action_on_Pikbarre_S5})}
\label{subsec:subsection_4}

In this section, we let $X$ be one of the following del Pezzo surfaces of degree $5$:
\begin{enumerate}
\item $\rho(\Gal(\kb/\k))=\langle(12345)\rangle \simeq \Z/5\Z$ in $\Sym_{5}$ as indicated in Figure \ref{fig:figure(n)_option_Gal(kbarre/k)-action_on_Pikbarre},\label{it:item_(a)_cases_minimal_dP5}
\item $\rho(\Gal(\kb/\k))=\langle(12345),(25)(34)\rangle \simeq \text{D}_{5}$ in $\Sym_{5}$ as indicated in Figure \ref{fig:figure(o)_option_Gal(kbarre/k)-action_on_Pikbarre},\label{it:item_(b)_cases_minimal_dP5}
\item $\rho(\Gal(\kb/\k))=\langle(12345),(2354)\rangle \simeq \text{GA}(1,5)$ in $\Sym_{5}$ as indicated in Figure \ref{fig:figure(p)_option_Gal(kbarre/k)-action_on_Pikbarre},\label{it:item_(c)_cases_minimal_dP5}
\item $\rho(\Gal(\kb/\k))=\langle(12345),(123)\rangle \simeq \mathcal{A}_{5}$ in $\Sym_{5}$ as indicated in Figure \ref{fig:figure(q)_option_Gal(kbarre/k)-action_on_Pikbarre},\label{it:item_(d)_cases_minimal_dP5}
\item $\rho(\Gal(\kb/\k))=\langle(12345),(12)\rangle \simeq \Sym_{5}$ as indicated in Figure \ref{fig:figure(r)_option_Gal(kbarre/k)-action_on_Pikbarre_S5}.\label{it:item_(e)_cases_minimal_dP5}
\end{enumerate}

In each case, either the Petersen diagram of $X_{\kb}$ contains two orbits of $(-1)$-curves of size five in which at least two of them intersect (see Figures \ref{fig:figure(n)_option_Gal(kbarre/k)-action_on_Pikbarre}, \ref{fig:figure(o)_option_Gal(kbarre/k)-action_on_Pikbarre}), or it is made up of a single orbit of $(-1)$-curves (see Figures \ref{fig:figure(p)_option_Gal(kbarre/k)-action_on_Pikbarre}, \ref{fig:figure(q)_option_Gal(kbarre/k)-action_on_Pikbarre}, \ref{fig:figure(r)_option_Gal(kbarre/k)-action_on_Pikbarre_S5}). It means that no curve can be birationally contracted over $\k$. On the other hand, there are five conic bundle structures on $X_{\kb}$ that have exactly three singular fibres, meaning that there is no conic bundle structure on $X$.
Thus $X \longrightarrow \ast$ is a Mori fibre space. Since $X$ is of degree $5$, recall that $X$ contains at least one $\k$-rational point $r \in X(\k)$ and is in particular rational (see Proposition \ref{Prop:Proposition_rationality_of_dP_5_over_a_field}). The configuration of $(-1)$-curves on $\Pi_{\kb}$ in each case implies that $r$ does not lie on any of them. Let $\eta_{1} : Y \longrightarrow X $ be the blow-up of $r$ and $E_{r}$ its associated exceptional divisor. Then $Y$ is a rational del Pezzo surface of degree $4$ and $Y_{\kb}$ contains an orbit of five disjoint $(-1)$-curves denoted $\lbrace D_{1r},D_{2r},D_{3r},D_{4r},C \rbrace$. More precisely, denoting $\pi : X_{\kb} \rightarrow \p^{2}_{\kb}$ the blow-up of the points $p_{1},p_{2},p_{3},p_{4}$ described in §\ref{subsec:subsection_0}, $D_{ir}$ is the strict transform of the line in $\p^{2}_{\kb}$ through $\pi(r)$ and $p_{i}$, and $C$ is the strict transform of the conic through $p_{1},\dots,p_{4},\pi(r)$. The contraction of the curve $\mathcal{C} := D_{1r} \cup D_{2r} \cup D_{3r} \cup D_{4r}\cup C$ yields a birational morphism $\eta_{2} : Y \longrightarrow \p^{2} $ onto a point $ q=\lbrace q_{1},q_{2},q_{3},q_{4},q_{5}\rbrace $ of degree $5$ whose geometric components are in general position in $\p^{2}_{\kb}$, and the image of $E_{r}$ by $\eta_{2}$ is a conic through $q$.
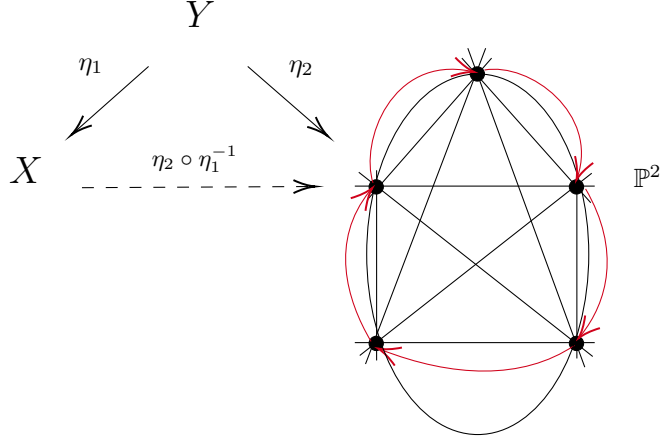
\begin{figure}[h]
\centering
\begin{tikzpicture}[x=0.6pt,y=0.6pt,yscale=-1,xscale=1, scale=0.8, every node/.style={scale=0.8}]

\draw    (260.6,143.2) -- (260.6,239.2) ;
\draw    (361.6,144.2) -- (361.6,238.2) ;
\draw    (249.6,151.2) -- (370.6,151.2) ;
\draw    (249.6,230.2) -- (370.6,230.2) ;
\draw  [fill={rgb, 255:red, 0; green, 0; blue, 0 }  ,fill opacity=1 ] (256.8,151.6) .. controls (256.8,149.61) and (258.41,148) .. (260.4,148) .. controls (262.39,148) and (264,149.61) .. (264,151.6) .. controls (264,153.59) and (262.39,155.2) .. (260.4,155.2) .. controls (258.41,155.2) and (256.8,153.59) .. (256.8,151.6) -- cycle ;
\draw  [fill={rgb, 255:red, 0; green, 0; blue, 0 }  ,fill opacity=1 ] (357.8,151.6) .. controls (357.8,149.61) and (359.41,148) .. (361.4,148) .. controls (363.39,148) and (365,149.61) .. (365,151.6) .. controls (365,153.59) and (363.39,155.2) .. (361.4,155.2) .. controls (359.41,155.2) and (357.8,153.59) .. (357.8,151.6) -- cycle ;
\draw  [fill={rgb, 255:red, 0; green, 0; blue, 0 }  ,fill opacity=1 ] (256.8,230.6) .. controls (256.8,228.61) and (258.41,227) .. (260.4,227) .. controls (262.39,227) and (264,228.61) .. (264,230.6) .. controls (264,232.59) and (262.39,234.2) .. (260.4,234.2) .. controls (258.41,234.2) and (256.8,232.59) .. (256.8,230.6) -- cycle ;
\draw  [fill={rgb, 255:red, 0; green, 0; blue, 0 }  ,fill opacity=1 ] (357.8,230.6) .. controls (357.8,228.61) and (359.41,227) .. (361.4,227) .. controls (363.39,227) and (365,228.61) .. (365,230.6) .. controls (365,232.59) and (363.39,234.2) .. (361.4,234.2) .. controls (359.41,234.2) and (357.8,232.59) .. (357.8,230.6) -- cycle ;
\draw    (251.6,144.2) -- (369.6,237.2) ;
\draw    (369.1,144.7) -- (252.1,236.7) ;
\draw    (195.6,90.8) -- (234.49,124.69) ;
\draw [shift={(236,126)}, rotate = 221.07] [color={rgb, 255:red, 0; green, 0; blue, 0 }  ][line width=0.75]    (10.93,-3.29) .. controls (6.95,-1.4) and (3.31,-0.3) .. (0,0) .. controls (3.31,0.3) and (6.95,1.4) .. (10.93,3.29)   ;
\draw [color={rgb, 255:red, 208; green, 2; blue, 27 }  ,draw opacity=1 ]   (366,152.6) .. controls (384.04,183.64) and (377.06,209.79) .. (363.66,226.48) ;
\draw [shift={(362.4,228)}, rotate = 310.54] [color={rgb, 255:red, 208; green, 2; blue, 27 }  ,draw opacity=1 ][line width=0.75]    (10.93,-4.9) .. controls (6.95,-2.3) and (3.31,-0.67) .. (0,0) .. controls (3.31,0.67) and (6.95,2.3) .. (10.93,4.9)   ;
\draw [color={rgb, 255:red, 208; green, 2; blue, 27 }  ,draw opacity=1 ]   (360.4,231.6) .. controls (334,249.33) and (298.68,247.66) .. (263.03,233.84) ;
\draw [shift={(261.4,233.2)}, rotate = 21.69] [color={rgb, 255:red, 208; green, 2; blue, 27 }  ,draw opacity=1 ][line width=0.75]    (10.93,-4.9) .. controls (6.95,-2.3) and (3.31,-0.67) .. (0,0) .. controls (3.31,0.67) and (6.95,2.3) .. (10.93,4.9)   ;
\draw  [fill={rgb, 255:red, 0; green, 0; blue, 0 }  ,fill opacity=1 ] (307.8,94.6) .. controls (307.8,92.61) and (309.41,91) .. (311.4,91) .. controls (313.39,91) and (315,92.61) .. (315,94.6) .. controls (315,96.59) and (313.39,98.2) .. (311.4,98.2) .. controls (309.41,98.2) and (307.8,96.59) .. (307.8,94.6) -- cycle ;
\draw    (256,158) -- (318.6,86.6) ;
\draw    (303.6,86.6) -- (368.6,159.6) ;
\draw    (315.6,81.6) -- (255,241) ;
\draw    (306.6,81.6) -- (365.6,242.6) ;
\draw  [color={rgb, 255:red, 0; green, 0; blue, 0 }  ,draw opacity=1 ] (255.6,185.6) .. controls (255.6,135.34) and (280.67,94.6) .. (311.6,94.6) .. controls (342.53,94.6) and (367.6,135.34) .. (367.6,185.6) .. controls (367.6,235.86) and (342.53,276.6) .. (311.6,276.6) .. controls (280.67,276.6) and (255.6,235.86) .. (255.6,185.6) -- cycle ;
\draw [color={rgb, 255:red, 208; green, 2; blue, 27 }  ,draw opacity=1 ]   (258.8,149.6) .. controls (252.1,128.72) and (268.11,85.13) .. (307.96,93.2) ;
\draw [shift={(309.8,93.6)}, rotate = 193.38] [color={rgb, 255:red, 208; green, 2; blue, 27 }  ,draw opacity=1 ][line width=0.75]    (10.93,-4.9) .. controls (6.95,-2.3) and (3.31,-0.67) .. (0,0) .. controls (3.31,0.67) and (6.95,2.3) .. (10.93,4.9)   ;
\draw    (145.6,90.8) -- (108.1,124.07) ;
\draw [shift={(106.6,125.4)}, rotate = 318.42] [color={rgb, 255:red, 0; green, 0; blue, 0 }  ][line width=0.75]    (10.93,-3.29) .. controls (6.95,-1.4) and (3.31,-0.3) .. (0,0) .. controls (3.31,0.3) and (6.95,1.4) .. (10.93,3.29)   ;
\draw  [dash pattern={on 4.5pt off 4.5pt}]  (112,152) -- (227.6,151.41) ;
\draw [shift={(229.6,151.4)}, rotate = 179.71] [color={rgb, 255:red, 0; green, 0; blue, 0 }  ][line width=0.75]    (10.93,-3.29) .. controls (6.95,-1.4) and (3.31,-0.3) .. (0,0) .. controls (3.31,0.3) and (6.95,1.4) .. (10.93,3.29)   ;
\draw [color={rgb, 255:red, 208; green, 2; blue, 27 }  ,draw opacity=1 ]   (315,92.6) .. controls (338.25,88.66) and (372.55,114.8) .. (361.92,146.55) ;
\draw [shift={(361.4,148)}, rotate = 290.63] [color={rgb, 255:red, 208; green, 2; blue, 27 }  ,draw opacity=1 ][line width=0.75]    (10.93,-4.9) .. controls (6.95,-2.3) and (3.31,-0.67) .. (0,0) .. controls (3.31,0.67) and (6.95,2.3) .. (10.93,4.9)   ;
\draw [color={rgb, 255:red, 208; green, 2; blue, 27 }  ,draw opacity=1 ]   (257.8,229.6) .. controls (245.26,208.43) and (236.94,178.81) .. (257.13,153.73) ;
\draw [shift={(258.4,152.2)}, rotate = 130.64] [color={rgb, 255:red, 208; green, 2; blue, 27 }  ,draw opacity=1 ][line width=0.75]    (10.93,-4.9) .. controls (6.95,-2.3) and (3.31,-0.67) .. (0,0) .. controls (3.31,0.67) and (6.95,2.3) .. (10.93,4.9)   ;

\draw (389,139.4) node [anchor=north west][inner sep=0.75pt]  [font=\normalsize]  {$\mathbb{P}^{2}$};
\draw (163,56.4) node [anchor=north west][inner sep=0.75pt]  [font=\Large]  {$Y$};
\draw (215,86.4) node [anchor=north west][inner sep=0.75pt]    {$\eta _{2}$};
\draw (74,135.4) node [anchor=north west][inner sep=0.75pt]  [font=\Large]  {$X$};
\draw (109,85.4) node [anchor=north west][inner sep=0.75pt]    {$\eta _{1}$};
\draw (146,129.4) node [anchor=north west][inner sep=0.75pt]  [font=\small]  {$\eta _{2} \circ \eta _{1}^{-1}$};
\end{tikzpicture}
\caption[]{Blow-up model for a del Pezzo surface as in case \ref{it:item_(a)_cases_minimal_dP5}, Figure (\ref{fig:figure(n)_option_Gal(kbarre/k)-action_on_Pikbarre}).}
\label{Fig:Figure_Blow-up_model_Z/5Z_case}
\end{figure}
Figure \ref{Fig:Figure_Blow-up_model_Z/5Z_case} shows the induced action of $\rho(\Gal(\kb/\k))$ on the image by $\eta_{2}\circ\eta_{1}^{-1}$ of $\Pi_{\kb}$ in case \ref{it:item_(a)_cases_minimal_dP5}. Let $F/\k$ be any splitting field of $q$. The action of $\Gal(\kb/\k)$ factors through the subgroups described respectively in \ref{it:item_(a)_cases_minimal_dP5}, \ref{it:item_(b)_cases_minimal_dP5}, \ref{it:item_(c)_cases_minimal_dP5}, \ref{it:item_(d)_cases_minimal_dP5}, \ref{it:item_(e)_cases_minimal_dP5}, and one can show from the induced action of $\Gal(F/\k)$ on $\lbrace q_{1},q_{2},q_{3},q_{4},q_{5} \rbrace$ that the Galois group $\Gal(F/\k)$ is the one given below in \ref{it:item_(i)_Gal(F/k)_cases_minimal_dP5}, \ref{it:item_(ii)_Gal(F/k)_cases_minimal_dP5}, \ref{it:item_(iii)_Gal(F/k)_cases_minimal_dP5}, \ref{it:item_(iv)_Gal(F/k)_cases_minimal_dP5}, \ref{it:item_(v)_Gal(F/k)_cases_minimal_dP5}, respectively.
\begin{enumerate}
\item\label{it:item_(i)_Gal(F/k)_cases_minimal_dP5} $\Gal(F/\k) \simeq \Z/5Z$,
\item $\Gal(F/\k) \simeq \D_{5}$,\label{it:item_(ii)_Gal(F/k)_cases_minimal_dP5}
\item $\Gal(F/\k) \simeq \text{GA}(1,5)$,\label{it:item_(iii)_Gal(F/k)_cases_minimal_dP5}
\item $\Gal(F/\k) \simeq \mathcal{A}_{5}$,\label{it:item_(iv)_Gal(F/k)_cases_minimal_dP5}
\item $\Gal(F/\k) \simeq \Sym_{5}$.\label{it:item_(v)_Gal(F/k)_cases_minimal_dP5}
\end{enumerate} 
Conversely, blowing up $\p^{2}$ in a point $q=\lbrace q_{1},q_{2},q_{3},q_{4},q_{5} \rbrace$ of degree $5$ with splitting field $F$ such that $\Gal(F/\k)$ is isomorphic to one of the groups described in \ref{it:item_(i)_Gal(F/k)_cases_minimal_dP5}, \ref{it:item_(ii)_Gal(F/k)_cases_minimal_dP5}, \ref{it:item_(iii)_Gal(F/k)_cases_minimal_dP5}, \ref{it:item_(iv)_Gal(F/k)_cases_minimal_dP5}, \ref{it:item_(v)_Gal(F/k)_cases_minimal_dP5}, and then blowing down the strict transform of the conic passing through $q$ gives a del Pezzo surface of degree $5$ as in \ref{it:item_(a)_cases_minimal_dP5}, \ref{it:item_(b)_cases_minimal_dP5}, \ref{it:item_(c)_cases_minimal_dP5}, \ref{it:item_(d)_cases_minimal_dP5}, \ref{it:item_(e)_cases_minimal_dP5}, respectively.\\

For each given $\Gal(\kb/\k)$-action on the Petersen graph $\Pi_{\kb}$, we will describe the automorphism group of $X$ in Proposition \ref{Prop:Proposition_final_cases} below. To prove point (\ref{it:item_(1)_Proposition_final_cases}) in \ref{Prop:Proposition_final_cases}, we will need the following lemma.

\begin{lemma}
Let $X$ be a del Pezzo surface of degree $5$. Suppose that $\Aut_{\k}(X)=\langle \alpha \rangle \simeq \Z/5\Z $. Then over $\kb$, $\alpha$ is the lift of a birational quadratic transformation of $\p^{2}_{\kb}$ of order $5$, and up to conjugation, it is the lift of the order five Cremona transformation defined by $ [x:y:z] \dashmapsto [x(z-y):z(x-y):xz] $.
\label{lem:Lemma_Autk(X)_Z/5Z_case}
\end{lemma}

\begin{proof}
Note that since $\alpha$ in an automorphism of $X_{\kb}$ of order $5$, it has two orbits of $(-1)$-curves of size five. Over $\kb$, one can contract the four disjoint $(-1)$-curves $E_{1},E_{2},E_{3},E_{4}$ onto four points $p_{1},p_{2},p_{3},p_{4}$ in general position in $\p^{2}_{\kb}$ and get a birational morphism $ \pi : X_{\kb} \longrightarrow \p^{2}_{\kb} $, so that $\phi := \pi \circ \alpha \circ \pi^{-1} : \p^{2}_{\kb} \dashrightarrow \p^{2}_{\kb} $ is a birational transformation (see Figure \ref{Fig:Figure_Lemma_Autk(X)_Z/5Z_case}).
\begin{figure}[h]
\centering \begin{tikzcd}[sep=huge]
X_{\kb} \arrow[r, "\alpha","\sim"'] \arrow[d, "\pi"']
& X_{\kb} \arrow[d, "\pi" ] \\
\mathbb{P}^{2}_{\overline{\mathbf{k}}} \arrow[r, dashrightarrow, "\phi" ]
&  \mathbb{P}^{2}_{\overline{\mathbf{k}}}
\end{tikzcd}
\caption{Construction of the birational map $\phi$.}
\label{Fig:Figure_Lemma_Autk(X)_Z/5Z_case}
\end{figure} Let us determine its degree $d_{\phi}$. We let $L \subseteq \p^{2}_{\kb} $ be a general line, we denote $ \phi^{-1}(L) =: C $ its preimage by $\phi$, and we compute the multiplicities $ m_{p_{i}}(\phi) := m_{p_{i}}(\phi^{-1}(L)) = m_{p_{i}}(C) $ for $ i=1,\dots,4 $. If $\alpha(E_{i}) = E_{j} $ for some $j$, then $ m_{p_{i}}(C) = \tilde{C} \cdot E_{i} = \alpha(\tilde{C}) \cdot \alpha(E_{i}) = \tilde{L} \cdot E_{j} = 0 $, and if $ \alpha(E_{i}) \neq E_{j} $ for all $j$, then $\alpha(E_{i})=D_{jk} $ and $ \tilde{C} \cdot E_{i} = \alpha(\tilde{C}) \cdot \alpha(E_{i}) = \tilde{L} \cdot D_{jk} = 1 $ because $L$ is general. So $ m_{p_{i}}(\phi^{-1}(L)) \in \lbrace 0,1 \rbrace $. Applying now Noether's relations (see for instance \cite[Proposition 2.34]{lam20}) to $d_{\phi}$ and to the $m_{p_{i}}(\phi)$'s, we get $d_{\phi} \in \lbrace 1,2 \rbrace$. If $d_{\phi}=1$, then such an automorphism of $\p^{2}_{\kb}$ would permute the four points $p_{1},\dots,p_{4}$ via an element of $\Sym_{4}$, which does not contain any element of order $5$. It follows that $\phi$ is a quadratic birational transformation of $\p^{2}_{\kb}$ of order $5$ with $3$ base-points among $\lbrace p_{1},p_{2},p_{3},p_{4} \rbrace$ (see \cite[Corollary 2.36]{lam20}), and up to conjugation, there is only one such birational quadratic transformation which is defined by $ \phi : [x:y:z] \dashmapsto [x(z-y):z(x-y):xz] $ (\hspace{1sp}\cite{bebl04}, see also \cite[Proposition 4.6.1]{dfe04} and \cite[Example 4.19]{yas16}). It has base-points $p_{1}=[0:0:1], p_{2}=[1:0:0], p_{3}=[0:1:0]$ and it sends $p_{4}=[1:1:1]$ onto $p_{1}$.
\end{proof}

\begin{proposition}
Let $X$ be a del Pezzo surface of degree $5$ with $\rk \, \NS(X) = 1$. Then the following holds.
\begin{enumerate}
\item If $ \rho(\Gal(\kb/\k))=\langle(12345)\rangle \simeq \Z/5\Z $, then $\Aut_{\k}(X)=\langle \widehat{\phi} \rangle \simeq \Z/5Z$, where $\widehat{\phi}$ is the lift of a quadratic birational map of $\p^{2}_{\kb}$ of order $5$.
\label{it:item_(1)_Proposition_final_cases}
\item If $\rho(\Gal(\kb/\k))=\langle(12345),(25)(34)\rangle \simeq \D_{5}$, then $\Aut_{\k}(X)=\lbrace \id \rbrace$.\label{it:item_(2)_Proposition_final_cases}
\item If $\rho(\Gal(\kb/\k))=\langle(12345),(2354)\rangle \simeq \GA(1,5)$, then $\Aut_{\k}(X) = \lbrace \id \rbrace$.\label{it:item_(3)_Proposition_final_cases}
\item If $\rho(\Gal(\kb/\k))=\langle(12345),(123)\rangle \simeq \mathcal{A}_{5}$, then $\Aut_{\k}(X) = \lbrace \id \rbrace$.\label{it:item_(4)_Proposition_final_cases}
\item If $\rho(\Gal(\kb/\k))=\langle(12345),(12)\rangle \simeq \Sym_{5}$, then $\Aut_{\k}(X)=\lbrace \id \rbrace$.\label{it:item_(5)_Proposition_final_cases}
\end{enumerate}
\label{Prop:Proposition_final_cases}
\end{proposition}

\begin{proof}
(\ref{it:item_(1)_Proposition_final_cases}) In this case the $\Gal(\kb/\k)$-action on the $(-1)$-curves of $X_{\kb}$ is given by $ E_{4} \mapsto E_{1} \mapsto D_{34} \mapsto D_{14} \mapsto D_{12} \mapsto E_{4} $ and $ D_{13} \mapsto E_{3} \mapsto D_{23} \mapsto E_{2} \mapsto D_{24} \mapsto D_{13} $, which gives the two orbits of size five on $\Pi_{\kb}$ represented in Figure \ref{fig:figure(n)_option_Gal(kbarre/k)-action_on_Pikbarre}. The action of $\Aut_{\k}(X)$ on the set of $(-1)$ curves of $X$ yields an isomorphism $ \Psi : \Aut_{\k}(X) \overset{\sim}{\longrightarrow} \Aut(\Pi_{\kb,\rho}) $ by Lemma \ref{lem:Lemma_faithful_action_Aut_k(X)_on_Pi_L_rho}, where $\Aut(\Pi_{\kb,\rho}) \simeq \lbrace \sigma \in \Sym_{5} \, \vert \, (12345) \circ \sigma = \sigma \circ (12345) \rbrace = \langle (12345) \rangle$.
Now consider over $\kb$ the order five birational quadratic map $ \phi : \p^{2}_{\kb} \dashrightarrow \p^{2}_{\kb}, \, [x:y:z] \dashmapsto [x(z-y):z(x-y):xz] $. Up to change of coordinates in $\p^{2}_{\kb}$ we can choose the four points $p_{1}=[0:0:1], p_{2}=[1:0:0], p_{3}=[0:1:0]$ and $p_{4}=[1:1:1]$. Then $\phi$ has base-points $p_{1}, p_{2}, p_{3}$, it sends $p_{4}$ onto $p_{1}$, it sends the lines $L_{34}$ onto $L_{14}$, $L_{14}$ onto $L_{12}$, $L_{24}$ onto $L_{13}$, and it contracts the lines $L_{12}$ onto $p_{4}$, $L_{23}$ onto $p_{2}$, $L_{13}$ onto $p_{3}$. Its birational inverse $ \phi^{-1} : [x:y:z] \dashmapsto [z(z-x):(z-y)(z-x):z(z-y)] $ has base-points $p_{2}, p_{3}, p_{4}$, it sends $p_{1}$ onto $p_{4}$, and it contracts the lines $L_{34}$ onto $p_{1}$, $L_{23}$ onto $p_{3}$, $L_{24}$ onto $p_{2}$. Thus, by blowing up the points $p_{1},p_{2},p_{3},p_{4}$, the map $\phi$ lifts to an automorphism $\widehat{\phi}$ of $X_{\kb}$ that acts exactly as $(12345)$ on $\Pi_{\kb}$. If $g \in \Gal(\kb/\k)$, then $(g \widehat{\phi} g^{-1}) \widehat{\phi}^{-1} \in \Aut(X_{\kb})$ preserves the curves $E_{1},E_{2},E_{3},E_{4}$. It therefore descends to an element of $\Aut(\p^{2}_{\kb})$ fixing four points in general position and is hence equal to the identity. It follows that $\widehat{\phi}$ is defined over $\k$ and so $\Aut_{\k}(X) = \langle \widehat{\phi} \rangle \simeq \Z/5\Z$.

(\ref{it:item_(2)_Proposition_final_cases}) In this case the action of the Galois group $\Gal(\kb/\k)$ on the set of $(-1)$ curves of $X_{\kb}$ is given by $ E_{4} \mapsto E_{1} \mapsto D_{34} \mapsto D_{14} \mapsto D_{12} \mapsto E_{4} $, $ D_{13} \mapsto E_{3} \mapsto D_{23} \mapsto E_{2} \mapsto D_{24} \mapsto D_{13} $ and $ E_{1} \leftrightarrow D_{34}, D_{23} \leftrightarrow D_{24}, E_{3} \leftrightarrow D_{13}, E_{4} \leftrightarrow D_{14}, D_{12}^{\circlearrowleft}, E_{2}^{\circlearrowleft} $, which still gives two orbits of $(-1)$-curves of size five as indicated in Figure \ref{fig:figure(o)_option_Gal(kbarre/k)-action_on_Pikbarre}. The action of $\Aut_{\k}(X)$ on the set of $(-1)$-curves of $X$ yields an isomorphism $\Psi : \Aut_{\k}(X) \overset{\sim}{\longrightarrow} \Aut(\Pi_{\kb,\rho})$ by Lemma \ref{lem:Lemma_faithful_action_Aut_k(X)_on_Pi_L_rho}, where $\Aut(\Pi_{\kb,\rho})=\lbrace \id \rbrace $ since there is no non-trivial element of $\Sym_{5}$ that commutes with both $(12345)$ and $(25)(34)$.

The points (\ref{it:item_(3)_Proposition_final_cases}), (\ref{it:item_(4)_Proposition_final_cases}) and (\ref{it:item_(5)_Proposition_final_cases}) are proven analogously to point  (\ref{it:item_(2)_Proposition_final_cases}).
\end{proof}

\begin{proof}[Proof of Theorem \ref{thm:Theorem_Main}]
Case \ref{cas:Case1_Main_Theorem} comes from the description given in §\ref{subsec:subsection_4}, as well as Lemma \ref{lem:Lemma_Autk(X)_Z/5Z_case} and Proposition \ref{Prop:Proposition_final_cases}. Case \ref{cas:Case2_Main_Theorem} follows from Proposition \ref{Prop:Proposition_0_Id}. Cases \ref{s-cas:Case(a)_Case3_Main_Theorem}, \ref{s-cas:Case(b)_Case3_Main_Theorem}, \ref{s-cas:Case(c)_Case3_Main_Theorem}, \ref{s-cas:Case(d)_Case3_Main_Theorem}, \ref{s-cas:Case(j)_Case3_Main_Theorem}, \ref{s-cas:Case(k)_Case3_Main_Theorem}, \ref{s-cas:Case(l)_Case3_Main_Theorem} are a consequence of Propositions \ref{Prop:Proposition1_Z/2Z_transposition}, \ref{Prop:Proposition2_Z/2Z_double_transposition}, \ref{Prop:Proposition_3_Z/3Z}, \ref{Prop:Proposition_8_Sym3_case_1}, \ref{Prop:Proposition_7_Z/6Z}, \ref{Prop:Proposition_?_S3xZ2}, \ref{Prop:Proposition9_Sym3_2}, respectively, and the cases \{\ref{s-cas:Case(e)_Case3_Main_Theorem}, \ref{s-cas:Case(f)_Case3_Main_Theorem}, \ref{s-cas:Case(g)_Case3_Main_Theorem}, \ref{s-cas:Case(h)_Case3_Main_Theorem}, \ref{s-cas:Case(i)_Case3_Main_Theorem}\} come from Proposition \ref{Prop:Proposition_essai_regroupement}. Finally, Case \ref{cas:Case4_Main_Theorem} comes from Proposition \ref{Prop:Proposition5_Z2xZ2_1}.  
\end{proof}


\Addresses

\end{document}